\documentclass{amsart}
\usepackage{pinlabel}
\usepackage{amsmath, amsthm, amssymb, amscd, mathrsfs, eucal}
\usepackage{enumerate}
\usepackage[T1]{fontenc}
\usepackage{hyperref}
\usepackage{float}
\usepackage{mathtools}
\usepackage{color}
\usepackage[margin=1in,marginparwidth=23mm, marginparsep=2mm]{geometry}

\DeclareMathOperator{\col}{\colon}
\DeclareMathOperator{\Pbb}{\mathbb{P}}
\DeclareMathOperator{\Ncl}{\mathcal{N}}
\newcommand{\inj}{\hookrightarrow}
\usepackage{subfig}
\usepackage{graphicx}

\newtheorem{theorem}{Theorem}[section]
\newtheorem{lemma}[theorem]{Lemma}
\newtheorem{corollary}[theorem]{Corollary}
\newtheorem{conjecture}[theorem]{Conjecture}
\newtheorem{proposition}[theorem]{Proposition}
\newtheorem{question}[theorem]{Question}

\newtheorem*{claim*}{Claim}
\newtheorem{claim}[theorem]{Claim}

\theoremstyle{definition}
\newtheorem{definition}[theorem]{Definition}

\theoremstyle{theorem}
\newtheorem{theoremA}{Theorem}

\newtheorem{corollaryA}[theoremA]{Corollary}

\makeatletter                                                   %allows to repeat theorems
\newtheorem*{rep@theorem}{\rep@title}
\newcommand{\newreptheorem}[2]{%
	\newenvironment{rep#1}[1]{%
		\def\rep@title{#2 \ref{##1}}%
		\begin{rep@theorem}}%
		{\end{rep@theorem}}}
\makeatother

\newreptheorem{theorem}{Theorem}
\newreptheorem{corollary}{Corollary}
\newreptheorem{proposition}{Proposition}

\theoremstyle{remark}
\newtheorem{remark}[theorem]{Remark}
\newtheorem{example}[theorem]{Example}

\usepackage{color}
\usepackage{pdfcolmk}

\newcounter{joecomments}

\newcounter{ncomments}

\def\Icl{\mathcal I}

\def\Z{\mathbb Z}
\def\N{\mathbb N}
\def\R{\mathbb R}

\def\Q{\mathbb Q}
\def\scl{\mathrm{scl}}
\def\fsn{\mathrm{fsn}}

\def\gb{\bar{g}}
\def\defeq{\vcentcolon=}
\def\Scl{\mathcal{S}}

\def\wtt{\texttt{\rm w}}
\def\xtt{\texttt{\rm x}}

\def\att{\texttt{\rm a}}
\def\btt{\texttt{\rm b}}

\def\vtt{\texttt{\rm v}}

\def\Arm{\mathrm{A}}
\def\Crm{\mathrm{C}}
\def\Vrm{\mathrm{V}}
\def\Erm{\mathrm{E}}
\def\Gcl{\mathcal{G}}
\def\supp{\mathrm{supp}}
\def\opp{\mathrm{opp}}
\def\pf{\mathrm{pf}}
\def\st{\mathrm{St}}
\def\Lk{\mathrm{Lk}}

\numberwithin{equation}{section}

\usepackage{pdfpages}

\title{Stable commutator length in right-angled Artin and Coxeter groups}
\author{Lvzhou Chen}
\address{Department of Mathematics\\ The University of Texas at Austin \\ Austin, TX, USA}
\email[L.~Chen]{lvzhou.chen@math.utexas.edu}

\author{Nicolaus Heuer}
\address{Department of Pure Mathematics and Mathematical Statistics \\
	Centre for Mathematical Sciences \\
	University of Cambridge}
\email[N.~Heuer]{nh441@cam.ac.uk}

\date{}

\begin{document}
\maketitle

\begin{abstract}
We establish a spectral gap for stable commutator length (scl) of integral chains in right-angled Artin groups (RAAGs).
We show that this gap is \emph{not} uniform, i.e.\ there are RAAGs and integral chains with scl arbitrarily close to zero. 
We determine the size of this gap up to a multiplicative constant in terms of the 
\emph{opposite path length} of the defining graph. 
This result is in stark contrast with the known uniform gap $1/2$ for elements in RAAGs.
We prove an analogous result for right-angled Coxeter groups.

In a second part of this paper we relate certain integral chains in RAAGs to the \emph{fractional stability number} of graphs. 
This has several consequences:
Firstly, we show that every rational number $q \geq 1$ arises as the stable commutator length of an integral chain in some RAAG. 
Secondly, we show that computing scl of elements and chains in RAAGs is NP hard. 
Finally, we heuristically relate the distribution of $\scl$ for random elements in the free group to the distribution of fractional stability number in random graphs.

We prove all of our results in the general setting of graph products. In particular all above results hold verbatim for right-angled Coxeter groups. 
\end{abstract}

\section{Introduction}
%Let $X$ be a topological space and let $\gamma \col \sqcup_{j=1}^k S^1_j \to X$ be a collection of $k$ loops in $X$.
%The stable commutator length $\scl_X(\gamma)$ of $\gamma$ in $X$ is the least complexity of surfaces bounding $\gamma$ in the following sense.
%A surface $\Sigma$ with boundary $\partial \Sigma$ is \emph{admissible} for $\gamma$ if there is a map $f \col \Sigma \to X$ such that 
%its boundary map $\partial f$ makes the following diagram commute 
%$$
%\begin{CD}
%	\partial \Sigma @>{i}>> \Sigma\\
%	@V{\partial f}VV @V{f}VV\\
%	\sqcup_{j=1}^k S^1_j @>{\gamma}>> X.
%\end{CD}
%$$
%and such that $\partial f_*[\partial \Sigma]=n(\Sigma,f)[\sqcup S^1_j]$ for some positive integer $n(\Sigma, f)$, called the \emph{degree of $f$}.
%Then the stable commutator length (scl) is defined as 
%$$
%\scl_X(\gamma) := \inf \frac{-\chi^-(\Sigma)}{2 \cdot n(\Sigma,f)},
%$$
%where the infimum is taken over all admissible surfaces and $\chi^-(\Sigma)$ denotes the Euler characteristic of $\Sigma$ ignoring sphere and disk components (see Definition \ref{defn: scl via euler charac}). If no such admissible surfaces exist, we set $\scl_X(\gamma) = \infty$. 
%The stable commutator length of $\gamma$ only depends on $G=\pi_1(X)$ and the conjugacy classes $g_1,\cdots,g_k$ representing the free homotopy classes of the components of $\gamma$. 
%Thus we can define 
%$\scl$ on groups by setting $\scl_G(g_1 + \cdots + g_k) = \scl_{X}(\gamma)$. We refer to formal sums $g_1 + \cdots + g_k$ of elements as \emph{chains}. %with finite $\scl$

The stable commutator length is a relative version of the Gromov--Thurston norm. 
For a finite collection of loops $\gamma_1,\cdots,\gamma_k$ in a topological space $X$, its stable commutator length is the least complexity of surfaces bounding it, 
measured in terms of Euler characteristics (see Definition \ref{defn: scl via euler charac}). 
This only depends on the fundamental group $G=\pi_1(X)$ and the conjugacy classes $g_1,\cdots,g_k$ representing the free homotopy classes of $\gamma_1,\cdots,\gamma_k$, 
and it is denoted as $\scl_G(g_1+\cdots+g_k)$. We call this the stable commutator length of the (integral) chain $g_1+\cdots+g_k$.

%Let $G$ be a group and let $[G, G]$ be its commutator subgroup. 
%For an element $g \in [G,G]$, its commutator length $\cl_G(g)$ of $g$ is the least number of commutators needed to express $g$ as their product. 
%Its stable commutator length is $\scl_G(g) := \lim_{n \to \infty} \cl(g^n)/n$.
%
%The stable commutator length has the following geometric meaning: if $X$ is a topological space with fundamental group $G$ and if $\gamma \col S^1 \to X$ is a loop corresponding to an element $g \in [G,G]$. Then $\scl_G(g)$ measures the least complexity of surfaces bounding $\gamma$.
%This notion naturally generalizes to formal sums $c= g_1 + \cdots g_n$ of elements $g_i \in G$ such that $g_1 \cdots g_n \in [G,G]$. Those sums are called (null-homologous integral) \emph{chains}. The stable commutator length of a chain $g_1 + \cdots + g_n$ measures the least complexity of surfaces that bound loops $\gamma_1, \ldots, \gamma_n$ corresponding to the conjugacy classes $g_1, \ldots, g_n$ (see Definition \ref{defn: scl via euler charac}).

%The geometric definition interprets the stable commutator length as a relative version of the Gromov--Thurston norm on $H_2(X;\R)$. 
The stable commutator length arises naturally in geometry, topology and dynamics and has seen a vast development in recent years by Calegari and others 
\cite{Cal:rational,CF:sclhypgrp,BBF,Chen:sclBS,HL2020spectrum}.

A group $G$ has a \emph{spectral gap $C>0$ for elements} (resp. \emph{chains}) if $\scl_G(g)\notin(0,C)$ for any element (resp. any chain) $g$ in $G$.
The largest such $C$ is called the \emph{optimal} spectral gap of $G$ for elements (resp. \emph{chains}).
Various kinds of groups are known to have a gap for elements: word-hyperbolic groups \cite{CF:sclhypgrp}, finite index subgroups of mapping class groups \cite{BBF}, 
subgroups of right-angled Artin groups (defined below) \cite{Heuer}, and $3$-manifold groups \cite{CH:sclgap}; see Theorem \ref{thm: gap for elements summary}.
The spectral gap property can be used to obstruct group homomorphisms since the stable commutator length is non-increasing under homomorphisms.

In contrast, much less is known about spectral gaps for chains.
Calegari--Fujiwara \cite{CF:sclhypgrp} showed that hyperbolic groups have a spectral gap for chains.
Their estimates have been made uniform and explicit in certain families of hyperbolic groups (Theorem \ref{thm: gap for chains summary}).
To our best knowledge, all previously known nontrivial examples with a spectral gap for chains are direct products of hyperbolic groups.

In this article we establish a \emph{spectral gap} for chains in right-angled Artin groups. 
The \emph{right-angled Artin group $\Arm(\Gamma)$} associated to a simplicial graph $\Gamma$ is the group with presentation
$$
\Arm(\Gamma) = \langle \Vrm(\Gamma) \mid [v,w] ; (v,w) \in \Erm(\Gamma) \rangle,
$$
which is not hyperbolic unless the graph contains no edge.
Such groups are of importance due to their rich subgroup structure \cite{wise, HaglundWise, Agol, martin1, martin2}.

%We show that right-angled Artin groups have a spectral gap for chains, and the optimal gap can be controlled by the \emph{opposite path length}.
%\begin{definition}[Opposite Path Length]
%For an integer $m\ge1$, the \emph{opposite path of length $m$} is the simplicial graph $\Delta_m$ with vertex set $\Vrm(\Delta_m) = 
%\{ v_0, \ldots, v_m \}$ and edge set $\Erm(\Delta_m) = \{ (v_i, v_j) \mid |i-j| \geq 2 \}$.
%We define the \emph{opposite path length} of a simplicial graph $\Gamma$ to be
%$$
%\Delta(\Gamma) \defeq \min \{ m \mid \Delta_m \mbox{ is an induced subgraph of } \Gamma \}.
%$$
%Here a subgraph $\Lambda$ of $\Gamma$ is induced if any edge in $\Gamma$ connecting $u,v\in \Lambda$ belongs to $\Lambda$.
%\end{definition}

The gap is controlled by an invariant $\Delta(\Gamma)\ge0$ of the defining simplicial graph $\Gamma$ that we introduce, 
called \emph{the opposite path length}; see Section \ref{subsec: spectral gaps}.

\begin{theoremA} \label{theorem:main for raags}
	Let $G$ be the right-angled Artin %(or Coxeter) 
	group associated to a simplicial graph $\Gamma$.
	Then the optimal spectral gap for integral chains in $G$ is at least $\frac{1}{24 + 12 \Delta(\Gamma)}$ and at most $\frac{1}{\Delta(\Gamma)}$.
	%Moreover, any integral chain $c$ with $\scl_G(c)=0$ is equivalent to the trivial chain in the sense of Definition \ref{def: equivalent chains}. 
	%This is close to optimal as there is an integral chain $c$ with $\scl_G(c)\in [\frac{1}{24 + 12 \Delta}, \frac{1}{\Delta}]$.
\end{theoremA}

%Our results are proved in the more general setting of graph products.

The gap cannot be uniform among all right-angled Artin groups as there is an explicit finite graph $\Delta_m$ with $\Delta(\Delta_m)=m$ for any $m\in\Z_+$.
The nonuniformness of the gap for chains is striking since right-angled Artin groups are known to have a uniform spectral gap $1/2$ for elements \cite{Heuer,RAAGgap1}. 
This is the first class of groups where the optimal gap for chains is known to be different from the optimal gap for elements.
Using the nonuniformness, we construct countable groups where this difference becomes more apparent (Section \ref{subsec:gps new scl spec}).

We prove these results in the much more general setting of graph products (Theorem \ref{theorem: graph prod}).
In particular, Theorem \ref{theorem:main for raags} holds verbatim for right-angled Coxeter groups, which are defined in the same way as right-angled Artin groups except that generators have order $2$.
For right-angled Coxeter groups, no gap was previously known in general, even for elements. 

For a simplicial graph $\Gamma$ we will construct a graph $D_{\Gamma}$ and a chain $c_{\Gamma}$ in $\Arm(D_{\Gamma})$, called the \emph{double chain of $\Gamma$}; see Definition \ref{def:double chain}.
We will relate the stable commutator length of this chain linearly to the \emph{fractional stability number} (Definition \ref{def: fsn}) of $\Gamma$ (Theorem \ref{theorem:scl and fsn}).
The latter invariant is well studied \cite{scheinerman2011fractional}. It is known that computing the fractional stability number is NP hard \cite{grotschel1981ellipsoid} and that every rational number $q \geq 2$ is the fractional stability number of some graph \cite[Proposition 3.2.2]{scheinerman2011fractional}.
As consequences of this connection, we obtain the following two theorems.

\begin{theoremA}[NP-hardness, Theorem \ref{thm: np hard sec}] \label{theorem:np hard}
	Unless P=NP, there is no algorithm that, given a simplicial graph $\Gamma$, an element $w \in \Arm(\Gamma)$ and a rational number $q \in \Q^+$, decides if $\scl_{\Arm(\Gamma)}(w) \leq q$ with polynomial run time in $|\Vrm(\Gamma)| + |w|$. The same holds for chains. 
\end{theoremA}
This is in stark contrast to the case of free groups, as there is an algorithm by Calegari computing stable commutator length with polynomial run time in the word length of the input \cite{Cal:rational, Cal:scallop}; also compare to \cite{heuer2020computing}. 

\begin{theoremA}[Rational Realization, Theorem \ref{thm: rat real sec}] \label{theorem:rat realize}
	For every rational $q \in \Q_{\geq 1}$ there is an integral chain $c$ in a right-angled Artin group $A(\Gamma)$ such that $\scl_{\Arm(\Gamma)}(c) = q$.
\end{theoremA}
In the case of free groups, it is an unsolved conjecture of Calegari--Walker that the set of values of stable commutator lengths is dense in some intervals.

In the following subsections we will describe the generalization of our results to graph products, and collect some further results.

\subsection{Spectral Gaps for Integral Chains: Overview of the proof}\label{subsec: spectral gaps}
We now state the generalization of Theorem \ref{theorem:main for raags} to graph products and describe the main steps in its proof.

For a simplicial graph $\Gamma$, let $\{ G_v \}_{v \in \Vrm(\Gamma)}$ be a family of groups indexed by the vertex set $\Vrm(\Gamma)$ of $\Gamma$. 
The \emph{graph product} for this data is the free product $\star_{v \in \Vrm(\Gamma)} G_v$ subject to the relations $[g_v,h_w]$ 
for every $g_v \in G_v$ and $h_w \in G_w$ whenever $(v,w) \in \Erm(\Gamma)$ is an edge of $\Gamma$.
Graph products are generalizations of both right-angled Artin groups (which have vertex groups $\Z$) and right-angled Coxeter groups (which have vertex groups $\Z/2$).

For an integer $m\ge1$, the \emph{opposite path of length $m$} is the simplicial graph $\Delta_m$ with vertex set $\Vrm(\Delta_m) = 
\{ v_0, \ldots, v_m \}$ and edge set $\Erm(\Delta_m) = \{ (v_i, v_j) \mid |i-j| \geq 2 \}$.
We define the \emph{opposite path length} of a simplicial graph $\Gamma$ to be
$$
\Delta(\Gamma) \defeq \max \{ m \mid \Delta_m \mbox{ is an induced subgraph of } \Gamma \}.
$$
Here a subgraph $\Lambda$ of $\Gamma$ is induced if any edge in $\Gamma$ connecting $u,v\in \Lambda$ belongs to $\Lambda$.

\begin{theoremA}[Theorem \ref{thm: gap for graph products}]\label{theorem: graph prod}
	Let $\Gamma$ be a simplicial graph, let $\{ G_v \}_{v \in \Vrm(\Gamma)}$ be a family of groups and let $\Gcl(\Gamma)$ be the associated graph product. 
	If $c$ is an integral chain in $\Gcl(\Gamma)$ then either
	$\scl_{\Gcl(\Gamma)} (c) \geq \frac{1}{12\Delta(\Gamma)+24}$ or $c$ is equivalent (see below) to a chain supported on the vertex groups, 
	called \emph{a vertex chain}. For vertex chains, there is an algorithm to compute $\scl_{\Gcl(\Gamma)}(c)$ in terms of the stable commutator lengths in the vertex groups.
	
	Moreover, there is an integral chain $\delta$ on $\Gcl(\Gamma)$ such that 
	$$
	\frac{1}{12(\Delta(\Gamma)+2)} \leq \scl_{\Gcl(\Gamma)}(\delta) \leq \frac{1}{\Delta(\Gamma)}.$$
\end{theoremA}

The equivalence relation of chains, roughly speaking, is based on the following moves that does not change the stable commutator length. In an arbitrary group $G$ with a chain $c$ and elements $g,h \in G$ we have $\scl_G(c + g^n) = \scl_{G}(c + n \cdot g)$ for every $n \in \Z$ and $\scl_G(c + g) = \scl_G(c + h g h^{-1})$. If in addition $g$ and $h$ commute, we have $\scl_G(c + g \cdot h) = \scl_G(c + g + h)$. We say that two chains $c$, $c'$ in $G$ are \emph{equivalent}, if $c$ can be transformed into $c'$ by a finite sequence of these identities. See Definition \ref{def: equivalent chains} for the precise definition.

Formally, a \emph{vertex chain} is of the form $c=\sum_v c_v$, where each $c_v$ is a chain in the vertex group $G_v$. For right-angled Artin groups and right-angled Coxeter groups, any null-homologous vertex chain is equivalent to the zero chain and has zero scl. 
Thus Theorem \ref{theorem:main for raags} immediately follows from Theorem \ref{theorem: graph prod}.
Moreover, we have a uniform gap $1/60$ for all \emph{hyperbolic} right-angled Coxeter groups; see Corollary \ref{cor: hyp RACGs}.

In particular, Theorem \ref{theorem: graph prod} implies that groups with a gap for integral chains are preserved under taking graph products over finite graphs; see Corollary \ref{cor: gap pres by graph prod}.

%The spectral gap result in Theorem \ref{theorem: graph prod} is based on our estimates in graphs of groups (Section \ref{subsubsec:gaps on graph of grps}).

\subsubsection{Gaps for chains in graphs of groups} \label{subsubsec:gaps on graph of grps}
The spectral gap result in Theorem \ref{theorem: graph prod} is based on a simple criterion for spectral gaps in graphs of groups that we prove. 
For simplicity, we state it for amalgamations.
\begin{theoremA}[Theorem \ref{thm: no long pairing}, Long Pairings] \label{theorem:no long pair}
Let $G=A\star_C B$ be an amalgamation and let $\sum_{i \in I} g_i$ be an integral chain. Then either
$$
\scl_G(c) \geq \frac{1}{12 N}
$$
or $c = \sum_{i \in I} g_i$ has a term $g=g_i$ such that $g^N = h^k h' d$ as reduced elements, where $h$ is cyclically conjugate to the inverse of some term $g_j$ in $c$, 
$h'$ is a prefix (Definition \ref{def:ab-alternating}) of $h$ and $d \in C$.
\end{theoremA}

We give two proofs of this criterion in Section \ref{sec:gaps from short overlaps}, one using surfaces and the other using quasimorphisms.

To make use of this criterion, we reduce chains so that the exceptional algebraic relation $g^N = h^k h' d$ does not occur for a suitable $N$.
In Section \ref{sec:cm subgroups} we develop tools to achieve this goal for $N=D+2$, provided that all edge groups are \emph{BCMS-$D$ subgroups} (Definition \ref{def:bcmsm}).
BCMS-$D$ subgroups are generalizations of malnormal and central subgroups. In particular, malnormal subgroups are BCMS-$1$ and central subgroups are BCMS-$0$. 

As key examples, for a graph product over a graph $\Gamma$, the subgroup corresponding to any induced subgraph is BCMS-$D$ if the opposite path length $\Delta(\Gamma)=D$.
Theorem \ref{theorem: graph prod} is obtained from the following estimate.

\begin{theoremA}[Theorem \ref{thm:CM gap}, BCMS gap]\label{theorem: gap for BCMS}
Let $G$ be a graph of groups such that the embedding of every edge group $C\le G$ is a BCMS-$D$ subgroup. Then for any integral chain $c$ in $G$,
either $c$ is equivalent to a chain supported on vertex groups, or
$$
\scl_G(c) \geq \frac{1}{12 (D+2)}.
$$
\end{theoremA}

%We get the following immediate corollary:
The special case where every edge group is malnormal in $G$ is equivalent to that
the fixed point set of each $g\neq id\in G$ has diameter at most $1$ for the action on the Bass--Serre tree. In this case, we have the following corollary
\begin{corollaryA}
Let $G$ be a graph of groups such that the embedding of each edge group $C\le G$ is malnormal.
Then for any integral chain $c$ in $G$,
either $c$ is equivalent to a chain supported on vertex groups, or
$$
\scl_G(c) \geq \frac{1}{36}.
$$
\end{corollaryA}

A similar result was obtained by Clay--Forester--Louwsma for hyperbolic elements in a group acting $K$-acylindrically on a simplicial tree; see \cite[Theorem 6.11]{CFL16}.

\subsection{fractional stability number and stable commutator length}\label{subsec:scl and fsn}
The algorithm mentioned in Theorem \ref{theorem: graph prod} computes stable commutator lengths of vertex chains as certain graph-theoretic quantities; see Section \ref{sec:vertex chains}. 
As a special case, we discover a connection between stable commutator lengths of certain chains in right-angled Artin groups and the fractional stability numbers of graphs.

\begin{definition}[Double Graphs and double chains] \label{def:double chain} 
For a simplicial graph $\Gamma$ with vertices $\Vrm(\Gamma)$ and edges $\Erm(\Gamma)$ we define the double graph $D_{\Gamma}$ as the graph with vertex and edge set 
\begin{eqnarray*}
\Vrm(D_\Gamma) &=& \{ \att_v, \btt_v \mid v \in \Vrm(\Gamma) \} \mbox{ and} \\
\Erm(D_\Gamma) &=& \{ (\att_v, \att_w), (\att_v, \btt_w), (\btt_v, \att_w), (\btt_v, \btt_w) \mid (v,w) \in \Erm(\Gamma) \}.
\end{eqnarray*}
Let $d_{\Gamma}$ be the integral chain $\sum_{v \in \Vrm(\Gamma)} [\att_v, \btt_v]$. We call $D_\Gamma$ the \emph{double graph} and $d_{\Gamma}$ the double chain in $\Arm(D_\Gamma)$.
\end{definition}
A key feature of this construction is that $\Arm(D_\Gamma)$ is the graph product over the graph $\Gamma$ with vertex groups $F( \att_v, \btt_v)$. 

\begin{definition}[fractional stability number] \label{def: fsn}
A \emph{stable measure} is a collection of non-negative weights $x = \{ x_v \}_{v \in \Vrm }$ assigned to vertices of $\Gamma$ such that for any clique $C$ (i.e. a complete subgraph) in $\Gamma$ we have that $\sum_{c \in C} x_c \leq 1$.
The \emph{fractional stability number} of $\Gamma$ is the supremum of $\sum_{v \in \Vrm} x_v$ over all stable measures and denoted by $\fsn(\Gamma)$. 
\end{definition}

%For a simplicial graph $\Gamma$, a \emph{stable set} is a subset $S \subset \Vrm = \Vrm(\Gamma)$ of vertices in $\Gamma$ such that no two elements in $S$ are connected by an edge. 
%The largest size of stable sets is called the \emph{stability number of $\Gamma$.} 
%A \emph{stable measure} is a collection of non-negative weights $x = \{ x_v \}_{v \in \Vrm }$ assigned to vertices of $\Gamma$ such that for any clique $C$ (i.e. a complete subgraph) in $\Gamma$ we have that $\sum_{c \in C} x_c \leq 1$.
%The \emph{fractional stability number} of $\Gamma$ is the supremum of $\sum_{v \in \Vrm} x_v$ over all stable measures.

%%Using the algorithm to compute $\scl$ in vertex chains in Theorem \ref{theorem: graph prod}, we can relate the fractional stability number to stable commutator length for certain chains.
%For a simplicial graph $\Gamma$ with vertices $\Vrm(\Gamma)$ and edges $\Erm(\Gamma)$ we define the double graph $D_{\Gamma}$ as the graph with vertex and edge set 
%\begin{eqnarray*}
%	\Vrm(D_\Gamma) &=& \{ \att_v, \btt_v \mid v \in \Vrm(\Gamma) \} \mbox{ and} \\
%	\Erm(D_\Gamma) &=& \{ (\att_v, \att_w), (\att_v, \btt_w), (\btt_v, \att_w), (\btt_v, \btt_w) \mid (v,w) \in \Erm(\Gamma) \}.
%\end{eqnarray*}
%Let $d_{\Gamma}$ be the integral chain $\sum_{v \in \Vrm(\Gamma)} [\att_v, \btt_v]$. We call $D_\Gamma$ the \emph{double graph} and $d_{\Gamma}$ the double chain in $\Arm(D_\Gamma)$.

\begin{theoremA}[\protect{Theorem \ref{thm: scl and fsn}}] \label{theorem:scl and fsn}
	Let $\Gamma$ be a graph and let $D_\Gamma$ and $d_\Gamma$ be the associated double graph and double chain respectively.  Then
	$$
	\scl_{\Arm(D_{\Gamma})}(d_\Gamma) = \frac{1}{2}\cdot \fsn(\Gamma).
	$$
	%where $\fsn(\Gamma)$ denotes the fractional stability number of $\Gamma$.
\end{theoremA}

Combining with known results about fractional stability numbers, we deduce Theorems \ref{theorem:np hard} and \ref{theorem:rat realize}. 
See Section \ref{sec:vertex chains} for the more general results about computations of stable commutator lengths of vertex chains in graph products.

The distributions of stable commutator length of random elements in free groups and $\fsn$ of random graphs are depicted in Figure \ref{figure:scl and fsn intro}.
They exhibit a strikingly similar behavior: For both distributions values with small denominators appear more frequently, and the histograms exhibit some self-similarity.
In Section \ref{subsec:statistics of scl and fsn} we analyze the distribution of $\scl$ and $\fsn$ further. This analysis allows us to describe a $5$-parameter random variable $X$ (Definition \ref{defn:dist X}) which exhibits qualitatively the same distribution as $\scl$ and $\fsn$. 
We use $X$ to model both $\scl$ and $\fsn$ in Figure \ref{figure:scl and fsn sec}. 
While this is purely heuristic, it suggests that the distribution of $\scl$ and $\fsn$ converge to a similar distribution for large words or graph sizes; see Question \ref{quest: dist scl and fsn similar}.

\begin{figure}[!tbp]
	\centering
	\subfloat[Histogram of $\scl_{F_2}(w)$ for $50.000$ random words $w$ uniformly choosen of length $24$ in $F_2$ ]{\includegraphics[width=0.7\textwidth]{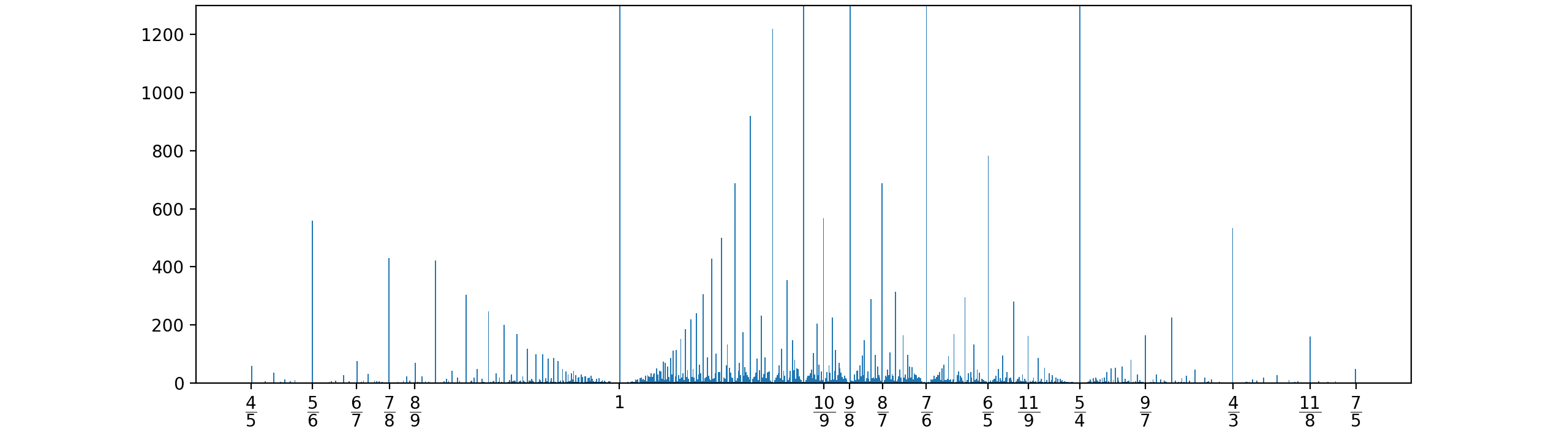}\label{fig:scl}}
	\\
	\subfloat[Histogram of $\scl_{\Arm(D_\Gamma)}(d_\Gamma) = \frac{1}{2} \cdot \fsn(\Gamma)$ for $50.000$ random graphs $\Gamma$ uniformly choosen on 24 vertices. Here, $d_\Gamma$ is the double chain in $\Arm(D_\Gamma)$ (Definition \ref{def:double chain}).]{\includegraphics[width=0.7\textwidth]{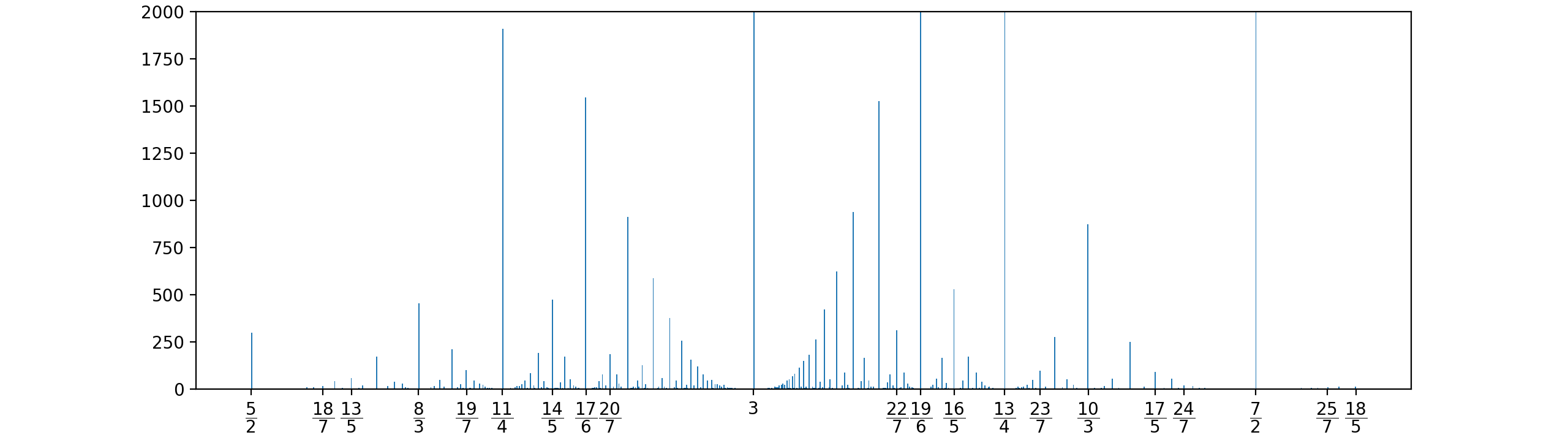}\label{fig:fsn}}
	\caption{$\scl$ for random words in the free group vs. $\scl$ of random chains $d_\Gamma$ in right-angled Artin groups $\Arm(D_\Gamma)$.
		In both cases, $\scl$ is rational and values with small denominator appear more frequent and the histogram exhibits a fractal behavior.
		In Section~\ref{subsec:statistics of scl and fsn} we explain this distribution as the interference of (rounded) Gaussian distributions.} \label{figure:scl and fsn intro}
\end{figure}

%The distributions of $\scl$ and $\fsn$ in Figure \ref{figure:scl and fsn intro}
%are qualitatively very similar: In both cases, values with small denominators appear more frequently, and the histograms exhibit some self-similarity.
%In Section \ref{subsec:statistics of scl and fsn} we analyze the distribution of $\scl$ and $\fsn$ further. 
%We describe a $5$-parameter random variable $X$ (Definition \ref{defn:dist X}) which exhibits qualitatively the same distribution as $\scl$ and $\fsn$. 
%We use $X$ to model both $\scl$ and $\fsn$ in Figure \ref{figure:scl and fsn sec}. 
%While this is purely heuristic, it suggests that the distribution of $\scl$ and $\fsn$ converge to a similar distribution for large words or graph sizes; see Conjecture \ref{conj: dist scl and fsn similar}.

%\begin{conjectureA}[Conjecture \ref{conj: dist scl and fsn similar}]
%There is a distribution $Y$ such that for large $n \to \infty$, both the random variable $\scl(w)$, for $w$ uniformly chosen from $\{ w \in [F_2, F_2] \mid |w| = 2 \cdot n \}$ and $\fsn(\Gamma)$ where $\Gamma$ is uniformly chosen among all graphs with $n$ vertices converge almost surely to $Y$.  
%\end{conjectureA}

\subsection{Groups with interesting scl spectrum} \label{subsec:gps new scl spec}
The scl \emph{spectrum} of a group is the range of the map $\scl_G: [G,G]\to \R_{\ge0}$.
The nonuniformness of spectral gap in Theorem \ref{theorem:main for raags} allows us to construct groups with interesting spectrum.
There are few (classes of) groups where the spectrum of $\scl$ is fully known; see \cite[Remark 5.20]{Cal:sclbook} and \cite{heuer2019full, zhuang2007irrational}. 

\begin{theoremA} \label{theorem: no uniform chain gap}
There is a countable (right-angled Artin) group $G$ such that $\scl_G(g) \ge 1/2$ for all $g\neq id \in [G,G]$ but there is no spectral gap for chains in $G$.
\end{theoremA}

\begin{theoremA} \label{theorem: eventually dense}
There is a countable group $G$ such that
$\scl_G(g) \geq 1/2$ for all $g\neq id \in [G,G]$, and its scl spectrum is dense in $[3/2, \infty)$.
\end{theoremA}
To the authors' best knowledge there was no group known that has a spectral gap for elements and the spectrum of elements becomes eventually dense, 
though free groups are conjectured to have this property.

%In fact, we may choose $C = 1/2$, $D = 3/2$ and $G = \Arm(\Delta_\infty) \star F_3$ in the last Theorem.

These results are proved in Section \ref{subsec: scl spec on a delta inf}.

\subsection{Organization} \label{subsec:organ}
This article is organized as follows. In Section \ref{sec:Background} we recall basic results of stable commutator length, graph of groups and graph products respectively.
In Section \ref{sec:gaps from short overlaps} we prove Theorem \ref{theorem:no long pair} estimating stable commutator length in graphs of groups.
In Section \ref{sec:cm subgroups} we will develop the theory of BCMS-$D$ subgroups and prove Theorem \ref{theorem: gap for BCMS}.
Then we apply this to graph products of groups and prove Theorem \ref{theorem: graph prod} in Section \ref{sec: gap in graph prod of groups}.
Finally in Section \ref{sec:vertex chains} we compute stable commutator lengths of vertex chains in graph products and relate them to fractional stability numbers.

\subsection*{Acknowledgments}
We would like to thank Danny Calegari, Daniel Groves, Elia Fioravanti, Jason Manning, Henry Wilton and Daniel Wise for helpful discussions. 
We are grateful to Jing Tao for suggesting a characterization of elements with zero scl in right-angle Coxester groups, which helped us discover and correct a mistake in an earlier version. 
We really appreciate the referee for the constructive suggestions, which helped improving the exposition.
Nicolaus Heuer is supported by the Herchel Smith Fund.

\setcounter{tocdepth}{1}
\tableofcontents

\section{Background} \label{sec:Background}
We briefly introduce several concepts and set up some notations related to stable commutator length and graphs of groups.
All results in this section are standard. Readers familiar with these topics may skip this section and refer to it when necessary. 
%\ncomm{.. or not? I feel that some of the graph product stuff might be new?} 

\subsection{Stable Commutator Length}
We give the precise definition of the stable commutator length (scl) and recall some basic results. The reader may refer to \cite[Chapter 2]{Cal:sclbook} for details.

Given a group $G$, let $X$ be a topological space with fundamental group $G$. 
An integral chain is a finite formal sum of elements in $G$.
Given an integral chain $\sum g_j$, consider loops $\gamma_j$ in $X$ so that
the free homotopy class of $\gamma_j$ represents the conjugacy class of $g_j$ for each $j$.

An \emph{admissible surface} is a pair $(S,f)$, where $S$ is a compact oriented surface and $f:S\to X$ is a continuous map such that
the following diagram commutes and $\partial f_*[\partial S]=n(S,f)[\sqcup S^1_j]$ for some integer $n(S,f)>0$,
called the \emph{degree} of the admissible surface.
\[
\begin{CD}
	\partial S @>{i}>> S\\
	@V{\partial f}VV @V{f}VV\\
	\sqcup S^1_j @>{\sqcup \gamma_j}>> X
\end{CD}
\]
Admissible surfaces exist if the chain is null-homologous, i.e. $\sum [g_j]=0\in H_1(G;\Q)$.
Let $\chi^-(S)$ be the Euler characteristic of $S$ after removing disk and sphere components.

\begin{definition}\label{defn: scl via euler charac} 
	For any null-homologous integral chain $\sum g_j$ in $G$, we define
	$$\scl_G(\sum g_j)\defeq\inf_{(S,f)}\frac{-\chi^-(S)}{2\cdot n(S,f)}.$$
	When the chain represents a nontrivial rational homology class, we make the convention that $\scl_G(\sum g_j)=+\infty$.
\end{definition}

We often omit the map $f$ and refer to an admissible surface $(S,f)$ simply as $S$.

In the special case where the chain is an element $g\in[G,G]$, this agrees with the algebraic definition 
using commutator lengths. 
See \cite[Chapter 2]{Cal:sclbook} for more details as well as an algebraic definition for scl of integral chains.

\begin{lemma}\label{lemma: pants chain}
	For an integral chain $c=gh-g-h$ with $g,h\in G$, we have $\scl_G(c)\le 1/2$.
\end{lemma}
\begin{proof}
	The fundamental group of a pair of pants $S$ is the free group of rank $2$, where the generators $a,b$ can be chosen so that the boundary loops with the induced orientation are represented by $ab$, $a^{-1}$ and $b^{-1}$ respectively. The homomorphism $F_2\to G$ determined by $a\mapsto g$ and $b\mapsto h$ corresponds to a map $f:S\to X$, which provides an admissible surface for $c$ of degree one. Hence 
	$$\scl_G(c)\le \frac{-\chi(S)}{2}=\frac{1}{2}.$$
\end{proof}

Let $C_1(G)$ be the space of real $1$-chains. By identifying $g^{-1}$ with $-g$ in $C_1(G)$, scl is defined for any
finite sum $\sum t_i g_i\in C_1(G)$, where $t_i\in \Z$. It is known that scl is linear on rays and satisfies the
triangle inequality, and thus extends to a (semi-)norm on $C_1(G)$.

\begin{proposition}[scl as a norm] \label{prop:scl is norm}
	Scl is a semi-norm on $C_1(G)$. In particular, $\scl(c_1 + c_2) \le \scl(c_1) + \scl(c_2)$ for any $c_1,c_2\in C_1(G)$.
\end{proposition}

\begin{definition}[Equivalent chains]\label{def: equivalent chains}
	Let $E(G)$ be the subspace of $C_1(G)$ spanned by elements of the following forms:
	\begin{enumerate}
		\item $g^n-n\cdot g$, where $n\in\Z$ and $g\in G$, \label{item: linear equiv}
		\item $hgh^{-1}-g$, where $g,h\in G$, and \label{item: conj equiv}
		\item $gh-g-h$, where $g$ and $h$ are \emph{commuting} elements in $G$.\label{item: comm equiv}
	\end{enumerate}
	We say two chains $c$ and $c'$ are \emph{equivalent} if they differ by an element in $E(G)$.
\end{definition}
Note that this is slightly different from the usual definition (e.g. \cite[Definition 2.78]{Cal:sclbook}) by adding (\ref{item: comm equiv}).

\begin{proposition}[scl of equivalent chains]\label{prop: equivalent chains have the same scl}
	If $c$ and $c'$ are equivalent chains then
	$$
	\scl_G(c) = \scl_G(c').
	$$
\end{proposition}
\begin{proof}
	Since scl is a semi-norm, this is to show that scl vanishes on each basis element of $E(G)$. For chains of the first two kinds, see \cite[Section 2.6]{Cal:sclbook}. For a chain $gh-g-h$, where $g$ and $h$ commute, since $(gh)^n=g^n\cdot h^n$ for any $n\in\Z_+$, there is a thrice-punctured sphere with boundary components representing $(gh)^n$, $g^{-n}$ and $h^{-n}$ respectively. This gives rise to an admissible surface $S$ for the chain $gh-g-h$ of degree $n$, which has $-\chi(S)=1$. Letting $n$ go to infinity, we have $\scl_G(gh-g-h)=0$.
\end{proof}

We collect a few properties of stable commutator length. The main reference is \cite{Cal:sclbook}.
\begin{proposition}[Monotonicity and Retract] \label{prop: mono and retract}
Let $H, G$ be groups and let $f: H \to G$ be a homomorphism. Then for any chain $c$ in $C_1(H)$ we have $\scl_{H}(c) \geq \scl_{G}(f(c))$. 
If in addition $H$ is a retract of $G$, i.e. there is a homomorphism $r:G\to H$ such that $r\circ f=id_H$, then for any chain $c$ in $H$ we have that $\scl_H(c) = \scl_G(c)$. 
\end{proposition}

\begin{proposition}\label{prop: vertex chains on free products}
	If $c=c_1+c_2$ is a chain in $G = G_1 \star G_2$, where $c_1$ is supported on $G_1$ and $c_2$ is supported on $G_2$ then  $\scl_G(c) =  \scl_{G_1}(c_1) + \scl_{G_2}(c_2)$.
\end{proposition}
\begin{proof}
	This is a special case of \cite[Theorem 6.2]{CH:sclgap} since $G$ is a graph of groups with vertex groups $G_1,G_2$ and a trivial edge group.
\end{proof}

\begin{proposition}[\protect{\cite[Theorem 2.101]{Cal:sclbook}}] \label{prop: formulas for scl}
Let $G$ be a group and let $c = \sum_{i=1}^n g_i$ be a chain. Let $\tilde{G} = G \star \langle t_1 \rangle \star \cdots \star \langle t_{n-1} \rangle $ be the free product of $G$ with $n-1$ infinite cyclic groups. Then 
$$
\scl_G(c) = \scl_{\tilde{G}}(g_1\cdot \prod_{i=1}^{n-1} t_i g_{i+1} t_i^{-1} ) - \frac{n-1}{2}.
$$
\end{proposition}

\begin{proposition}[Index formula \protect{\cite[Corollary 2.81]{Cal:sclbook}}] \label{prop: index formula}
	Let $H\trianglelefteq G$ be a finite index normal subgroup. The quotient $F=G/H$ acts on $H$ by outer-automorphisms $h\mapsto f.h$, where $f.h$ is a well-defined conjugacy class in $H$. Then for any $h\in H$, we have
	$$\scl_G(h)=\frac{1}{|F|}\scl_H(\sum_{f\in F} f.h).$$
\end{proposition}

\subsection{Quasimorphisms}

Let $G$ be a group. A map $\phi \col G \to \R$ is called a \emph{quasimorphism} if there is a constant $D > 0$ such that $|\phi(g) + \phi(h) - \phi(gh)| \leq D$ for all $g, h \in G$. The infimum of all such $D$ is called the \emph{defect} of $\phi$ and denoted by $D(\phi)$.
Every bounded map and every homomorphism to $\R$ are trivially quasimorphisms but there are many nontrivial examples; see Example \ref{exmp:brooks}. 
A quasimorphism is called \emph{homogeneous} if $\phi(g^n) = n \cdot \phi(g)$ for every $g\in G$ and $n \in \Z$. 
Every quasimorphism $\phi \col G \to \R$ has a unique associated homogeneous quasimorphism $\bar{\phi}$ defined via
$$
\bar{\phi}(g) := \lim_{n \to \infty} \frac{\phi(g^n)}{n}
$$
which we call the \emph{homogeneous representative} of $\phi$.
\begin{proposition}[Homogeneous Representative, \protect{\cite[Lemma 2.58]{Cal:sclbook}}]  \label{prop:homog rep}
Let $\phi \col G \to \R$ be a quasimorphism with defect $D(\phi)$. Then the homogeneous representative 
$\bar{\phi}$ is in bounded distance to $\phi$ and satisfies $D(\bar{\phi}) \leq 2 D(\phi)$.
\end{proposition}
Here two quasimorphisms $\phi, \psi \col G \to \R$ are in bounded distance if $\phi-\psi$ is bounded in the supremum norm.

Quasimorphisms are intimately connected to scl through Bavard's duality:
\begin{theorem}[Bavard's Duality Theorem \cite{bavard}, \protect{\cite[Theorem 2.79]{Cal:sclbook}}] \label{thm:bavard}
For any chain $c = \sum_{i \in I} n_i g_i$ with real coefficients $n_i \in \R$ we have
$$
\scl_G(c) = \sup_{\phi} \frac{\sum_{i \in I} n_i \phi(g_i)}{2D(\phi)},
$$
where the supremum is taken over all homogeneous quasimorphisms $\phi \col G \to \R$. Moreover, this supremum is achieved.
\end{theorem}

One can actually choose the homogeneous quasimorphism achieving the supremum in Bavard's duality to be the homogenization of a quasimorphism with nice properties. A quasimorphism $\phi$ is called \emph{antisymmetric} if $\phi(g)=-\phi(g^{-1})$ for all $g\in G$.
\begin{proposition}[Extremal Quasimorphisms] \label{prop:extremal qm}
Let $G$ be a group. For any chain $c$ in $G$ there is a quasimorphism $\phi \col G \to \R$ with $D(\phi)=1/4$ that achieves the supremum of Bavard's duality, i.e.\ such that
$$
\scl_G(c) = \bar{\phi}(c)
$$
where $\bar{\phi}$ is the homogenization of $\phi$. 
Moreover, we may choose $\phi$ to be antisymmetric.
%In particular, the increase of the defect $D(\bar{\phi})=1/2$ is maximal; see Proposition \ref{prop:homog rep}.
\end{proposition}

\begin{proof}
	The statement without the moreover part is well known, and follows from the proof of \cite[Theorem 2.70]{Cal:sclbook}. 
	Now suppose $\psi$ is such a quasimorphism with $D(\psi)=1/4$ and $\bar{\psi}(c)=\scl_G(c)$.
	Let $\phi(g)\defeq (\psi(g)-\psi(g^{-1}))/2$. Then $\phi$ is an antisymmetric quasimorphism with $D(\phi)\le D(\psi)=1/4$.
	It also follows by definition that $\bar{\phi}=\bar{\psi}$, and in particular $\bar{\phi}(c)=\bar{\psi}(c)=\scl_G(c)$.
	Thus by Barvard's duality, we must also have $D(\phi)\ge 1/4$ and hence $D(\phi)=1/4$. This gives us the desired quasimorphism $\phi$.
\end{proof}

\begin{lemma}\label{lemma: quasimorphism for commuting elements}
	For any homogeneous quasimorphism $\phi$ on $G$, we have $\phi(gh)=\phi(g)+\phi(h)$ if $g$ and $h$ commute.
\end{lemma}
\begin{proof}
	Note that for any $n\in\Z_+$ we have $(gh)^n=g^nh^n$ and
	$$|\phi(gh)-\phi(g)-\phi(h)|=\frac{1}{n}|\phi(g^nh^n)-\phi(g^n)-\phi(h^n)|\le D(\phi)/n.$$
	Taking $n\to\infty$ we have $\phi(gh)=\phi(g)+\phi(h)$.
\end{proof}

\begin{proposition}\label{prop:chains on direct products}
	Let $c$ be a chain in $G \cong G_1 \times G_2$. Then $c$ is equivalent to a chain $c_1 + c_2$ where $c_1$ is supported on $G_1$ and $c_2$ is supported on $G_2$,
	and $c_1,c_2$ are integral chains if $c$ is.
	Moreover,
	$$
	\scl_G (c) = \max \{ \scl_{G_1} (c_1), \scl_{G_2} (c_2) \}.
	$$
\end{proposition}
\begin{proof}
	Each element $g\in G$ can be written as $g_1g_2$ for some $g_1\in G_1$ and $g_2\in G_2$, and thus $g$ is equivalent to $g_1+g_2$ as chains. The first claim easily follows from this.
	
	Every homogeneous quasimorphism $\phi$ on $G$ restricts to quasimorphisms $\phi_1$ and $\phi_2$ on $G_1$ and $G_2$ respectively. %Using the projection $G\to G_i$, we think of $\phi_i$ as a quasimorphism on $G$, $i=1,2$. 
	Then for the decomposition $g=g_1g_2$ above for any $g\in G$, we have $\phi(g)=\phi(g_1)+\phi(g_2)=\phi_1(g_1)+\phi_2(g_2)$ by Lemma \ref{lemma: quasimorphism for commuting elements}. It follows that $D(\phi)=D(\phi_1)+D(\phi_2)$ and $\phi(c)=\phi_1(c_1)+\phi_2(c_2)$ for the decomposition above.
	
	Let $\phi$ be an extremal homogeneous quasimorphism for a chain $c$. For the decomposition $c=c_1+c_2$ and $\phi=\phi_1+\phi_2$, we have
	$$\scl_G(c)=\frac{\phi(c_1+c_2)}{2D(\phi)}\le\frac{|\phi_1(c_1)|+|\phi_2(c_2)|}{D(\phi_1)+D(\phi_2)}\le \max\left\{\frac{|\phi_1(c_1)|}{D(\phi_1)},\frac{|\phi_2(c_2)|}{D(\phi_2)}\right\}\le\max\{\scl_{G_1}(c_1),\scl_{G_2}(c_2)\}$$
	by Bavard's duality. This proves the second claim since the other direction $\scl_G(c_1+c_2)\ge \scl_{G_i}(c_i)$ follows by the monotonicity of scl under the projection $G\to G_i$, where $i=1,2$.
\end{proof}

\begin{example}[Brooks Quasimorphisms]\label{exmp:brooks} 
We describe a family of quasimorphisms on non-abelian free groups that certify a spectral gap in free groups. 
%These examples are due to Brooks \cite{brooks}. We will generalize these maps to maps on graph of groups in Section \ref{subsec:no long pairings qm}. 
Let $F(\Scl)$ be the free group on a generating set $\Scl$ and let $w \in F(\Scl)$ be a reduced word. For an element $x \in F(\Scl)$, let $\nu_w(x)$ be the maximal number of times that $w$ is a subword of $x$ i.e.\ the maximal $n$ such that $x = x_0 w x_1 \cdots w x_n$, where $x_0, \ldots, x_n \in F(\Scl)$ and this expression is reduced. We define $\phi_w \col F(\Scl) \to \Z$ via $\phi_w \col x  \mapsto \nu_w(x) - \nu_{w^{-1}}(x)$. This map is called the \emph{Brooks quasimorphism for $w$}. The family of these maps were introduced by Brooks in \cite{brooks} to show that the vector space of quasimorphisms is infinite dimensional. We will generalize Brooks quasimorphisms from free groups to amalgamated free products and HNN extensions in Section \ref{subsec:no long pairings qm}. 
\end{example}

\subsection{Spectral Gaps in Stable Commutator Length} \label{subsec: gaps in scl}
We summarize some known methods and results on scl spectral gaps. 

\begin{definition}
	We say a group $G$ has a \emph{spectral gap} $C>0$ for elements (resp. integral chains) if 
	$\scl_G(c)\notin (0,C)$ for all elements (resp. integral chains) $c$ in $G$.
\end{definition}

The spectral gap property can be used to obstruct certain homomorphisms using monotonicity of scl (Proposition \ref{prop: mono and retract}).
A gap result for integral chains can also be used to estimate index of certain kinds of subgroups using the index formula (Proposition \ref{prop: index formula}).

There are two main approaches to prove spectral gap results in a group $G$ .

In light of Theorem \ref{thm:bavard} one approach is to construct for a given element $g$ (resp. chain $c$) a homogeneous quasimorphism $\phi_g$ (resp. $\phi_c$) of unit defect s.t.\ $\phi_g(g) \geq C$ (resp. $\phi_c(c) \geq C$) for a uniform $C > 0$. 
However, it is notoriously difficult to construct these maps which witness the optimal gap. For the free group, only two such constructions are available \cite{Heuer, CH:sclgap}.

The other approach is to give a uniform lower bound of the complexity of all admissible surfaces. This is usually done by first simplifying admissible surfaces (sometimes in the language of disk diagrams) into certain normal form and then making use of a particular structure of the normal form; See for instance \cite{DH91,Culler,Chen:sclfpgap,IK,RAAGgap2,CH:sclgap}.

Here we list some known spectral gap results for elements in Theorem \ref{thm: gap for elements summary} and for chains in Theorem \ref{thm: gap for chains summary}. The list is by no means extensive.

\begin{theorem}\label{thm: gap for elements summary}
	\leavevmode
	\begin{enumerate}
		\item (Calegari--Fujiwara \cite[Theorem A]{CF:sclhypgrp}) Any $\delta$-hyperbolic group with a generating set $S$ has a spectral gap $C=C(|S|,\delta)$ for elements. 
		Moreover, an element $g$ has $\scl_G(g)=0$ if and only if $g^n$ is conjugate to $g^{-n}$ for some $n\in\Z_+$.
		\item (Bestvina--Bromberg--Fujiwara \cite[Theorem B]{BBF}) Let $G$ be a finite index subgroup of the mapping class group $\mathrm{Mod}(\Sigma)$ of a possibly punctured closed orientable surface $\Sigma$. 
		Then $G$ has a spectral gap $C(G)$ for elements.
		\item (Chen--Heuer \cite[Theorem C]{CH:sclgap}) For any orientable $3$-manifold $M$, its fundamental group has a spectral gap $C(M)$ for elements.
		\item (Heuer \cite[Theorem 7.3]{Heuer}) Any (subgroup of a) RAAG has a spectral gap $1/2$ for elements. Moreover, any nontrivial element has positive scl. A new topological proof is given in \cite{CH:sclgap}. Weaker results are obtained in \cite{RAAGgap1} and \cite{RAAGgap2}.
		\item (Clay--Forester--Louwsma \cite[Theorem 6.9]{CFL16}) Let $\{G_v\}$ be a family of groups with a uniform gap for elements. Then their free product also has a spectral gap for elements.
		\item (Chen--Heuer \cite[Theorem F]{CH:sclgap}) Let $\{G_v\}$ be a family of groups without $2$-torsion such that they have a uniform gap for elements. Then their graph product also has a spectral gap for elements. The assumption on $2$-torsion is unnecessary by our Theorem \ref{thm: gap for graph products}.
	\end{enumerate}
\end{theorem}

\begin{theorem}\label{thm: gap for chains summary}
	\leavevmode
	\begin{enumerate}
		\item (Calegari--Fujiwara \cite[Theorem A']{CF:sclhypgrp}) Any $\delta$-hyperbolic group with generating set $S$ has a spectral gap $C=C(|S|,\delta)$ for integral chains. 
		Moreover, an integral chain has zero scl if and only if it is equivalent to the zero chain. 
		The following families of hyperbolic groups have uniform gaps even though the numbers of generators are unbounded.
		\item (Tao \cite[Theorem 1.1]{tao}) Any free group has a spectral gap $C=1/8$ for integral chains.
		\item (Chen--Heuer \cite[Proposition 9.1]{CH:sclgap}) Free products of cyclic groups have a spectral gap $C=1/12$ for integral chains. This is sharp for $\Z/2\star\Z/3$.
		\item (Chen--Heuer \cite[Theorem 9.5]{CH:sclgap}) There is a uniform constant $C>0$ such that 
		the orbifold fundamental group of any closed hyperbolic $2$-dimensional orbifold has a spectral gap $C$ for integral chains.
	\end{enumerate}
\end{theorem}

Note by Proposition \ref{prop:chains on direct products} that groups with spectral gaps for chains is closed under direct products. Corollary \ref{cor: gap pres by graph prod} generalizes this to graph products. 
The authors are unaware of any groups that were previously known to have a spectral gap for chains other than direct products of hyperbolic groups.

\subsection{Amalgamated free products}
Let $G = A \star_C B$ be the amalgamated free product of groups $A$ and $B$ over a subgroup $C$. For any $g \in G\setminus C$, we may write
\begin{equation}\label{eqn: ABalt}
	g = \wtt_1 \cdots \wtt_n
\end{equation}
where $\wtt_i \in A \setminus C$ or $\wtt_i \in B \setminus C$ for all $i \in \{1, \ldots, n \}$ such that the $\wtt_i$'s alternate between $A \setminus C$ and $B \setminus C$.

\begin{remark} \label{rmk:notation and vertex elements amalgamated free prod}
In an amalgamated free product $G = A \star_C B$ we use text font (e.g. $\att, \btt$) to denote elements of $A \setminus C$ or $B \setminus C$. We refer to those elements as \emph{vertex elements}.
Ordinary roman letters (e.g. $a,b$) denote generic elements in $G$.
\end{remark}

\begin{definition}[(cyclically) reduced form for amalgamated free products] \label{def:ab-alternating}
We say that for an element $g \in G \setminus C$
the expression (\ref{eqn: ABalt})
is the \emph{reduced form of $g$}. We define the \emph{length of $g$} as $n$ and denote it by $|g|$. 
Given the normal form (\ref{eqn: ABalt}) a 
\emph{prefix} of $g$ is an element $h \in G \setminus C$ with normal form $h
=\wtt_1 \cdots \wtt_m$ where $0\le m < n$.

If $\wtt_1$ and $\wtt_n$ as in the reduced form (\ref{eqn: ABalt}) lie in different sets $A \setminus C$ and $B \setminus C$ then we say that $g$ is \emph{cyclically reduced}.

For $x_1, \ldots, x_m \in G \setminus C$ 
we say that the expression 
$g= x_1 \cdots x_m$ is \emph{ a reduced decomposition of $g$} if there are reduced forms of each $x_i$ such that their concatenation is a reduced form of $g$.
Observe that $g$ is cyclically reduced if and only if the expression $g \cdot g$ is reduced. 
\end{definition} 

The reduced forms of an element are unique up to multiplication by $C$:
\begin{proposition}[Reduced form for amalgamated free products \cite{trees}] \label{prop:normal form amalg free prod}
Let $G = A \star_C B$ be an amalgamated free product and suppose that
$$
\wtt_1 \cdots \wtt_n = \wtt'_1 \cdots \wtt'_{n'}
$$
where all $\wtt$ terms alternate between $A \setminus C$ and $B \setminus C$.
Then $n = n'$ and there are elements $d_0, \ldots, d_n \in C$ with $d_0 = e = d_n$ such that $\wtt_i = d_{i-1} \wtt'_i d_i^{-1}$ for all $i \in \{1, \ldots, n \}$.
\end{proposition}

\begin{corollary} \label{corr:replacement}
Let $G = A \star_C B$ be an amalgamated free product. Suppose that
$$
x_1 \cdots x_n = x'_1 \cdots x'_n
$$
are two reduced decompositions such that $|x_i| = |x'_i|$ for all $i \in \{1, \ldots, n \}$. Then there are elements $d_0, \ldots, d_n \in C$ with $d_0 = e = d_n$ such that
$x_i = d_{i-1} x'_i d_i^{-1}$ for all $i \in \{1, \ldots, n \}$.
\end{corollary}

We will also need the following result later.

\begin{proposition} \label{prop:triangles amalg}
Let $g, h \in G$ be two elements. Then there are elements $y_1, y_2, y_3 \in G$ in reduced form and vertex elements 
(see Remark \ref{rmk:notation and vertex elements amalgamated free prod}) $\xtt_1, \xtt_2, \xtt_3 \in G$ such that
\begin{eqnarray*}
g &=& y_1^{-1} \xtt_1 y_2 \\
h &=& y_2^{-1} \xtt_2 y_3 \\
(gh)^{-1} &=& y_3^{-1} \xtt_3 y_1
\end{eqnarray*}
as reduced decompositions, where possibly some $y_i$ (resp. $\xtt_i$) is the identity represented by the empty word (resp. letter).
\end{proposition}

\begin{proof}
Let $g = \vtt_1 \ldots \vtt_m$ and $h = \wtt_1 \cdots \wtt_n$ be reduced. Let $0\le i \leq \min \{m, n \}$ be the largest integer such that $\vtt_{m - i} \cdots \vtt_m \wtt_1 \cdots \wtt_i = c \in C$.

Set $y_1^{-1} = \vtt_1 \cdots \vtt_{m - i - 2}$, $y_2 = \vtt_{m-i} \cdots \vtt_m$ and $y_3 = \wtt_{i+2} \cdots \wtt_n$ and $\xtt_1 = \vtt_{m-i-1}$, $\xtt_2 = \wtt_{i+1}$ and $\xtt_3 = (\vtt_{m-i-1} c \wtt_{i+1} )^{-1}$.
By the minimality of $i$ we see that $\xtt_3 \not \in C$ unless $i=\min(m,n)$, in which case $\xtt_3=id$ is represented by the empty word. Thus all of the expressions 
\begin{eqnarray*}
g &=& y_1^{-1} \xtt_1 y_2 \\
h &=& y_2^{-1} \xtt_2 y_3 \\
(gh)^{-1} &=& y_3^{-1} \xtt_3 y_1
\end{eqnarray*}
are reduced. 
\end{proof}

\subsection{HNN extensions}\label{subsec: HNN}

Suppose $C$ and $C'$ are subgroups of a group $A$ and $\phi: C\to C'$ is an isomorphism.
Let $G = A \star_C$ be the associated HNN extension, obtained as the quotient of $A\star\langle t\rangle$ by relations $tct^{-1}=\phi(c)$ for all $c\in C$.
For any $g\in G\setminus \{ C, C' \}$, we may write
\begin{equation}\label{eqn: reduced}
	g = \wtt_1 \cdots \wtt_n	
\end{equation}
where 
\begin{enumerate}
	\item for each $i \in \{1, \ldots, n \}$, $\wtt_i$ takes one of the following types: 
	\begin{itemize}
	\item $a\in A\setminus C$,
	\item $t^{-1}a't$ with $a'\in A\setminus C'$,
	\item $at$ with $a\in A$, or
	\item $t^{-1}a$ with $a\in A$;
	\end{itemize}
	We call such types \emph{vertex elements} and denote them with text font e.g. $\att, \btt$.
	\item the possible types of any pair $(\wtt_i,\wtt_{i+1})$ %(indices taken mod $n$) 
	are indicated by the oriented edges in Figure \ref{fig: HNN}.
\end{enumerate}
Note that for any $c_1,c_2\in C$, the word $c_1\wtt_i c_2$ can be rewritten into one of the same type as $\wtt_i$, for instance, $c_1 \cdot t^{-1}a't \cdot c_2 = t^{-1} a'' t$ with $a''=\phi(c_1)a' \phi(c_2)$.

\begin{figure} 
	\centering
	\labellist
	\small \hair 2pt
	%	\pinlabel $S'_v$ at 0 150
	%	\pinlabel $\beta_1$ at 120 250
	%	\pinlabel $\beta_2$ at 120 30
	\endlabellist
	\includegraphics[scale=0.6]{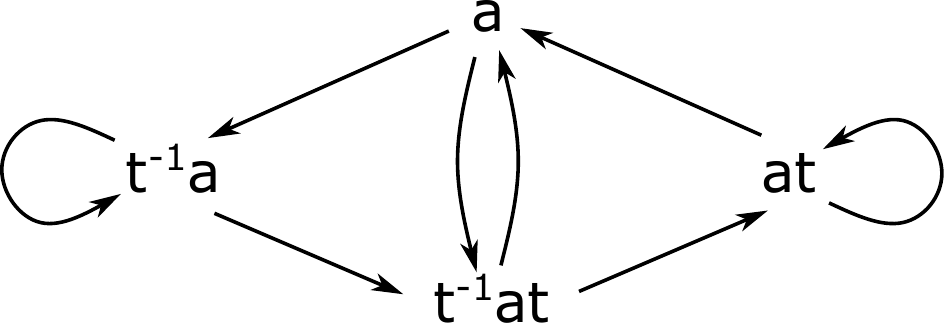}
	\caption{Possible concatenations of $\wtt_i$}
	\label{fig: HNN}
\end{figure}

\begin{definition}[(cyclically) reduced form for HNN extensions]\label{def:reduced form}
	We say an expression as in (\ref{eqn: reduced}) is a \emph{reduced form} of $g$. 
	Define the \emph{length of $g$} to be $n$ in (\ref{eqn: reduced}), denoted as $|g|$.
	Given the reduced form (\ref{eqn: reduced}) of $g$, a \emph{prefix of $g$} is some $h= \wtt_1 \cdots \wtt_m$ with $0\le m<n$.
	We say that $g$ is \emph{cyclically reduced} if the reduced expression $g =  \wtt_1 \cdots \wtt_n$ satisfies in addition that $(\wtt_n, \wtt_1)$ is as in Figure \ref{fig: HNN} and we call such an expression a \emph{cyclically reduced form}.
We say that $h$ is a cyclic conjugate of $g$ if the reduced form of $h$ is a cyclic permutation of the reduced form of $g$.

	For a reduced element $g$, we say an expression $g=x_1 \cdots x_m$ is a \emph{reduced decomposition} if there are reduced forms of each $x_i$ so that the concatenation is a reduced form of $g$.	
	Observe that $g$ is cyclically reduced if and only if $g \cdot g$ is a reduced decomposition.
\end{definition}

Reduced forms of a given element $g$ is essentially unique:
\begin{lemma}[Britton's lemma \cite{LyndonSchupp}] \label{lemma:normal form HNN}
	Let $G = A \star_C$ be an HNN extension and suppose that
	$$
	\wtt_1 \cdots \wtt_n = \wtt'_1 \cdots \wtt'_{n'}
	$$
	are two reduced forms of $g\in G\setminus C$.
	Then $n = n'$ and there are elements $d_0, \ldots, d_n \in C$ with $d_0 = e = d_n$ such that $\wtt_i = d_{i-1} \wtt'_i d_i^{-1}$ for all $i \in \{1, \ldots, n \}$.
\end{lemma}
From this we see that $|g|$ does not depend on the choice of reduced forms, and a reduced decomposition $g=x_1 \cdots x_m$ does not depend on the choice of reduced forms of $x_i$'s.

\begin{corollary} \label{corr:replacement for HNN}
	Let $G = A \star_C$ be an HNN extension. Suppose that
	$$
	x_1 \cdots x_n = x'_1 \cdots x'_n
	$$
	are two reduced decompositions such that $|x_i| = |x'_i|$ for all $i \in \{1, \ldots, n \}$. Then there are elements $d_0, \ldots, d_n \in C$ with $d_0 = e = d_n$ such that
	$x_i = d_{i-1} x'_i d_i^{-1}$ for all $i \in \{1, \ldots, n \}$.
\end{corollary}

\begin{proposition} \label{prop:triangles hnn}
Let $g, h \in G$ be two elements. Then there are elements $y_1, y_2, y_3 \in G$ and vertex elements $\xtt_1, \xtt_2, \xtt_3 \in G$ such that
\begin{eqnarray*}
g &=& y_1^{-1} \xtt_1 y_2 \\
h &=& y_2^{-1} \xtt_2 y_3 \\
(gh)^{-1} &=& y_3^{-1} \xtt_3 y_1
\end{eqnarray*}
as reduced expressions, where $y_i$ and $\xtt_i$ might be the identity.
\end{proposition}

\begin{proof}
The proof of Proposition \ref{prop:triangles amalg} works verbatim, interpreting vertex elements and reduced forms in the HNN extension context.
\end{proof}

\subsection{Graphs of Groups}
We briefly introduce graphs of groups to state the results of Sections \ref{sec:gaps from short overlaps} and \ref{sec:cm subgroups} more compactly. 
Graph of groups is a generalization of both amalgamated free products and HNN extensions discussed in the previous sections. We refer to \cite{trees} for details.

Let $\Gamma$ be an oriented connected graph with vertex set $V$ and edge set $E$. 
Each edge $e\in E$ is oriented with origin $o(e)$ and terminus $t(e)$. Denote the same edge with opposite orientation by $\bar{e}$, which provides an involution on $E$ satisfying $t(\bar{e})=o(e)$ and $o(\bar{e})=t(e)$. %We use $\{e,\bar{e}\}$ to represent unoriented edges.

A \emph{graph of groups} with underlying graph $\Gamma$ is a collection of \emph{vertex groups} $\{G_v\}_{v\in V}$, \emph{edge groups} $\{G_e\}_{e\in E}$, and injections $t_e: G_e\inj G_{t(e)}$, such that $G_e=G_{\bar{e}}$.
Fix a pointed $K(G_v,1)$ space $X_v$ for each $v$ and a pointed $K(G_e,1)$ space $X_e$ for each $e$. Each injection $t_e$ determines a map $X_e\to X_{t(e)}$, based on which we can form a mapping cylinder $M_{e,t(e)}$, where we think of $X_e$ and $X_{t(e)}$ as the subspaces on its boundary.
Glue all such mapping cylinders along their boundary by identifying $X_v$ in all $M_{e,v}$ (with $t(e)=v$) and identifying $X_e$ with $X_{\bar{e}}$ in $M_{e,t(e)}$ and $M_{\bar{e},t(\bar{e})}$.

We refer to the resulting space $X$ as the graph of spaces associated to the graph of groups, where the image of each $X_e$ is called an edge space.
The fundamental group $\pi_1(X)$ is called the \emph{fundamental group of the graph of groups}. When there is no danger of ambiguity, we will simply refer to $G$ as the graph of groups.

\begin{theorem}[\cite{trees}] \label{thm:graph of groups}
Every fundamental group of a graph of groups can be written as a sequence of amalgamated free products and HNN extensions over the edge groups. 
\end{theorem}

\section{Graph products of groups} \label{sec:graph product of groups} 
Graph products of groups generalize both right-angled Artin and right-angled Coxeter groups. They were introduced by Green in her thesis \cite{green}.
We go through some basic concepts and then establish the pure factor decomposition and the centralizer theorem (Theorem \ref{thm:centralizer}). 
We will need these results in Sections \ref{sec: gap in graph prod of groups} and \ref{sec:vertex chains}.

Let $\Gamma$ be a finite simplicial graph with vertex set $\Vrm(\Gamma)$ and edge set $\Erm(\Gamma)$ and let $\{ G_v \}_{v \in \Vrm(\Gamma)}$ be a family of groups. Then, the \emph{graph product} $\Gcl(\Gamma, \{ G_v \}_{v \in \Vrm(\Gamma)} )$ associated to this data is defined as the free product $\star_v G_v$ of the vertex groups subject to the relations $[g_v, g_w]$ for every $g_v \in G_v$, $g_w \in G_w$ with $(v,w) \in \Erm(\Gamma)$. If the family of groups $G_v$ is understood we will simply denote the group as $\Gcl(\Gamma)$.

When all $G_v=\Z$ (resp. $\Z/2$), we refer to $\Gcl(\Gamma)$ as the right-angled Artin (resp. Coxeter) group, denoted as $\Arm(\Gamma)$ (resp. $\Crm(\Gamma)$).

A normal form of elements is developed in \cite{green}.
Every element $g \in \Gcl(\Gamma)$ can be written as a product $g_1 \cdots g_n$ where each $g_i$ is in some vertex group. Following \cite[Definition 3.5]{green} we say that $n$ is the \emph{syllable length} in such an expression. There are three types of moves on the set of words representing the same element: 
\begin{itemize}
	\item (syllable shuffling) if there is a subsequence $g_i \cdots g_j$ with $1 \leq i < j \leq n$ and $g_j\in G_{v_j}$ such that every $g_k$ lies in a vertex group $G_{v_k}$ and $v_k$ is adjacent to $v_j$ for all $i<k<j$, then we can replace it by $g_i g_j g_{i+1}\cdots g_{j-1}$, and similarly if every $v_k$ is adjacent to $v_i$;
	\item (merging) if two consecutive letters $g_i,g_{i+1}$ lie in the same vertex group $G_v$, we can merge them into a single letter $g_i g_{i+1}\in G_v$;
	\item (deleting) if some $g_i=1$, then we can delete it.
\end{itemize}
Note that syllable shuffling preserves the syllable length while the other two moves reduce it.

We say that an expression $g_1 \cdots g_n$ is \emph{reduced} if 
\begin{itemize}
\item each $g_i$ is nontrivial, and
\item there is no subsequence $g_i \cdots g_j$ with $1 \leq i < j \leq n$ such that $g_i, g_j$ lie in the same vertex group $G_v$, and every $g_k$ lies in a vertex group $G_{v_k}$ with $v_k$ adjacent to $v$ for all $i<k<j$.
\end{itemize}

\begin{lemma}[\protect{\cite[Theorem 3.9]{green}}] \label{lemma: normal form for graph products}
Every element $g \in \Gcl(\Gamma)$ can be written as a reduced expression. This expression has minimal syllable length, and is unique up to syllable shufflings.
\end{lemma}

The minimal syllable length of words representing $g$ is denoted $|g|$, which is achieved by a word representing $g$ if and only if the the word is reduced.

Similarly, a word is (proper) \emph{cyclically reduced} if every cyclic permutation of its letters is reduced.
\begin{lemma}[\protect{Proof of \cite[Theorem 3.24]{green}}] \label{lemma: cyclically reduced form}
	Every conjugacy class in $\Gcl(\Gamma)$ contains an element represented by a cyclically reduced word. Any two cyclically reduced words in the same conjugacy class differ by a cyclic permutation of the letters and syllable shuffling.
\end{lemma}

Given a reduced expression $g=g_1\cdots g_n$, its \emph{support} is the induced subgraph consisting of vertices $v$ such that some $g_i$ lies in $G_v$. Since syllable shuffling does not change the support, by Lemma \ref{lemma: normal form for graph products}, the support does not depend on the choice of reduced expressions. We denote it by $\supp(g)$.

For an element $g \in \Gcl(\Gamma)$ some conjugate $\bar{g}=p^{-1}gp$ is represented by a cyclically reduced word. By Lemma \ref{lemma: cyclically reduced form}, the support $\supp(\bar{g})$ does not depend on the choice of $p$ and we set $\Theta(g) := \supp(\bar{g})$. %\ncomm{maybe $\Delta$ is a bit confusing given that $\Delta$ later will denote a specific graph (the opposite path)? }. 

The process of putting a word into a cyclically reduced word in the same conjugacy class does not enlarge the support (see the Proof of \cite[Theorem 3.24]{green}), 
thus $\Theta(g)$ is the smallest support of elements in the conjugacy class of $g$.
\begin{lemma}\label{lemma: smallest support among conjugates}
	For any $g\in \Gcl(\Gamma)$, we have $\Theta(g)\subset \supp(g)$.
\end{lemma}

Given two elements $g,h \in \Gcl(\Gamma)$, we can relate the normal form of $g \cdot h$ to the reduced expressions of $g,h$ as follows:
\begin{proposition} \label{prop:normal forms for products}
	For any elements $g,h\in \Gcl(\Gamma)$, there is a (possibly empty) clique $q=\{v_1,\ldots, v_k\}$ for some $k\ge0$ such that 
	we may write $g = g_0 q_g x$ and $h = x^{-1} q_h h_0$ as reduced expressions
	with $q_g = g_{1} \cdots g_{k}$ and $q_h = h_1 \cdots h_k$ with $g_i, h_i \in G_{v_i}$ and none of $g_i, h_i, g_i \cdot h_i$ is the identity for all $i \in \{ 1, \ldots, k \}$ such that 
	a reduced expression for $g h $ is given by 
	$$
	g \cdot h = g_0 \cdot q_{gh} \cdot h_0
	$$ where $q_{gh}$ is given by $q_{gh} = (g_{1} h_1) \cdots (g_{k} h_k)$.
\end{proposition}
\begin{proof}
	Given $g,h \in \Gcl(\Gamma)$ as in the proposition, choose $x$ to be a word with the maximal syllable length such that $g = g' x$ and $h = x^{-1} h'$ are reduced expressions for some words $g',h'$.
	Given $g'$ and $h'$, choose $q_g$ and $q_h$ to be words with maximal syllable length such that the support of $q_g$ and $q_h$ is a clique $q = \{v_1,\ldots, v_k\}$
	and we can write $g'=g_0 q_g$, $h'=q_h h_0$ as reduced expressions for some words $g_0,h_0$. 
	Define $q_{gh}$ as in the proposition. By the maximality of $x$, none of the terms in $q_g$ and $q_h$ cancel and thus the support of $q_{gh}$ is also equal to $q$.
	Note that for the expression
	$$
	g_0 \cdot q_{gh} \cdot h_0,
	$$
	$g_0\cdot q_{gh}$ is reduced since $g'=g_0\cdot q_g$ is reduced and has the same support. 
	Similarly, $q_{gh} \cdot h_0$ is reduced.
	Finally, one cannot shuffle a letter in $g_0$ to merge with another in $h_0$ since this would contradict the maximality of $g_q$ and $h_q$ by Lemma \ref{lemma: normal form for graph products}. 
	Thus $g_0 \cdot q_{gh} \cdot h_0$ is a reduced expression for $g \cdot h$.  
\end{proof}

Next we introduce the pure factor decomposition. For a graph $\Gamma$, the \emph{opposite graph} $\Gamma^\opp$ is the graph with the same vertices as $\Gamma$ and where two vertices are adjacent if and only if they are not adjacent in $\Gamma$.

Let $C^*_1, \ldots, C^*_{\ell^*}$ be the connected components of $\Theta(g)^\opp$ each consisting of a single vertex and let $C_1,\ldots,C_\ell$ be the connected components of $\Theta(g)^\opp$ with more than one vertices. Letters of $\bar{g}$ in different components can be shuffled across. By combining letters in the same component via shuffling, we can write $\bar{g}$ as $\gamma^*_1 \cdots \gamma^*_{\ell^*} \cdot g_1\cdots g_\ell$ with $\supp(\gamma^*_i) = C^*_i$ and $\supp(g_i)=C_i$. Then it is easy to see that every $g_i$ is cyclically reduced.

Now write $g_i=\gamma_i^{e_i}$ such that $e_i \in \Z_+$ and $\langle \gamma_i\rangle$ is maximal cyclic. Such an expression exists by \cite[Corollary 47]{centralgraphprod}. 
We get 
\begin{equation}\label{eqn: pure factor}
	g=p \cdot \gamma_1^* \cdots \gamma^*_{\ell^*} \cdot \gamma_1^{e_1}\cdots \gamma_\ell^{e_\ell} \cdot  p^{-1}
\end{equation}

\begin{definition}[Pure factor decomposition, pure factors, and primitive pure factors] \label{def: pure factor decomp}
	For any element $g$, an expression (\ref{eqn: pure factor}) is called a \emph{pure factor decomposition} of $g$, where each $\gamma_i^*$ and $g_i=\gamma_i^{e_i}$ is called a \emph{pure factor} of $g$.
	
	If $g=g^{e_1}$ with $e_1=1$ is its own pure factor decomposition and $|g| \geq 2$, then $g$ is called a \emph{primitive pure factor}.
\end{definition}

\begin{lemma} \label{lemma: pure factor unique up to cyclic conj and permutat}
	Each pure factor of $g$ is unique up to cyclic conjugation. The set of pure factors of $g$ up to cyclic conjugation is uniquely determined by $g$.
\end{lemma}
\begin{proof}
	This directly follows from Lemma \ref{lemma: cyclically reduced form} and the fact that letters in different pure factors commute with each other.
\end{proof}

\subsection{Centralizers in Graph Products}
The goal of this subsection is to describe the centralizer of any element $g$ in a graph product $\Gcl(\Gamma)$.

Recall that $\Theta(g)$ is the support of any cyclically reduced representative of $g$, that $C^*_1, \ldots, C^*_{\ell^*}$ denote the connected components of $\Theta(g)^\opp$ that each consists of a single vertex and that $C_1, \ldots, C_\ell$ denote the connected components of $\Theta(g)^\opp$ containing more than one vertex. Finally let $D(g)$ be the subset of $\Vrm(\Gamma)\setminus\Vrm(\Theta(g))$ consisting of vertices which are adjacent to every vertex of $\Theta(g)$.

The following result fully characterizes the centralizer of an element in terms of the pure factors.
\begin{theorem}[Centralizer Theorem]\label{thm:centralizer} 
	Let $g \in \Gcl(\Gamma)$ be an element with pure factor decomposition
	$$
	g = p \cdot \gamma_1^* \cdots \gamma_{\ell^*}^* \cdot \gamma_1^{e_1} \cdots \gamma_\ell^{e_\ell} \cdot p^{-1},
	$$
	where $\supp(\gamma_i^*)=C^*_i$, $\supp(\gamma_i)=C_i$ and let $D(g)$ be defined as above. Then an element $h \in \Gcl(\Gamma)$ commutes with $g$ if and only if
	$$
	h = p \cdot \zeta_1^* \cdots \zeta_{\ell^*}^* \cdot \gamma_1^{f_1} \cdots \gamma_\ell^{f_\ell} \cdot z \cdot  p^{-1},
	$$
	where $\zeta_i^*$ lies in the centralizer $Z_{G^*_{i}}(\gamma^*_i)$, %\jcomm{Changed from $Z(G^*_{i})$, same below. I guess this is what you meant, right?} 
	where $G^*_{i}$ is the vertex group of $C^*_i$, $f_i\in\Z$, and $\supp(z) \subset D(g)$.
\end{theorem}

This generalizes several similar results: In the case where $g$ is itself a single pure factor, this is proved by Barkauskas \cite[Theorem 53]{centralgraphprod}. In the case of right-angled Artin groups this has been done by Droms--Servatius--Servatius \cite{servatius1989surface} and in the case of graph products of abelian groups this has been done by Corredor--Gutierrez \cite[Centralizer Theorem]{centregraphprod}.

\begin{lemma}\label{lemma: centralizer, special case}
	If $g$ is cyclically reduced with pure factor decomposition
	$$
	g = \gamma_1^* \cdots \gamma_{\ell^*}^* \cdot \gamma_1^{e_1} \cdots \gamma_\ell^{e_\ell},
	$$
	where $\supp(\gamma_i^*)=C^*_i$, $\supp(\gamma_i)=C_i$ and the set $D(g)$ is defined as above.
	If $h \in \Gcl(\Gamma)$ commutes with $g$ and is supported on $\Theta(g)\cup D(g)$, %\ncomm{changed from $\Gamma(g)$ .. is this what you meant?}\jcomm{Yes, thanks!} 
	then
	$$
	h = \zeta_1^* \cdots \zeta_{\ell^*}^* \cdot \gamma_1^{f_1} \cdots \gamma_\ell^{f_\ell} z,
	$$
	where $\zeta_i^* \in Z_{G_{i}^*}(\gamma^*_i)$, where $G_{i}^*$ is the vertex group of $C^*_i$,
	$f_i\in\Z$, and $\supp(z) \subset D(g)$.
\end{lemma}
\begin{proof}
	Recall that every vertex in $D(g)$ is adjacent to all vertices in $\Theta(g)$ and that letters supported on different components of $\Theta(g)^\opp$ commute with each other. 
	So we can express $h$ as a reduced expression
	$h=h^*_1\cdots h^*_{\ell^*}h_1\cdots h_\ell z$, where each $h^*_i$ (resp. $h_i$) is a reduced word with $\supp(h^*_i)\subset C^*_i$ (resp. $\supp(h_i)\subset C_i$), and $z$ is a reduced word with $\supp(z)\subset D(g)$. Then 
	$$hgh^{-1}=\prod h^*_i\gamma_i^*(h^*_i)^{-1}\cdot \prod h_i\gamma_i^{e_i}h_i^{-1}.$$
	Since different factors have disjoint support, we observe that $hgh^{-1}=g$ if and only if
	$h^*_i\gamma_i^*(h^*_i)^{-1}=\gamma_i^*$ and $h_i\gamma_i^{e_i}h_i^{-1}=\gamma_i^{e_i}$ for each $i$.
	
	This reduces the problem to the case of a single pure factor. Hence by \cite[Theorem 53]{centralgraphprod}, we must have $h^*_i\in Z_{G^*_i}(\gamma^*_i)$ where $G^*_i$ is the group associated to the vertex $C^*_i$ and $h_i=\gamma_i^{f_i}$ for some $f_i\in \Z$.
\end{proof}

\begin{lemma}\label{lemma: noncommutative by support}
	Suppose $g$ and $h$ are reduced words where the last letter $h_v$ of $h$ lies in $G_v$ for some vertex $v\notin\supp(g)$ that is not adjacent to some $u\in\supp(g)$. 
	Then $g$ and $h$ do not commute.
\end{lemma}
\begin{proof}
	Express $\Gcl(\Gamma)$ as an amalgam $A\star_{C}B$, where $A=\Gcl(\st(u))$, $B=\Gcl(\Gamma\setminus\{u\})$, and $C=\Gcl(\Lk(u))$. 
	Here $\st(u)$ and $\Lk(u)$ denote the star and the link of $u$ in $\Gamma$ respectively. 
	For a reduced expression $g=g_1\cdots g_n$, we can pick out letters in $G_u$ to obtain $g_1\cdots g_n=b_0 g_{i_1} b_1\cdots b_{s-1} g_{i_s} b_s$, where
	each $g_{i_k}\in G_u$ and each $b_k$ is the product of letters outside $G_u$ sitting in between $g_{i_k}$ and $g_{i_{k+1}}$. 
	Note that $b_k\in B\setminus C$ for all $k\neq 0, s$ since we start with a reduced expression of $g$.
	If $b_0\in C$, then we can shuffle it across $a_1$.
	Thus we assume either $b_0\in B\setminus C$ or $b_0=id$.
	The same can be done for $b_s$, except for the case where $s=1$ and both $b_0,b_s\in C$, in which we may assume one of them to be the identity.
	
	In summary, this naturally expresses $g$ as a reduced word 
	$g=b_0a_1 b_1 \cdots b_{s-1} a_s b_s$
	in the amalgam $A\star_{C}B$, where $s\ge1$ and $a_k=g_{i_k}\in G_u\subset A\setminus C$ and $b_k\in B\setminus C$ for each $k$,
	except that possibly $b_0,b_s=id$, or one of them is the identity and the other lies in $C$ when $s=1$.
	%\jcomm{maybe put this part ahead as background.}%\#
	
	Similarly we have $h=\beta_0\alpha_1\beta_1\cdots\alpha_t\beta_t$ for some $t\ge0$, where each $\alpha_i\in A\setminus C$ and $\beta_i\in B\setminus C$ except possibly $\beta_0=id$.
	Note that we must have $\beta_t\in B\setminus C$ since $h_v$ is the last letter of $h$ as a reduced word in the graph product and $v\notin \st(u)$.
	
	As words in the amalgam, we have
	$$gh=b_0\cdots a_1\cdots a_s (b_s\beta_0)\alpha_1\cdots \alpha_t\beta_t,$$
	$$hg=\beta_0\alpha_1\cdots \alpha_t(\beta_t b_0)\cdots a_1\cdots a_s b_s.$$
	Since $\beta_t$ as a reduced word in the graph product contains $h_v$ and $v\notin\supp(g)$, while $\supp(b_0)\subset \supp(g)$, we know $\beta_t\cdot b_0\in B\setminus C$.
	Thus $hg$ is a reduced word in the amalgam except that possibly $\beta_0,b_s=id$.
	
	If $gh=hg$, when written as reduced words in the amalgam they must have the same length and start and end on elements in the same factor groups (i.e. $A$ and $B$).
	There are eight cases depending on whether $b_0,b_s,\beta_0\in C$, but there are only two cases where $gh$ and $hg$ can be written as reduced words of the same type and length:
	\begin{enumerate}
		\item $b_0=\beta_0=id$ and $b_s\notin C$, or
		\item $b_0,b_s,\beta_0\notin C$, where $b_s\beta_0\notin C$.
	\end{enumerate}
	In both cases, $hg$ ends with $b_s$ and $gh$ ends with $\beta_t$ (or $b_s\beta_0$ when $t=0$). If $gh=hg$, then we must have $\beta_t\in Cb_sC$ (or $\beta_0\in b_s^{-1}Cb_sC$ when $t=0$).
	Any element in $Cb_sC$ (or $b_s^{-1}Cb_sC$) as a word in the graph product is supported on $\supp(g)\cup \st(u)$, however $\beta_t$ contains $h_v$ and $v\notin \supp(g)\cup \st(u)$.
	This is a contradiction. Hence $gh\neq hg$.
\end{proof}

\begin{lemma}\label{lemma: centralizer support}
	Suppose $g$ is cyclically reduced. Let $D(g)$ be the set of vertices outside $\supp(g)$ and adjacent to all those in $\supp(g)=\Theta(g)$ as above. 
	If $h \in \Gcl(\Gamma)$ commutes with $g$, then $h$ is supported on $\Theta(g)\cup D(g)$.
\end{lemma}
\begin{proof}
	Write $h$ in a reduced expression. Denote $\supp(g)\cup D(g)$ by $\Delta$ and suppose $\supp(h)\not\subset \Delta$. 
	Let $h_v$ be the last letter in $h$ with the property that $h_v\in G_v$ for some $v\notin \Delta$. 
	Then $h_v$ cuts $h$ into a reduced expression $h_p h_v h_s$, where $\supp(h_s)\subset\Delta$. 
	As vertices in $D(g)$ are adjacent to all vertices in $\Theta(g)$, 
	by shuffling letters of $h_s$ in $D(g)$ to the end, we may represent $h_s=h'_s h_z$ as a reduced word so that $\supp(h'_s)\subset \Theta(g)$ and $\supp(h_z)\subset D(g)$.
	
	As $h_z$ commutes with $g$, we know $h'=h_p h_v h'_s$ also commutes with $g$, and $h_v$ is also the last letter in $h'$ supported outside $\Delta$.
	Then the conjugate $h'_sh_p h_v$ must commute with $g'=h'_s g (h'_s)^{-1}$. Note that $\supp(g')=\Theta(g)$ by Lemma \ref{lemma: smallest support among conjugates} since we know 
	$\supp(g')\subset\supp(g)=\Theta(g)$ as $\supp(h'_s)\subset \Theta(g)$. Applying Lemma \ref{lemma: noncommutative by support} to $h'_sh_ph_v$ and $g'$ we get a contradiction.
	Thus we must have $\supp(h) \subset \Delta$. 
\end{proof}

Now we prove Theorem  \ref{thm:centralizer}.
\begin{proof}[Proof of Theorem  \ref{thm:centralizer}]
	Since the centralizer of $p^{-1}gp$ is $p^{-1}Z_{\Gcl(\Gamma)}(g)p$, it suffices to prove the theorem assuming $g=\bar{g}$ is cyclically reduced, i.e. $p=id$. Then by Lemma \ref{lemma: centralizer support}, any $h$ commuting with $g$ must be supported in $\Theta(g)\cup D(g)$. Thus the result follows from Lemma \ref{lemma: centralizer, special case}.
\end{proof}

\begin{definition}[pure factor chain] \label{def:pure factor chain}
Suppose that  $g \in \Gcl(\Gamma)$ has an associated pure factor decomposition
$$
g=p \cdot \gamma_1^* \cdots \gamma^*_{\ell^*} \cdot \gamma_1^{e_1}\cdots \gamma_\ell^{e_\ell} \cdot  p^{-1}
$$
where $\gamma^*_i$ and $\gamma_i$ and $e_i$ are as in Equation (\ref{eqn: pure factor}).
Then we define the associated \emph{pure factor chain $g^{\pf}$ of $g$} initially as
$$
g^{\pf} = \gamma^*_1 + \cdots + \gamma^*_{\ell^*} + e_1 \gamma_1 + \cdots e_l \gamma_l,
$$
and then remove $e_i \gamma_i$ (resp. $\gamma^*_i$) if $\gamma_i$ (resp. $\gamma^*_i$) is conjugate to its inverse.

For an integral chain $c = \sum_{i=1}^n c_i$ we define the associated pure factor chain $c^{\pf}$ as follows: 
Set $c^1 = \sum_{i=1}^n c_i^{\pf}$. If there is a term $g_1^{-1}$ and $h_1 g_1 h_1^{-1}$ for some $g_1,h_1 \in \Gcl(\Gamma)$ in $c^1$, define $c^2$ as the chain $c^1$ without $g^{-1}$ and $h g h^{-1}$. 
If $c^i$ is defined but still has terms $g_i^{-1}$ and $h_i g_i h_i^{-1}$ for some $g_i,h_i \in \Gcl(\Gamma)$, define $c^{i+1}$ as $c^i$ without $g_i^{-1}$ and $h_i g_i h_i^{-1}$. 
Every such step reduces the number of terms by two, and thus, this process will eventually stop. We call the resulting chain the \emph{pure factor chain} $c^{\pf}$ associated to $c$.
Note that $c^{\pf}$ is equivalent to $c$. 
\end{definition}

By Lemma \ref{lemma: pure factor unique up to cyclic conj and permutat}, the pure factor chains for different pure factor decompositions are equivalent in the sense of Definition \ref{def: equivalent chains}. 
%\ncomm{this means unfortunately that pure factor chains are not unique .. }
%\jcomm{One might be able to get some uniqueness by minimizing the syllable length of $p$ in all possible pure factor decompositions. I think this is how things work in RAAGs.}

\begin{proposition} \label{prop: pure factor chain}
Let $c$ be an integral chain in $\Gcl(\Gamma)$ equivalent (Definition \ref{def: equivalent chains}) to a chain that consists of terms just supported on the vertex groups. 
Let $c^{\pf}$ be a pure factor chain. 
Then $c^{\pf}$ consists of terms which are just supported on vertex groups.
\end{proposition}
\begin{proof}
%Let $c$ be an integral chain in $\Gcl(\Gamma)$ and 
Let $\gamma \in \Gcl(\Gamma)$ be a primitive pure factor (Definition \ref{def: pure factor decomp}) so that $\gamma$ is not conjugate to $\gamma^{-1}$. 
For any element $g \in \Gcl(\Gamma)$ we define $\sigma_\gamma(g) = n$ if $\gamma^n$ up to cyclic conjugation is a pure factor of $g$ for some $n \in \Z\setminus\{0\}$. 
This is well defined by Lemma \ref{lemma: pure factor unique up to cyclic conj and permutat}.
The number $n$ is uniquely determined since $\gamma^n$ is cyclically reduced and has length $|n||\gamma|$, and $\gamma$ is not conjugate to $\gamma^{-1}$.
Set $\sigma_\gamma(g) = 0$ if no conjugate of $\gamma^n$ for any $n$ is a pure factor of $g$.

For a chain $c = \sum_{i \in I} \lambda_i c_i$ set $\sigma_\gamma(c) := \sum_{i \in I}\lambda_i \sigma_\gamma(c_i)$. 

\begin{claim}
If $c$ and $c'$ are equivalent chains (Definition \ref{def: equivalent chains}). Then $\sigma_\gamma(c) = \sigma_\gamma(c')$.
\end{claim}

\begin{proof}
It suffices to show that $\sigma_\gamma(c)=0$ for each basis element $c$ in $E(G)$ as in Definition \ref{def: equivalent chains}. Apparently $\sigma_\gamma(g)=\sigma_\gamma(pgp^{-1})$ since the pure factors of $g$ up to cyclic conjugation only depends on the conjugacy class of $g$. The fact that $\sigma_\gamma(g^n)=n\sigma_\gamma(g)$ for all $n\in\Z$ follows from the definition.
	
It remains to show that $\sigma_\gamma(x_1)+\sigma_\gamma(x_2)=\sigma_\gamma(x_1x_2)$ for two commuting elements $x_,x_2\in G$. 
If $\sigma_\gamma(x_1)=\sigma_\gamma(x_2) = \sigma_\gamma(x_1 \cdot x_2) = 0$ then the result trivially holds.

Without loss of generality assume that $\sigma_\gamma(x_1) \neq 0$. Let
$$
x_1 = p \cdot \gamma_1^* \cdots \gamma_{\ell^*}^* \cdot \gamma_1^{e_1} \cdots \gamma_\ell^{e_\ell} \cdot p^{-1}
$$
be the pure factor decomposition of $x_1$ with $\gamma_1 = \gamma$.  Then $x_2$ has to be of the form
$$
x_2 = p \cdot \zeta_1^* \cdots \zeta_{\ell^*}^* \cdot \gamma_1^{f_1} \cdots \gamma_\ell^{f_\ell} \cdot z \cdot  p^{-1}
$$
by Theorem \ref{thm:centralizer}. Note by the definition of $D(g)$ that $\supp(z)$ is disjoint from the support of any $\gamma^*_i$ and $\gamma_i$.
Thus $z$ does not contribute to $\sigma_\gamma(x_2)$ and hence $\sigma_\gamma(x_2)=f_1$.
For the same reason, we have $\sigma_\gamma(x_1x_2)=e_1+f_1$ from the expression
$$
x_1 x_2 = p \cdot (\gamma_1^* \zeta_1^* ) \cdots (\gamma_{\ell^*}^* \zeta_{\ell^*}^*) \cdot \gamma_1^{e_1+f_1} \cdots \gamma_\ell^{e_\ell + f_\ell} \cdot z \cdot  p^{-1}.
$$
Thus $\sigma_\gamma(x_1 x_2) = e_1 + f_1 = \sigma_\gamma(x_1) + \sigma_\gamma(x_2)$. This shows the claim.
\end{proof}

To conclude the proof of Proposition \ref{prop: pure factor chain}, Let $c$ be an integral chain which is equivalent to a chain $c'$ where every term is supported on a vertex. 
Let $c^{\pf}$ be a pure factor chain associated to $c$.
If $c^{\pf}$ has a term $\gamma$ which is not supported on vertices, then it is not conjugate to its inverse as such terms are removed in the beginning of the construction of $c^{\pf}$. This term gives us a primitive pure factor $\gamma$ such that $\sigma_\gamma(c^{\pf}) \neq 0$ since the number of terms in $c^{\pf}$ cannot be further reduced. 
On the other hand, we have $|\gamma| \geq 2$, since $\gamma$ is not supported in a vertex. Thus $\sigma_\gamma(c') = 0$. This contradicts the above claim since $c$ and $c'$ are equivalent.
\end{proof}

We show in Corollary \ref{cor: RACG gap for elements} that an element $g$ in a RACG has $\scl(g)=0$ if and only if $g$ is equivalent to the zero chain. Jing Tao asked us if this can be characterized more explicitly in the following form. We confirm this explicit characterization.

\begin{proposition}\label{prop: element zero chain in RACG}
	For a RACG $\Crm(\Gamma)$, an element $g$ is equivalent to the zero chain if and only if $g$ is conjugate to $g^{-1}$. Moreover, this is equivalent to $g=ab$ with $a^2=b^2=id$.
\end{proposition}
\begin{proof}
	In any group, if $g=ab$ with $a^2=b^2=id$, then $g^{-1}=ba$ is conjugate to $g=ab$. It is also clear that if $g$ is conjugate to $g^{-1}$ then $g$ is equivalent to the zero chain.
	
	Let $g$ be any element in a RACG $\Crm(\Gamma)$ with pure factorization 
	$$
	g=p \cdot \gamma_1^* \cdots \gamma^*_{\ell^*} \cdot \gamma_1^{e_1}\cdots \gamma_\ell^{e_\ell} \cdot  p^{-1}.
	$$
	As each $\gamma^*_{i}$ necessarily has order two as it lies in a vertex group $\Z/2$, the pure factor chain $g^\pf=\sum e_i \gamma_i$, where the summation runs over $i$ such that $\gamma_i$ that is not conjugate to $\gamma_i^{-1}$. By Proposition \ref{prop: pure factor chain}, if $g$ is equivalent to the zero chain, then $g^\pf$ is literally the zero chain, so $\gamma_i$ is conjugate to $\gamma_i^{-1}$ for all $1\le i\le \ell$. As $\gamma_i$'s and $\gamma^*_i$'s all commute, it follows that $g$ and $g^{-1}$ are conjugate. 
	
	Moreover, to see that $g=ab$ for some $a^2=b^2=id$, it suffices to show this for each $\gamma_i$ due to the commutativity. Each $\gamma_i$ is written as a cyclically reduced word $w$, and reversing the order of the letters gives a word $\bar{w}$ representing $\gamma_i^{-1}$, which must also be cyclically reduced. As these two cyclically reduced words represent the same conjugacy class since $\gamma_i$ is conjugate to $\gamma_i^{-1}$, by Lemma \ref{lemma: cyclically reduced form}, we know up to syllable shuffling $w$ and $\bar{w}$, they differ by a cyclic permutation. That is, there is a reduced expression $uv$ equivalent to $w$ so that $vu$ is equivalent to $\bar{w}$, where $u,v$ are reduced words. It has the property that $uvvu=\gamma_i\cdot \gamma_i^{-1}=id$. By the following claim (with $n=2$), we conclude that $u^2=v^2=id$ as desired.
	\begin{claim}
		For any $n\ge1$, suppose both $u_1u_2\cdots u_n$ and $u_n\cdots u_2 u_1$ are reduced expressions in $\Crm(\Gamma)$, where each $u_i$ is a reduced word. If $u_1u_2\cdots u_n\cdot u_n\cdots u_2u_1=id$, then $u_i^2=id$ for all $i$.
	\end{claim}
	\begin{proof}
		We proceed by induction on the total length $\sum_i|u_i|$ of the word $u_1\cdots u_n$. The result is immediate if the total length is $1$. Suppose the result holds when the total length is at most $L-1$ for $L\ge2$, and consider such an expression with total length $L$. The expression $u_1u_2\cdots u_n\cdot u_n\cdots u_2u_1$ must be reducible by assumption. 
		Each $u_i$ appears twice in the expression, we distinguish them by denoting the copy on the right as $u'_i$ to avoid confusion. 
		
		As $u_1u_2\cdots u_n\cdot u'_n\cdots u'_2u'_1$ is the product of two reduced expressions, it must be the case that some letter $g_v$ in some $u_i$ can be shuffled all the way to merge with a letter $g'_v$ in some $u'_j$, where $v$ is a vertex in $\Gamma$, and $1\le i,j\le n$. We necessarily have $g_v=g'_v$ as the vertex group $G_v=\Z/2$.
		
		We first show $i=j$. If $i<j$, then $u_j$ sits in between $u_i$ and $u'_j$, so the $g_v$ in $u_i$ can be shuffled across $u_j$ which also contains a copy of $g_v$, contradicting that $u_1\cdots u_n$ is reduced. If $i>j$, then the $g_v$ in $u'_j$ can be shuffled across $u'_i$ which contains a copy of $g_v$, contradicting that $u'_n\cdots u'_1$ is reduced.
		
		Now given $i=j$, suppose $u_i=g_1\cdots g_k$ as a reduced word, where $g_i$ is the generator of the vertex group $G_{v_i}$ of some vertex $v_i$. Then for some $1\le s,t\le k$, the letter $g_s=g_v$ in $u_i$ can be shuffled across $(g_{s+1}\cdots g_k)u_{i+1}\cdots u_n\cdot u'_n\cdots u'_{i+1} (g_1\cdots g_{t-1})$ to cancel with $g_t=g'_v$ inside $u'_i$, where $v=v_s=v_t$. We must have $s\ge t$ as otherwise $g_s$ can be shuffled across $g_{s+1}\cdots g_{t-1}$ inside $u_i$ to cancel $g_t$, contradicting that $u_i$ is reduced.
		
		First consider the case $s>t$. Then $g_s$ can be shuffled to the end of $u_i$ and $g_t$ can be shuffled to the head of $u'_i=u_i$, so $u_i$ is equivalent to $xv_ix$ as a reduced expression for some reduced word $v_i$, where $x=g_s=g_t$. It follows that 
		$$u_1\cdots u_n=u_1\cdots u_{i-1} xv_ix u_{i+1}\cdots u_n=u_1\cdots u_{i-1}xv_i u_{i+1}\cdots u_n x$$
		as words equivalent up to syllable shuffling. It follows that the subword $u_1\cdots u_{i-1}xv_i u_{i+1}\cdots u_n $ of total length $L-1$ is reduced, as part of the last reduced expression.
		Similarly, 
		$$u'_n\cdots u'_1=u'_n\cdots u'_{i+1} xv'_i x u'_{i-1}\cdots u'_1=xu'_n\cdots u'_{i+1} v'_i x u'_{i-1}\cdots u'_1$$
		as equivalent reduced words, and $u'_n\cdots u'_{i+1} v'_i x u'_{i-1}\cdots u'_1$ is reduced, where $v'_i=v_i$.
		Hence by the induction hypothesis, as
		\begin{align*}
			(u_1\cdots u_{i-1}xv_i u_{i+1}\cdots u_n)(u'_n\cdots u'_{i+1} v'_i x u'_{i-1}\cdots u'_1)&=(u_1\cdots u_{i-1}xv_i u_{i+1}\cdots u_n x)(xu'_n\cdots u'_{i+1} v'_i x u'_{i-1}\cdots u'_1)\\
			&=u_1u_2\cdots u_n\cdot u_n\cdots u_2u_1=id,
		\end{align*}
		where we think of $x$ and $v_i$ both as reduced words in the reduced expression, we must have $u_j^2=id$ for all $j\neq i$ and $v_i^2=id$, which implies $u_i^2=xv_i^2 x=x^2=id$.
		
		Now consider the remaining case $s=t$. In this case $g_s=g_t$ commutes with all remaining letters in $u_i$ as well as those in $u_j$ for $j>i$ by the same analysis as above. So we can write $u_i$ equivalently as $v_ix$ and $xv_i$, which are reduced expressions where $x=g_s$ commutes with $v_i$. The by the same argument, $u_1\cdots u_{i-1}v_i u_{i+1}\cdots u_n$ is a reduced word of total length $L-1$ as a subword of $u_1\cdots u_{i-1}v_i u_{i+1}\cdots u_nx$ and similarly $u'_n\cdots u'_{i+1}v_i u_{i-1}\cdots u_1$ is reduced. As the product of them is the identity, the induction hypothesis implies $u_j^2=id$ for all $j\neq i$ and $v_i^2=id$, which implies $u_i^2=v_i^2 x^2=x^2=id$.
	\end{proof}
	
\end{proof}

\section{Gaps from short overlaps} \label{sec:gaps from short overlaps}
Let $G$ be a group splitting over a subgroup $C$, that is, $G$ is either an amalgam $A\star_{C} B$ or an HNN extension $A\star_C$. 
In either case, $G$ is a graph of groups with a unique edge group $C$, realized as a graph of spaces $X$ with a single edge space. 

Consider an integral chain $d=\sum g(i)$, where each $g(i)=\wtt_1(i)\cdots \wtt_{L_i}(i)$ is a cyclically reduced word and does not lie in the vertex groups. For any integral chain $d+d'$, where $d'$ is a sum of elements in vertex groups, any admissible surface $S$ of degree $n$ for $d+d'$ can be considered as an admissible surface for $d$ of the same degree with extra boundary components representing curves in vertex groups. This is called an admissible surface for $d$ \emph{relative to the vertex groups}.

Then $S$ can be simplified into the \emph{simple normal form} in the sense of \cite[Section 3.2]{Chen:sclBS}, which does not increase $-\chi^-(S)$ and does not change the degree. This means that $S$ is obtained by gluing \emph{pieces} together, where each piece is a polygon possibly containing a hole in the interior, with $2k$ sides alternating between \emph{arcs} and \emph{turns} for some $k\in\Z_+$; see Figure \ref{fig: simple}. Topologically, each piece is either a disk or an annulus.
Turns are places that these pieces glue along, and arcs are part of $\partial S$. They carry labels that we describe as follows.

In the case of an amalgam, each piece is either supported in $A$ or $B$. If a piece is supported in $A$, then each arc is labeled by some $\wtt_i(k)\in A\setminus C$, and each turn is labeled by some element $c\in C$, which we refer to as the winding number of the turn. The product of labels on the polygonal boundary of each piece supported in $A$ (resp. $B$) defines a conjugacy class in $A$ (resp. $B$), which is $id$ if and only if the piece is a disk (i.e. has no hole inside).

In the case of an HNN extension, each piece is supported in the vertex group $A$. Each arc is labeled by some $\wtt_i(k)\in A\setminus C$, and each turn is labeled by some element $c\in C$, the winding number of the turn. Recall from Section \ref{subsec: HNN} that each $\wtt_i(k)$ falls into one of four types. If a turn travels from some $\wtt_i(k)$ to $\wtt_j(\ell)$, then the possible types of $(\wtt_i(k),\wtt_j(\ell))$ are
$$(at\text{ or } a, t^{-1}a \text{ or } a)\quad \text{and} \quad (t^{-1}a \text{ or } t^{-1}at, at \text{ or } t^{-1}at).$$
It follows that the product of labels on the polygonal boundary defines a conjugacy class in $A$. The conjugacy class is $id$ if and only if the piece is a disk.

In both cases, each disk piece has at least two turns since each $\wtt_i(k)\notin C$.

\begin{figure} 
	\centering
	\labellist
	\small \hair 2pt
	\pinlabel $\wtt_{i_1}(k_1)$ at -30 55
	\pinlabel $c_1$ at 100 55
	
	\pinlabel $\wtt_{i_2}(k_2)$ at 210 75
	\pinlabel $c_2$ at 140 55
	\pinlabel $\wtt_{i_3}(k_3)$ at 210 35
	\pinlabel $c_3$ at 270 55
	
	\pinlabel $\wtt_{i_4}(k_4)$ at 430 30
	\pinlabel $c_6$ at 375 -10
	\pinlabel $\wtt_{i_6}(k_6)$ at 325 30
	\pinlabel $c_5$ at 310 95
	\pinlabel $\wtt_{i_5}(k_5)$ at 375 110
	\pinlabel $c_4$ at 440 95
	\endlabellist
	\includegraphics[scale=0.5]{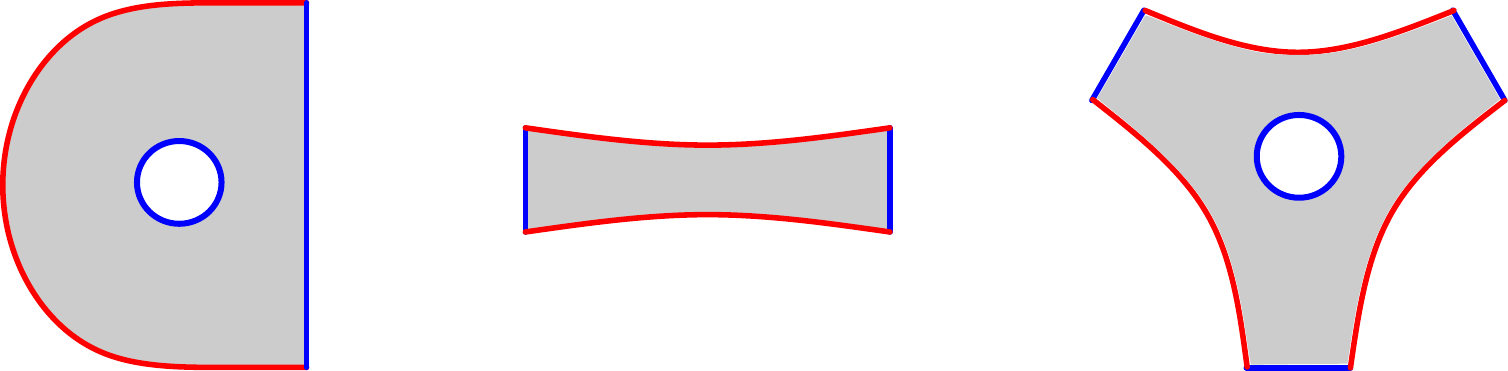}
	\caption{Pieces with or without a hole in the interior}
	\label{fig: simple}
\end{figure}
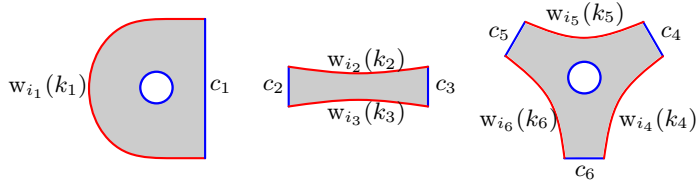

Pieces are glued together along \emph{paired turns} to form $S$. Here a turn from $\wtt_i(k)$ to $\wtt_j(\ell)$ with winding number $c\in C$ is uniquely paired with a turn from $\wtt_{j-1}(\ell)$ to $\wtt_{i+1}(k)$ with winding number $c^{-1}$. The gluing guarantees that each boundary component of $S$ is labeled by a conjugate of $g(i)^k$ for some $k\in\Z_+$.
The way we glue pieces together is encoded by the \emph{gluing graph} $\Gamma_S$, 
where each vertex corresponds to a piece and each edge corresponds to a gluing of two paired turns. 
For each vertex $v$, let $d(v)$ be its valence in $\Gamma_S$, and let $\delta(v)=1$ if the corresponding piece is a disk and $\delta(v)=0$ otherwise (i.e. for an annulus piece).
%Then the vertex set $V$ of $\Gamma_S$ decomposes into $V_A\sqcup V_D$ corresponding to annuli pieces and disk pieces.

If we cap off the hole in each annulus piece in $S$, then the surface deformation retracts to the graph $\Gamma_S$. Recall that the Euler characteristic $\chi(\Gamma_S)$ can be computed as $\sum_v [1-d(v)/2]$, so we have 
\begin{equation}\label{eqn: comput chi}
	-\chi(S)=-\chi(\Gamma_S)+|V_A|=\sum_v [d(v)/2-\delta(v)],
\end{equation}
where $|V_A|$ is the number of annuli pieces.
Note that $d(v)/2-\delta(v)\ge 0$, and the equality holds if and only if $v$ is a disk piece with two turns (i.e. $v$ has valence $2$ in $\Gamma_S$).

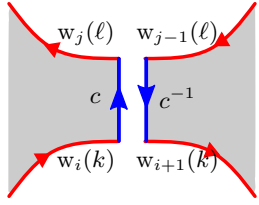
\begin{figure}
	\labellist
	\small \hair 2pt
	
	\pinlabel $\wtt_i(k)$ at 43 20
	\pinlabel $\wtt_j(\ell)$ at 43 90
	\pinlabel $\wtt_{j-1}(\ell)$ at 93 90
	\pinlabel $\wtt_{i+1}(k)$ at 93 20
	\pinlabel $c$ at 48 55
	\pinlabel $c^{-1}$ at 93 55
	%\pinlabel $S_v$ at 118 75
	%\pinlabel $S_u$ at 238 75
	
	\endlabellist
	\centering
	\includegraphics[scale=0.7]{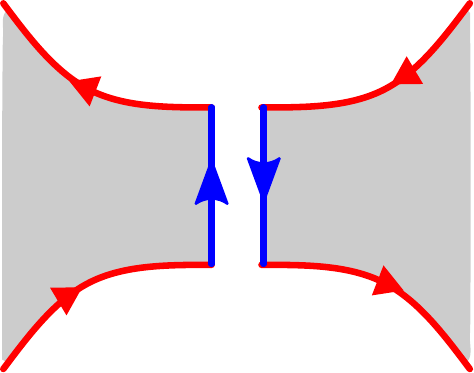}
	\caption{Two paired turns}\label{fig: pair}
\end{figure}
\begin{theorem}\label{thm: no long pairing}
Suppose $G$ is a group that splits over a subgroup $C$. Let $c = \sum_{i=1}^n g(i)$
be an integral chain in $G$ where each term either lies in a vertex group or is cyclically reduced.

Fix an integer $N \in \Z_+$. Then either
$$
\scl_G(c) \geq \frac{1}{12 N}
$$
or for any cyclically reduced $g=g(i)$, $i \in \{1, \ldots, n \}$, we have
$$
g^N =  h^k h' d,
$$
where %this expression is reduced, 
\begin{itemize}
\item $h$ is a cyclically reduced word conjugate to $g(j)^{-1}$ for some $j \in \{1, \ldots, n \}$, and $k\in\Z_{\ge0}$,
\item $h'$ is a prefix of $h$ and 
\item $d \in C$.
\end{itemize}
\end{theorem}
\begin{proof}
	Without loss of generality, assume $g(1)$ is cyclically reduced and no such equations hold for $g=g(1)$. We show $\scl_G(c) \geq \frac{1}{12 N}$.
	
	Start with any admissible surface $S$ for $c$ without sphere or disk components. 
	For any large integer $M$, there is a finite normal cover $\tilde{S}$ of $S$ where each component of $\partial \tilde{S}$ covers some component of $\partial S$ with degree greater than $M$. 
	In particular, this shows that, up to taking finite covers, any boundary component of $S$ winding around $g(1)$ represents $g(1)^{qN+r}$ for some $q,r\in\Z_+$ where the remainder $r$ is negligible compared to $q$. 
	Thus in the following estimate, we will assume for simplicity that whenever a boundary component of $S$ winds around $g(1)$, it actually winds around $g(1)$ some $N$-multiple of times.
	
	Remove elements in $c$ supported on vertex groups to obtain an integral chain $c_0$.
	Then as explained above, we can think of $S$ as an admissible surface for $c_0$ relative to the vertex groups.
	Up to homotopy and compression, we can put $S$ into the simple normal form, which does not affect the boundary; see \cite[Lemma 3.7]{Chen:sclBS}. 
	For each boundary component representing $g(1)^{qN}$, cut it into $q$ segments, so that each segment is labeled by the cyclically reduced word representing $g(1)^N$. 
	Each segment consists of $L_1\cdot N$ distinct arcs (some with same labels) and thus witnesses $L_1\cdot N$ pieces, 
	some of which might be counted multiple times (since some arcs might lie on the same piece). 
	
	We claim that at least one of these pieces witnessed along a segment is represented by a vertex $v$ in the gluing graph $\Gamma_S$ such that $d(v)/2-\delta(v)>0$. 
	If not, then each such a piece is a disk with two turns. 
	Such rectangles glue to a long strip (see Figure \ref{fig: strip}), whose boundary shows that $g(1)^N c_1 w c_2=id$ for some $c_1,c_2\in C$, 
	where $w$ is the word on the opposite side of $g(1)^N$ and must be a reduced subword of some $g(j)^m$ ($g(j)$ represents the loop that the boundary component on the opposite side of the strip maps onto). 
	In algebraic terms, this implies an equation that should not exist by our assumption.
	
	Therefore, for each segment $\sigma$ as above, we can choose a piece $v(\sigma)$ witnessed by $\sigma$ so that $d(v)/2-\delta(v)>0$. It is possible that $v(\sigma)=v(\sigma')$ for distinct segments $\sigma,\sigma'$. Thinking of such pieces as vertices on $\Gamma_S$, each $v=v(\sigma)$ either has $d(v)\ge 3$ or has $d(v)\le2$ and $\delta(v)=0$. In the former case, such a vertex is witnessed by at most $d(v)$ segments, and hence each segment witnessing $v$ contributes at least $\frac{1}{d(v)}[d(v)/2-\delta(v)]\ge \frac{d(v)-2}{2d(v)}\ge 1/6$ to the right-hand side of equation (\ref{eqn: comput chi}). In the latter case, such a vertex has $\delta(v)=0$, and hence each segment witnessing $v$ contributes at least $\frac{1}{d(v)}[d(v)/2-\delta(v)]= 1/2$ to the right-hand side of (\ref{eqn: comput chi}). Thus in any case, each segment contributes at least $1/6$ to $-\delta(S)$, and the total number of such segments is $n/N$, where $n$ is the degree of $S$.
	
	\begin{figure} 
		\centering
		\labellist
		\small \hair 2pt
		\pinlabel $g(1)^N$ at 220 -10
		\pinlabel $w$ at 225 63
		
		\pinlabel $c_2$ at -5 40
		\pinlabel $\wtt_1(1)$ at 25 20
		\pinlabel $\wtt_2(1)$ at 70 20
		\pinlabel $\cdots$ at 110 40
		
		\pinlabel $\wtt_{L_1-1}(1)$ at 145 20
		\pinlabel $\wtt_{L_1}(1)$ at 195 20
		\pinlabel $\wtt_1(1)$ at 235 20
		\pinlabel $\wtt_2(1)$ at 280 20
		\pinlabel $\cdots$ at 320 40
		
		\pinlabel $\wtt_{L_1-1}(1)$ at 355 20
		\pinlabel $\wtt_{L_1}(1)$ at 400 20
		\pinlabel $c_1$ at 435 40
		\endlabellist
		\includegraphics[scale=0.7]{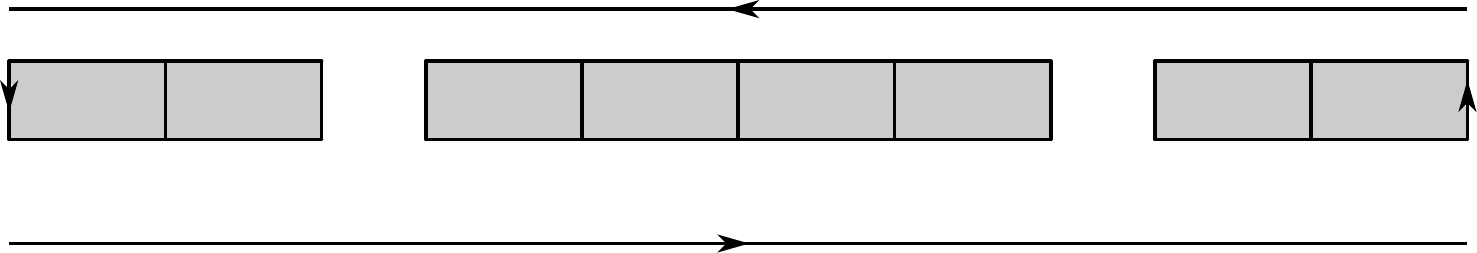}
		\caption{Rectangular pieces glue to a long strip.}
		\label{fig: strip}
	\end{figure}
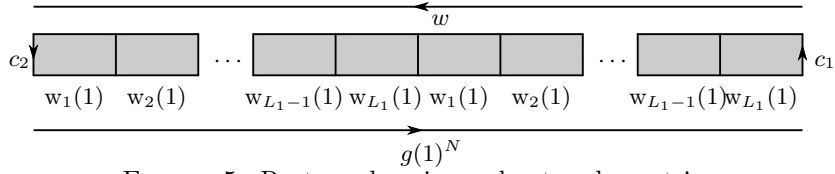
	
	Hence we obtain
	$$\frac{-\chi(S)}{2n}\ge \frac{1}{6}\cdot \frac{n}{N} \cdot \frac{1}{2n}=\frac{1}{12N}.$$
	Since $S$ is arbitrary, this gives the desired estimate.
	
	%The assumption implies that if we read along a boundary component of $\partial S$ representing a power of $g$, then for every $N$ copies of $g$, we should observe either a vertex $v$ of degree $d\ge3$ in the graph $\Gamma_S$ encoding the gluing of pieces in $S$, or a piece that is not a disk (and the corresponding vertex has degree at most $2$). In the former case, such a vertex contributes at least $(d-2)/2$ to $-\chi(S)$, and the contribution can be counted at most $d$ times. Note that $(d-2)/2d\ge 1/6$ achieved at $d=3$, thus the net contribution (without being counted multiple times) to $-\chi(S)$ is at least $1/6$. In the latter case, such a piece contributes at least $1$ to $-\chi(S)$, possibly counted twice. Hence the net contribution is at least $1/2$. In summary, on average every copy of $g$ on $\partial S$ contributes at least $1/(6N)$ to $-\chi(S)$. This shows $-\chi(S)/2n\ge 1/(12 N)$.
\end{proof}

\subsection{Proof of Theorem \ref{thm: no long pairing} using quasimorphisms} \label{subsec:no long pairings qm}
In this section we will give an alternative proof to Theorem \ref{thm: no long pairing} using explicit quasimorphisms. 
The quasimorphisms will be similar to the counting quasimorphisms discovered by Brooks \cite{brooks}; see also Example \ref{exmp:brooks}. For amalgamated free products, this is also similar to \cite{CFL16}.

Let $G$ be an amalgamated free product or HNN extension which splits over a group $C$. Let $w \in G$ be a cyclically reduced element. Then we define $\nu_{w} \col G \to \N$ as follows. For any $g \in G$ let $\nu_w(g)$ be the largest integer $n$ such that $g$ has reduced decomposition
$$
g = g_0 w_1 g_1 \cdots w_n g_n,
$$
where $g_i \in G$ is possibly the empty word and $w_i \in C w C$. 
We define
$$
\phi_w = \nu_w - \nu_{w^{-1}}.
$$

\begin{proposition}
The map $\phi_w \col G \to \R$ is a quasimorphism with defect $D(\phi_w) \leq 3$. 
\end{proposition}

\begin{proof}
We need the following claim for the proof.

\begin{claim}\label{claim:counting subwords}
Let $y \xtt y'$ be a reduced expression where $\xtt$ is a vertex element. Then
$$
\nu_w(y \xtt y') - \nu_w(y) - \nu_w(y') \in \{ 0, 1 \}.
$$
\end{claim}
\begin{proof}
Suppose $\nu_w(y) = n$ with $y = y_0 w_1 y_1 \cdots w_n y_n$ and $\nu_w(y')=n'$ with $y' = y'_0 w'_1 y'_1 \cdots w'_{n'} y'_{n'}$, where $w_i,w'_i \in C w C$.
Then
$$
y \xtt y' = y_0 w_1 y_1 \cdots w_n (y_n \xtt y'_0) w'_1 y'_1 \cdots w'_{n'} y'_{n'}
$$
is a reduced decomposition and thus $\nu_w (y \xtt y' ) \geq \nu_w(y) + \nu_w(y')$.

On the other hand, suppose that $\nu_w(y \xtt y' ) = m$ and we have a reduced decomposition of $y \xtt y'$ that contains $m$ disjoint copies of words in $CwC$. We also have a reduced expression of $y \xtt y'$ induced from arbitrary reduced words representing $y$ and $y'$. By chopping up the second reduced expression so that subwords have lengths matching the first reduced decomposition, it follows from Corollaries \ref{corr:replacement} and \ref{corr:replacement for HNN} that all the subwords in the first expression representing elements in $CwC$ give disjoint subwords of $y$ or $y'$, except when the subword intersects $\xtt$, which can occur for at most one subword. Thus
$$
\nu_w (y \xtt y' ) \leq \nu_w(y) + \nu_w(y')+1,
$$
which shows the claim.
\end{proof}

Let $g,h \in G$. Using Propositions \ref{prop:triangles amalg} and \ref{prop:triangles hnn} we see that there are elements $y_1,y_2,y_3 \in G$ and vertex elements $\xtt_1,\xtt_2,\xtt_3$ such that
\begin{eqnarray*}
g &=& y_1^{-1} \xtt_1 y_2 \\
h &=& y_2^{-1} \xtt_2 y_3 \\
(gh)^{-1} &=& y_3^{-1} \xtt_3 y_1
\end{eqnarray*}
as reduced expressions.

Using Claim \ref{claim:counting subwords} we see that
\begin{eqnarray*}
|\phi_w(g)-\nu_w(y_1^{-1}) - \nu_w(y_2) + \nu_{w^{-1}}(y_1^{-1}) + \nu_{w^{-1}}(y_2) | & \leq & 1 \\
|\phi_w(h)-\nu_w(y_2^{-1}) - \nu_w(y_3)  + \nu_{w^{-1}}(y_2^{-1}) + \nu_{w^{-1}}(y_3)| & \leq & 1 \\
|\phi_w(gh)-\nu_w(y_1^{-1}) - \nu_w(y_3) +  \nu_{w^{-1}}(y_1^{-1}) + \nu_{w^{-1}}(y_3)| & \leq & 1.
\end{eqnarray*}
Using that $\nu_{w^{-1}}(g) = \nu_w(g^{-1})$ for any $g$ we obtain
$$
|\phi_w(g) + \phi_w(h) - \phi_w(gh)| \leq 3,
$$
which shows the proposition.
\end{proof}

We can now prove Theorem \ref{thm: no long pairing} using quasimorphisms:

\begin{proof}[Proof of Theorem \ref{thm: no long pairing}]
Let $G$ be a group which splits over $C$ and let $\sum_{i=1}^n g(1)$ be some integral chain where every term either lies in a vertex group or is cyclically reduced and let $N \in \Z_+$ be some integer.
Suppose for some $g=g(i)$ cyclically reduced, the equation $g^N=h^kh'd$ as in the theorem does not hold. Then for $w=g^N$, we know $\nu_{w^{-1}}(g(j)^m)=0$ for all $j$ and $m\in\Z_+$. It follows that
$$
\phi_w(g(j)^m) \geq 0 
$$
for any $j \in \{1, \ldots, n\}$.
Moreover,
$$
\phi_{w}(g(i)^m) \ge \left\lfloor \frac{m}{N} \right\rfloor
$$
and thus $\bar{\phi}_w(g(i)) \ge \frac{1}{N}$ for the homogenization. 

We conclude that
$$
\sum_{j=1}^n \bar{\phi}_w(g(j)) \geq \frac{1}{N}.
$$
On the other hand, we have that $D(\phi_w) \leq 3$ and thus $D(\bar{\phi}_w) \leq 6$ by Proposition \ref{prop:homog rep}. By Bavard's Duality Theorem (Theorem \ref{thm:bavard}) we obtain
$$
\scl_G(\sum_{i=1}^n g(i)) \geq \frac{1}{12 N},
$$
which completes the proof of Theorem \ref{thm: no long pairing}.
\end{proof}

\section{Central/Malnormal Subgroups} \label{sec:cm subgroups}

In this section we will use Theorem \ref{thm: no long pairing} to give a criterion for chains in certain amalgamated free products and HNN extensions to have a gap in stable commutator length. 

In order to apply Theorem \ref{thm: no long pairing} we need to solve the following equation
for some fixed integer $N \in \N$
\begin{eqnarray} \label{equn: main}
g^N = h^k h' c,
\end{eqnarray}
where both sides are reduced decompositions (see Definitions \ref{def:ab-alternating} and \ref{def:reduced form}),  where $|g| \geq |h|$, $h'$ is a prefix of $h$, $c \in C$ and $k \geq N$. 

In order to solve Equation (\ref{equn: main}), we define and study BCMS-$D$ subgroups $H$ of a group $G$ for an integer $D$ (Definition \ref{def:bcmsm}).
Central subgroups are BCMS-$0$ and malnormal subgroups are BCMS-$1$. 
As a key example, if $\Lambda \subset \Gamma$ is an induced subgraph of a graph $\Gamma$ then 
the associated subgroup $\Arm(\Lambda)$ of the RAAG $\Arm(\Gamma)$ is a BCMS-$D$ subgroup for some $D$; see Lemma \ref{lemma: Lambda has property CM}.

If the subgroup $C$ that $G$ splits over is BCMS-$D$, then for $N = D+2$ we solve Equation (\ref{equn: main}) as follows:
\begin{itemize}
\item if $|g|=|h|$ then Equation (\ref{equn: main}) reduces to $g^N = h^N c$ for some $c \in C$. We show that there is some element $z \in C$ which commutes with $g$ such that $g = h z$, so that $c = z^N$; see Proposition \ref{prop:gmhm}.
\item if $|g| > |h|$ then Equation (\ref{equn: main}) implies that there is some element $x \in G$ and an element $c \in C$ which commutes with $x$ such that $g = x^m c$ for some $m \geq 2$; see Proposition \ref{prop:sol to main equn}.
\end{itemize}

In both cases, equation (\ref{equn: main}) only holds when $g$ can be replaced by a simpler equivalent integer chain. This way we show:
\begin{theorem} \label{thm:CM gap}
Let $G$ be the fundamental group of a graph of groups such that the embedding of every edge group $C \le G$ has property BCMS-$D$.  Let $c$ be an integral chain in $G$. Then either $c$ is equivalent (Definition \ref{def: equivalent chains}) to an integral chain $\tilde{c}$ such that every term lies in a vertex group or
$$
\scl_G(c) \geq \frac{1}{12 (D+2)}.
$$
\end{theorem}

This section is organized as follows. 
In Sections \ref{subsec:property cm} and \ref{subsec:bcmsm} we define CM-subgroups and BCMS-$D$ subgroups respectively.
In Sections \ref{subsec:normal forms for elements in BCMS groups} and \ref{subsec:equations in bcms amalgamated free prod} we prove properties of BCMS-$D$ subgroups related to Equation (\ref{equn: main}). Then we solve Equation (\ref{equn: main}) in Section \ref{subsec:the equation gm hkhpc in bcms amalg free prod} and prove Theorem \ref{thm:CM gap} in Section \ref{subsec:proof of the cm gap thm}.

\subsection{CM-subgroups} \label{subsec:property cm}

In this section we introduce \emph{central/malnormal subgroups (CM-subgroups)}.
CM-subgroups are generalizations of two very different types of subgroups: central subgroups and malnormal subgroups. Recall that a subgroup $H \le G$ is \emph{central}, if for every element $g \in G$ and every element $h \in H$ we have that $g h g^{-1} = h$.
On the other hand, a subgroup $H \le G$ is \emph{malnormal}, if for  every element $g \in G \setminus H$ and every element $h \in H$ we have that $g h g^{-1} \notin H$.

We say that an element $g \in G$ is a \emph{CM-representative for $H \le G$}, if for every $h \in H$ either
\begin{itemize}
\item[(i)] $g h g^{-1} = h$, or
\item[(ii)] $g h g^{-1} \notin H$.
\end{itemize}

For a subset $S$ of $G$, let $Z_H(S)$ be the subgroup of elements in $H$ commuting with all elements of $S$. When $S=\{g\}$, we simply denote it as $Z_H(g)$. Then $g$ is a CM-representative for $H$ if and only if $gHg^{-1}\cap H=Z_H(g)$.

\begin{proposition}[Uniqueness of CM-representatives]\label{prop: CM-representatives unique up to conjugation}
Let $g$ be a CM-representative for $H\le G$. Then $g' \in H g H$ is a CM-representative if and only if there are elements 
$h \in H$, $z \in Z_H(Z_H(g))$ such that $g' = h z g h^{-1}$.
In this case, we have that $Z_H(g') = h Z_H(g) h^{-1}$.
\end{proposition}
\begin{proof}
%First we show $g'=hzgh^{-1}$ is a CM-representative for any $h\in H$ and $z \in Z_H(Z_H(g))$.
%By the comment above, we need to show $g' H g'^{-1}\cap H=Z_H(g')$, which we compute as follows:
%$$g'Hg'^{-1}\cap H=[hzg(h^{-1}Hh)g^{-1}(hz)^{-1}]\cap H=hz (gHg^{-1}) (hz)^{-1}\cap [(hz)H(hz)^{-1}]=hz (gHg^{-1}\cap H) (hz)^{-1},$$
%where we used the fact that $hz\in H$. 
%Since $g$ is a CM-representative, we have $gHg^{-1}\cap H=Z_H(g)$. Thus
%$$hz (gHg^{-1}\cap H) (hz)^{-1}=h (zZ_H(g)z^{-1}) h^{-1}=hZ_H(g)h^{-1}=Z_H(hgh^{-1}),$$ 

First assume that $g$ is a CM-representative and let $g' = h z g h^{-1}$ for some $h \in H$ and $z \in Z_H(Z_H(g))$. We show that $g'$ is a CM-representative.
For any $x\in H$, we have $g'x g'^{-1}=hzg(h^{-1}xh)g^{-1}z^{-1}h^{-1}$. Since $h^{-1}xh\in H$ and $g$ is a CM-representative, either $g(h^{-1}xh)g^{-1}\notin H$ or $g(h^{-1}xh)g^{-1}=h^{-1}xh$. In the former case we have $g'x g'^{-1}\notin H$ since $hz\in H$, while in the latter case we have $h^{-1}xh\in Z_H(g)$ and $g'x g'^{-1}=hz(h^{-1}xh)z^{-1}h^{-1}=h(h^{-1}xh)h^{-1}=x$.
Thus $g'$ is a CM-representative, and the calculation shows that $x\in Z_H(g')$ if and only if $h^{-1}xh\in Z_H(g)$, i.e. $x\in hZ_H(g)h^{-1}$.

Conversely, if $g'=h_1 g h_2$ is a CM-representative for some $h_1,h_2\in H$, then by what we proved above, so is $g''=hg$, where $h=h_2 h_1$. 
Then for any $x\in Z_H(g)$, we have $g''x g''^{-1}=hgxg^{-1}h^{-1}=hxh^{-1}\in H$. Since $g''$ is a CM-representative, we must have $hxh^{-1}=g''xg''^{-1}=x$ for all $x\in Z_H(g)$.
Hence $h\in Z_H(Z_H(g))$.
\end{proof}

\begin{definition}[CM-subgroups and CM-choice] \label{def:cm subgroup}
We say that $H \le G$ is a \emph{CM-subgroup of $G$}, if
for every $g \in G$ there is an element $\bar{g} \in H g H$ such that $\bar{g}$ is a CM-representative for $H$.

A \emph{CM-choice} for a CM-subgroup $H\le G$ is a choice of one CM-representative for each double coset $H g H$ with $g \in G$.
\end{definition}

Every central or malnormal subgroup $H \le G$ is a CM-subgroup.
The motivating example for CM-subgroups come from right-angled Artin groups:
We will see that for any induced subgraph $\Lambda \subset \Gamma$ the associated right-angled Artin group $\Arm(\Lambda)$ is a CM-subgroup of $\Arm(\Gamma)$ (Lemma \ref{lemma: Lambda has property CM}). We will have this application in mind throughout this section.

\begin{example} \label{exmp: first example of CM subgroup}
Consider the graph $\Delta_1$ with vertex set $\{ v_0, v_1 \}$ and empty edge set and the graph $\Delta_2$ with vertex set $\{ v_0, v_1, v_2 \}$ and a single edge $(v_0, v_2)$.
The associated right-angled Artin groups are $\Arm(\Delta_1) \cong \Z \star \Z$ and $\Arm(\Delta_2) \cong \Z \star \Z^2$.

The subgroup $\Arm(\Delta_1)$ arises naturally as a subgroup of $\Arm(\Delta_2)$ and is neither 
central nor malnormal, but it is a CM-subgroup by Lemma \ref{lemma: Lambda has property CM}. 
Not every element of $\Arm(\Delta_2) \setminus \Arm(\Delta_1)$ is a CM-representative, such as $v_1 v_2 \in \Arm(\Delta_2) \setminus \Arm(\Delta_1)$: For $v_0 \in \Arm(\Delta_1)$ we have that
$(v_1 v_2) v_0 (v_1 v_2)^{-1} = v_1 v_0 v_1^{-1} \in \Arm(\Delta_1)$, but $(v_1 v_2) v_0 (v_1 v_2)^{-1} \neq v_0$.
However, $v_2 \in \Arm(\Delta_1) (v_1 v_2) \Arm(\Delta_1)$ \emph{is} a CM-representative. 

We will see that for every double coset $\Arm(\Delta_1) g \Arm(\Delta_1)$, an element with the shortest word length in the double coset is a CM-representative (Lemma \ref{lemma: Lambda has property CM}). This yields a natural CM-choice.
\end{example}

\begin{proposition}[Inheritance properties of CM-subgroups] \label{prop:inheritance CM subgroups}
Let $K \le H \le G$ be nested subgroups.
\begin{itemize}
\item If $K \le G$ is a CM-subgroup then $K \le H$ is a CM-subgroup.
\item If $K \le H$ is a CM-subgroup and $H\le G$ is a CM-subgroup then $K\le G$ is a CM-subgroup.
\end{itemize}
\end{proposition}

\begin{proof}
The first item is immediate. For the second item, for any $g \in G \setminus K$ we need to find a CM-representative in $K g K$. As $H \le G$ is a CM-subgroup, there is a CM-representative in $H g H$ for $H\le G$. By Proposition \ref{prop: CM-representatives unique up to conjugation} there is a CM-representative of the form $\bar{g}=gh$ for some $h \in H$. Similarly, since $K\le H$ is a CM-subgroup, we have a CM-representative $\bar{h}=kh$ for $h$ with $k\in K$.

Then $g'=gk^{-1}=\bar{g}\bar{h}^{-1}$ is a CM-representative in $KgK$. Indeed, for any $k_0\in K$, we have
$$g'k_0 g'^{-1}=\bar{g}\bar{h}^{-1} k_0 \bar{h}\bar{g}^{-1}.$$
As $g'k_0 g'^{-1}$ is the conjugate of $\bar{h}^{-1} k_0 \bar{h}\in H$ by $\bar{g}$,
it is either outside $H$ and hence outside $K$ or equal to $\bar{h}^{-1} k_0 \bar{h}$. In the latter case, either $\bar{h}^{-1} k_0 \bar{h}\notin K$ or $\bar{h}^{-1} k_0 \bar{h}=k_0$ since $\bar{h}$ is a CM-representative.
\end{proof}

\subsection{BCMS-$D$ subgroups} \label{subsec:bcmsm}
Given a proper CM-subgroup $H$ of a group $G$ and a CM-representative $g \in G \setminus H$ the centralizer $Z_{H}(g)$ measures how much the subgroup $H<G$ fails to be malnormal for the element $g$. It has an interesting structure in the motivating example of RAAGs.

\begin{example}\label{example: RAAG CM-subgroup}
Let $\Delta_1$ and $\Delta_2$ be the graphs defined in Example \ref{exmp: first example of CM subgroup}. We have seen that $v_2 \in \Arm(\Delta_2) \setminus \Arm(\Delta_1)$ is a CM-representative. Here $Z_{\Arm(\Delta_1)}(v_2) = \Arm(\Delta_0)$ where $\Delta_0$ is the graph with single vertex $v_0$.

More generally, we will see that if we choose CM-representatives to be elements in each double coset of minimal length then every such centralizer is the right-angled Artin group on an induced subgraph of the defining graph (Lemma \ref{lemma: Lambda has property CM}) and thus it is again a CM-subgroup.

On the other hand, if we choose the CM-representatives in a different way, the centralizers may not have this structure, but they only differ by conjugations according to Proposition \ref{prop: CM-representatives unique up to conjugation}
\end{example}

Let $H_0$ be a group and let $H_1$ be a proper CM-subgroup of $H_0$. Let $h_0\in H_0\setminus H_1$ be a CM-representative. Then $H_2 := Z_{H_1}(h_0)$ is a subgroup of $H_1$. There are three cases:
\begin{enumerate}[(i)]
	\item if $H_2 = H_1$ then $H_1$ lies in the centralizer of the element $h_0$,
	\item if $H_2 = \{ e \}$, then $H_1$ behaves like a malnormal subgroup with respect to the element $h_0$, or
	\item $\{ e \} \neq H_2 < H_1$ is a proper nontrivial subgroup. \label{item: H_2 case 3}
\end{enumerate}
If $h_0$ is as in case (\ref{item: H_2 case 3}) and $H_2$ is a CM-subgroup of $H_1$ then we may continue this process: Given a CM-representative $h_1 \in H_1 \setminus H_2$, define $H_3 = Z_{H_2}(h_1)$. 

Informally, if this process always yields CM-subgroups and eventually stops (in about $D$ steps), then we say the subgroup $H_1$ has \emph{bounded CM-subgroup sequence of depth $D$}, which we abbreviate as BCMS-$D$. We make this precise in the following definitions.

\begin{definition}[CM-subgroup sequence]
	In a group $H$, a CM-subgroup sequence of length $m+1$ is a sequence of nested subgroups $H=H_0> H_1>\cdots >H_{m+1}\ge H_{m+2}$ such that $H_{i+1}$ is a proper CM-subgroup of $H_i$ for all $0\le i\le m$ (\emph{not} including $i=m+1$) and $H_{i+2}=Z_{H_{i+1}}(g_i)$ for some $g_i\in H_i\setminus H_{i+1}$.
\end{definition}
For any CM-subgroup sequence $H_{m+2} \le \cdots \le H_0$, if $H_1$ is central we must have $H_2=H_1$, which forces $m=0$. If $H_1$ is malnormal, then we have $H_2=\{e\}=H_3$, forcing $m\le 1$. Note that not every nested sequence of proper CM-subgroups is a CM-subgroup sequence due to the requirement $H_{i+2}=Z_{H_{i+1}}(g_i)$. For instance, $\{e\}\le \Z \le \Z^2$ is a nested sequence of proper CM-subgroups, but $Z_{\Z}(g)\neq \{e\}$ for all $g\in \Z^2$.

It is important to note that, in the definition of CM-subgroup sequences, the only requirement on $H_{m+2}$ is that $H_{m+2}=Z_{H_{m+1}}(g_m)$ for some $g_m\in H_m\setminus H_{m+1}$, and in general it may not be a CM-subgroup of $H_{m+1}$. It is part of the definition of BCMS subgroups below that $H_{m+2}$ is required to be a proper CM-subgroup of $H_{m+1}$ except when $H_{m+2}=H_{m+1}$.

\begin{definition}[BCMS-$D$] \label{def:bcmsm}
	Let $D \in \Z_{\ge0}$ be an integer, and let $H_0$ be a group. 
	We say that a subgroup $H_1\le H_0$ (or really the pair $(H_0,H_1)$) has \emph{bounded CM-subgroup sequences of depth $D$ (BCMS-$D$)} 
	if $H_1$ is a CM-subgroup and for every CM-subgroup sequence $H_{m+2} \le \cdots \le H_0$ we have that either $H_{m+2} = H_{m+1}$ or that $H_{m+2} < H_{m+1}$ is a proper CM-subgroup. Moreover we require that every CM-subgroup sequence has length at most $D+1$, i.e. if $H_{m+2} \le \cdots \le H_0$ is a CM-subgroup sequence then $m \leq D$. We also say $H_1$ is a BCMS-$D$ subgroup.
\end{definition}

We see that central subgroups have BCMS-$0$ and malnormal subgroups have BCMS-$1$. Note that by definition a subgroup $H$ has BCMS-$D$ then it also has BCMS-$D'$ if $D'\ge D$, i.e. we do not require $D$ to be the optimal upper bound.

In general, verifying whether a CM-subgroup $H\le G$ has BCMS-$D$ requires one to check all CM-subgroup sequences. As we saw in Example \ref{example: RAAG CM-subgroup}, certain choices of CM-representatives have centralizers that are easier to study in some cases. We will show that one can restrict attention to some special families of CM-subgroup sequences corresponding to nice choices of CM-representatives to verify whether a CM-subgroup has BCMS-$D$.

We incorporate the choice of CM-representatives into the following notion.
\begin{definition}[CM-subgroup-choice] \label{def:cm-subgroup-choice}
A \emph{CM-subgroup-choice} $\Icl(G)$ is a CM-choice (Definition \ref{def:cm subgroup}) for every proper CM-subgroup $H < G$. 
\end{definition}

Given a CM-subgroup-choice $\Icl(G)$, whenever we have a chain of proper subgroups $K< H< G$ such that $K < H$ and $H < G$ are CM-subgroups. Then $K < G$ is also a CM-subgroup by Proposition \ref{prop:inheritance CM subgroups}. For any element $h \in H \setminus K$ the CM-subgroup-choice $\Icl(G)$ gives us a CM-representative of $h$ for $K < G$ which is also a CM-representative for $K < H$.

\begin{definition}[CM-sequence]
	Given a proper CM-subgroup $H_1<H_0$ and a CM-subgroup-choice $\Icl(H_0)$, a \emph{CM-sequence} of length $m+1$ is a sequence of elements $(h_0, \ldots, h_m)$ in $G$ such that 
	there is a CM-subgroup sequence $H_{m+2}\le H_{m+1}<\cdots<H_0$ satisfying
	\begin{itemize}
		\item $h_i\in H_i\setminus H_{i+1}$ is the CM-representative for $H_{i+1}$ provided by $\Icl(G)$ for all $0\le i\le m$, and
		\item $H_{i+2}=Z_{H_{i+1}}(h_i)$ for all $0\le i\le m$.
	\end{itemize}
	
	Given a \emph{CM-sequence} $(h_0, \ldots, h_m)$, it uniquely determines the CM-subgroup sequence $H_{m+2}\le H_{m+1}<\cdots<H_0$ by the relation $H_{i+2}=Z_{H_{i+1}}(h_i)$, which we refer to as the associated CM-subgroup sequence.
\end{definition}

Apparently, if $H_1$ is a BCMS-$D$ subgroup, then any CM-sequence $(h_0, \ldots, h_m)$ has length at most $D+1$, i.e. $m\le D$. Conversely, given a CM-subgroup-choice $\Icl(G)$, not every CM-subgroup sequence appears as one associated to some CM-sequence $(h_0, \ldots, h_m)$. However, it suffices to consider CM-subgroup sequences associated to CM-sequences to show that $H_1$ is a BCMS-$D$ subgroup.

\begin{proposition} \label{prop: bcms enough to check for subgroup choices}
	Fix a CM-subgroup-choice $\Icl(H_0)$. Suppose $H_1$ is a proper CM-subgroup of $H_0$, and for the CM-subgroup sequence $H_{m+2} \le \cdots \le H_0$ associated to any CM-sequence $(h_0, \ldots, h_m)$, we have that either $H_{m+2} = H_{m+1}$ or that $H_{m+2} < H_{m+1}$ is a proper CM-subgroup. Then $H_1$ has BCMS-$D$ if and only if every CM-sequence $(h_0, \ldots, h_m)$ has $m\le D$.
\end{proposition}
\begin{proof}
	Given any CM-subgroup sequence $H_{m+2} \le \cdots \le H_0$, we claim that for any $0\le k\le m$, there is a CM-subgroup sequence $H'_{m+2}\le \cdots \le H'_0$ with $H'_1=H_1$ and $H'_0=H_0$ such that
	\begin{itemize}
		\item $H_{m+2}=H_{m+1}$ if and only if $H'_{m+2}=H'_{m+1}$, and $H_{m+2} \le H_{m+1}$ is a proper CM-subgroup if and only if $H'_{m+2} \le H'_{m+1}$ is a proper CM-subgroup;
		\item there is a CM-sequence $(\bar{h}_0,\cdots,\bar{h}_k)$ whose associated CM-subgroup is $H'_{k+2}\le \cdots \le H'_0$. 
	\end{itemize}
	The claim with $k=m$ together with our assumption shows that, whenever we have a CM-subgroup sequence $H_{m+2} \le \cdots \le H_0$, we have that either $H_{m+2} = H_{m+1}$ or that $H_{m+2} \le H_{m+1}$ is a proper CM-subgroup. Moreover, there is a CM-sequence $(\bar{h}_0,\cdots,\bar{h}_m)$ of the same length, which proves the proposition.
	
	Thus it suffices to prove this claim, which we show by induction on $k$. For the base case $k=0$, by definition there is a CM-representative $h_0$ for $H_1<H_0$ (not necessarily from $\Icl(H_0)$) such that $Z_{H_1}(h_0)=H_2$. Let $\bar{h}_0$ be the CM-representative in $H_1 h_0 H_1$ chosen by $\Icl(H_0)$. By Proposition \ref{prop: CM-representatives unique up to conjugation}, there is some $h\in H_1$ and $z\in Z_{H_1}(H_2)$ such that $\bar{h}_0=h z h_0 h^{-1}$. In this case, 
	$hH_{m+2}h^{-1}\le \cdots\le h H_2 h^{-1} \le  H_1\le H_0$ is a CM-subgroup sequence where $h H_2 h^{-1} \le  H_1 \le  H_0$ is the CM-subgroup sequence associated to the CM-sequence $(\bar{h}_0)$ since $Z_{H_1}(\bar{h}_0)=h H_2 h^{-1}$ by Proposition \ref{prop: CM-representatives unique up to conjugation}.
	
	Suppose the claim holds for some $0\le k<m$, i.e. there is a CM-subgroup sequence $H'_{m+2}\le \cdots \le H'_0$ with $H'_1=H_1$ and $H'_0=H_0$, such that the relation between $H'_{m+2}$ and $H'_{m+1}$ corresponds to the relation between $H_{m+2}$ and $H_{m+1}$, and there is a CM-sequence $(\bar{h}_0,\cdots,\bar{h}_k)$ whose associated CM-subgroup is $H'_{k+2}\le \cdots \le H'_0$. Since $k<m$, there is a CM-representative $h_{k+1}\in H'_{k+1}\setminus H'_{k+2}$ such that $Z_{H'_{k+2}}(h_{k+1})=H'_{k+3}$. Let $\bar{h}_{k+1}=h z h_{k+1} h^{-1}$ be the CM-representative in $H'_{k+2} h_{k+1} H'_{k+2}$, where $h\in H'_{k+2}$ and $z\in Z_{H'_{k+2}}(H'_{k+3})$. Then 
	$hH'_{m+2}h^{-1}\le \cdots\le h H'_{k+3} h^{-1} \le  H'_{k+2}\le \cdots \le  H'_0$ is a CM-subgroup sequence where
	$h H'_{k+3} h^{-1} \le H'_{k+2}\le \cdots \le H'_0$ is the CM-subgroup sequence associated to the CM-sequence
	$(\bar{h}_0,\cdots,\bar{h}_k,\bar{h}_{k+1})$ since $Z_{H'_{k+2}}(\bar{h}_{k+1})=h H'_{k+3} h^{-1}$ by Proposition \ref{prop: CM-representatives unique up to conjugation}. This completes the induction and proves the proposition.
\end{proof}

In what follows, we will use the proposition above as an alternative definition of BCMS-$D$ subgroups since it is easier to check. In practice, only certain subgroups arise as $H_i$ in some CM-subgroup sequence associated to a CM-subgroup, and thus one only needs to fix the CM-subgroup-choice for these CM-subgroups of $H_0$. See the example below.

For every $D \in \Z_+$ there is a BCMS-$D$ subgroup of a group which is not a BCMS-$(D-1)$-subgroup.

\begin{example}
For $n \in \N$, let $\Delta_n$ be the graph with vertex and edge set
\begin{eqnarray*}
V(\Delta_n) &=& \{ v_0, \ldots, v_n \} \mbox{ and} \\
E(\Delta_n) &=& \{ (v_i,v_j) \mid |i-j| \geq 2 \}.
\end{eqnarray*}
 For $n \in \N$ and $i \in \{1, \ldots, n \}$ let $\Delta^i_n$ be the induced subgraph of $\Delta_n$ with vertex set
 $$
 V(\Delta^i_n) = \{ v_i, \ldots, v_n \}.
 $$
By Lemma \ref{lemma: Lambda has property CM} we have that $\Arm(\Delta^1_n) < \Arm(\Delta_n)$ is a CM-subgroup and that $v_0$ is a CM-representative.
We compute that $Z_{\Arm(\Delta^1_n)}(v_0) = \Arm(\Delta^2_n)$. 
More generally we will see that $\Arm(\Delta^{i}_n)$ is a CM-subgroup of $\Arm(\Delta^{i-1}_n)$, that $v_{i-1}$ is a CM-representative and 
that $Z_{\Arm(\Delta^{i}_n)}(v_{i-1}) = \Arm(\Delta^{i+1}_n)$ for $1 \leq i \leq n-1$.

Thus $(v_0, \ldots, v_n)$ is a CM-sequence of length $n+1$ and the associated CM-subgroup sequence is 
$$
\{ e \} \le \{ e \} \le \Arm(\Delta_n^n) \le \cdots \le \Arm(\Delta_n^1) \le \Arm(\Delta_n).
$$ 
We will see that those are, in some sense, the longest CM-sequence for subgroups associated to induced subgraphs on RAAGs (Lemma \ref{lemma: CM length and delta subgraphs}).
%\# say something about the CM-subgroup-choice here.
\end{example}

\subsection{Normal forms for elements in BCMS-$D$ subgroups}
\label{subsec:normal forms for elements in BCMS groups}

If $H_1$ is a BCMS-$D$ subgroup of $H_0$, given a CM-subgroup-choice $\Icl(H_0)$, then we may write every element as a product of CM-representatives
up to conjugation as follows:

\begin{proposition}[Normal form for elements] \label{prop:normal forms for elts}
Let $H_1$ be a BCMS-$D$ subgroup of $H_0$ with a CM-subgroup-choice $\Icl(H_0)$, and let $g \in H_0 \setminus H_1$ be an element.

Then there is $n \leq D$, a CM-sequence $(h_0, \ldots, h_n)$
with associated CM-subgroup sequence $H_{n+2} \le \cdots \le H_0$,
and a conjugate $g'$ of $g$ by an element of $H_1$ such that
$$
g' = h_0 \cdots h_n e_{n}
$$
with $e_n \in H_{n+2}$. Moreover, the integer $n$ and the CM-sequence $(h_0,\cdots, h_n)$ are uniquely determined by $g$, and $e_n$ is unique up to conjugation in $H_{n+2}$. % this expression is unique.
\end{proposition}
\begin{proof}
We inductively prove the following statement:
\begin{claim}
For every $m \geq 0$, there is a conjugate $g'$ of $g$ by an element of $H_1$ such that either
\begin{enumerate}[(i)]
\item
$g' = h_{0} \cdots h_m e_m$ for some $e_m \in H_{m+1}$, 
where $(h_0, \ldots, h_m)$ is a CM-sequence with $H_{m+2} \le \cdots \le H_0$ as the associated CM-subgroup sequence, or \label{item: normal form claim case 1}
\item
$g' = h_{0} \cdots h_j e_j$ for some $e_j \in H_{j+2}$ and $j \leq m$, where
$(h_0, \ldots, h_j)$ is a CM-sequence with $H_{j+2} \le \cdots \le H_0$ as the associated CM-subgroup sequence. \label{item: normal form claim case 2}
\end{enumerate}
\end{claim}

\begin{proof}
Let $h_{0}=hgh'$ be the CM-representative in $H_1gH_1$ provided by $\Icl(H_0)$, where $h,h'\in H_1$. 
Then $h_0=g' hh'$ for $g'\defeq hgh^{-1}$, and thus $g' = h_{0} e_{0}$ with $e_{0}\defeq (hh')^{-1} \in H_1$. This shows the claim for $m = 0$.

Suppose the claim is true for some $m \geq 0$.
If statement (\ref{item: normal form claim case 2}) holds for $m$ then it also holds for $m+1$ and we are done. Thus assume that $g' =  h_{0} \cdots h_m e_m$ where $g'$ is a conjugate of $g$ by
some element in $H_1$, $(h_0, \ldots, h_m)$ is a CM-sequence with associated CM-subgroup sequence $H_{m+2}\le \cdots \le H_0$, and $e_m \in H_{m+1}$. 

If $e_m \in H_{m+2}$ we are done as in case (\ref{item: normal form claim case 2}) as well. Otherwise, let $h_{m+1} \in H_{m+1} \setminus H_{m+2}$ be the CM-representative in $H_{m+2}e_m H_{m+2}$ given by $\Icl(H_0)$. 
Then $h^l h_{m+1} h^r =  e_m $ for some $h^l, h^r \in H_{m+2}$. Thus
$$
g' = h_{0} \cdots h_m h^l h_{m+1} h^r = h^l h_{0} \cdots h_m h_{m+1} h^r
$$
as $h^l$ commutes with all $h_0, \ldots, h_m$ by the definition of $H_{m+2}$. Conjugating both sides of the equation above by $h^l$ and setting $e_{m+1} = h^r h^l$ proves the claim. 
\end{proof}

Now as $H_1<H_0$ has property BCMS-$D$ we will arrive at item (\ref{item: normal form claim case 2}) of the claim eventually (if $m \geq D$). %We also see that this expression is unique as no choice is involved above for the CM-representatives since $\Icl(H_0)$ is fixed.

The uniqueness can be observed in the inductive construction above as follows. Note that $h_0$ is uniquely determined as the CM-representative in $H_1gH_1$ since $\Icl(H_0)$ is fixed. Next we show that $e_0$ is uniquely determined up to conjugation by an element in $H_2$. Suppose there is a different choice $e'_0\in H_1$ such that $h_0e'_0= h h_0 e_0 h^{-1}$ for some $h\in H_1$, then $h^{-1}=h_0 [e_0 h^{-1} (e'_0)^{-1}]h_0^{-1}$, which forces $h^{-1}=e_0 h^{-1} (e'_0)^{-1}$ as $h_0$ is a CM-representative for $H_1$. Thus $h^{-1}e'_0 h =e_0$, so $h_0e'_0=h h_0 e_0 h^{-1}=h h_0 h^{-1} e'_0$, which implies $h_0=h h_0 h^{-1}$, i.e. $h\in Z_{H_1}(h_0)=H_2$. This proves that $e'_0$ differs from $e_0$ via conjugation by some $h\in H_2$. In particular, the double coset $H_2 e_0 H_2$ is uniquely determined and so is $h_1$. Continuing this process, one can observe that each $h_i$ in the expression is uniquely determined, each element $e_i$ is unique up to conjugation by an element of $H_{i+2}$, and the integer $n$ is characterized as the first $n$ such that $e_n\in H_{n+2}$ (which is not ambiguous by the uniqueness up to conjugation).
\end{proof}

\begin{definition}[CM-reduced element] \label{def: CM reduced elts}
Suppose that $H_1$ is a BCMS-$D$ subgroup of $H_0$ with CM-subgroup-choice $\Icl(H_0)$. For any $g \in H_0 \setminus H_1$, we say that 
$g$ is \emph{CM-reduced} if we have $e_n = 1$ when $g$ is written as in the normal form given by Proposition \ref{prop:normal forms for elts} and no conjugation is involved. That is, $g=h_0\cdots h_n$ for a CM-sequence $(h_0,\cdots, h_n)$ and some $0\le n\le D$.
%Note that $e_n$ is unique up to conjugation in $H_{n+2}$ so this is well defined.
\end{definition}

\begin{proposition} \label{prop: cm reduced}
Let $H_1$ be a BCMS-$D$ subgroup of $H_0$ with CM-subgroup-choice $\Icl(H_0)$. Let $g \in H_0$ be an element and let $g'$, $e_n$ and $(h_0,\ldots, h_n)$ be as in the normal form from Proposition \ref{prop:normal forms for elts}. 
Then $g$ is conjugate to $he_n$ by an element of $H_1$ and equivalent to $h+e_n$ as a chain, where $h = h_0 \cdots h_n$ is CM-reduced.
\end{proposition}
\begin{proof}
This follows immediately from Proposition \ref{prop:normal forms for elts} and Definition \ref{def: equivalent chains} noting that $e_n$ commutes with $h_0, \ldots, h_n$.
\end{proof}

\subsection{Equations in amalgamated free products or HNN extensions}
\label{subsec:equations in bcms amalgamated free prod}

In the rest of Section \ref{sec:cm subgroups}, we will consider a group $G$ that splits over a BCMS-$D$ subgroup $C$. Note that if $g\in G\setminus C$ can be written as a cyclically reduced word in the sense of Definitions \ref{def:ab-alternating} or \ref{def:reduced form}, then naturally any element $cgc'\in CgC$ also has this property. In particular, in this case, any CM-representative in $CgC$ with respect to the CM-subgroup $C$ can be written as a cyclically reduced word. 

Similarly, if $g$ is CM-reduced with $g=c_0\cdots, c_m$ for a CM-sequence $(c_0,\cdots, c_m)$, then $g$ is cyclically reduced if and only if $c_0$ is. So we say $g$ is cyclically reduced and CM-reduced (e.g. in Proposition \ref{prop:gmhm} below) if $g$ can be written this way with $c_0$ cyclically reduced.

We will need the following proposition to compare terms in certain expressions in a group $G$ that splits over a BCMS-$D$ subgroup $C$. This is similar to Corollaries \ref{corr:replacement} and \ref{corr:replacement for HNN}.

\begin{proposition} \label{prop: normal form cm amalgamated free prod} 
Let $G$ be a group that splits over a BCMS-$D$ subgroup $C$. Let $\Icl(G)$ be a CM-subgroup-choice. For some $m \leq D$ let $(c_0, \ldots, c_m)$ is a CM-sequence and let $C_{m+2} \le \cdots \le C_1:= C \le C_0 := G$ be the associated CM-subgroup-sequence.

Suppose $c_0 \in G \setminus C$ (or equivalently $g$) can be written as a cyclically reduced word. %(Definitions \ref{def:ab-alternating} and \ref{def:reduced form}). 
Let $n \geq m+2$ and
suppose there are elements $x_1, \ldots, x_{n-1}, x_1', \ldots, x_{n-1}' \in C_{m+1}$ and $x_0, x_0', x_n, x_n' \in C_1$ such that
$$
x_0 c^{(m)} x_1 \cdots c^{(m)} x_n = x'_0 c^{(m)} x'_1 \cdots c^{(m)} x'_{n},
$$
where $c^{(m)} = c_{0} \cdots c_{m}$.
Then there are $d_1, \ldots, d_{n-m} \in C_{m+2}$ such that
$$
d_{i-1} x_i' d_{i}^{-1} = x_i
$$
for all $2 \leq i \leq n-m$.
\end{proposition}

\begin{proof}
We observe that by Corollaries \ref{corr:replacement} and \ref{corr:replacement for HNN} there are elements $d_0, \ldots, d_n \in C_1$ with $d_0 = e = d_n$ such that
$x_0 c^{(m)} x_1 = d_0 x_0' c^{(m)} x_1' d_1^{-1}$ and
$c^{(m)} x_i = d_{i-1} c^{(m)} x'_i d_i^{-1}$ for all $i \in \{ 2, \ldots, n \}$.
\begin{claim}
For every $j \in \{0, \ldots, m \}$ we have that $d_i \in C_{j+2}$ for all $i \in \{1, \ldots, n-j \}$.
\end{claim}

\begin{proof}
We proceed by induction. For $j = 0$ we write $c^{(m)} = c_0 c''$ with $c'' = c_1 \cdots c_m$.
Then we obtain
$$
c_0^{-1} d_{i-1} c_0 = c'' x_i d_i {x'_i}^{-1} {c''}^{-1}
$$
for all $i \in \{2, \ldots, n \}$ from $c^{(m)} x_i = d_{i-1} c^{(m)} x'_i d_i^{-1}$.
Observe that both $d_{i-1} \in C_1$ and $c'' x_i d_i {x'_i}^{-1} {c''}^{-1} \in C_1$. Since $c_0$ is a CM-representative for $C_1 < C_0$ we have $d_{i-1} \in C_2=Z_{C_1}(c_0)$. 
Since $d_n = e$ we conclude that  $d_i \in C_2$ for all $i \in \{ 1, \ldots, n \}$.

Suppose the claim is true for some $j-1 \in \{0, \ldots, m-1\}$. We wish to show that it is true for $j$ as well. 
We may write $c^{(m)} = c' c_j c''$ for $c' = c_0 \cdots c_{j-1}$ and $c'' = c_{j+1} \cdots c_m$.
By the induction hypothesis we have that $d_i \in C_{j+1}$ for all $i \in \{ 1, \ldots, n-j+1 \}$, so all such $d_i$ commute with $c'$.
As above we obtain
$$
c_j^{-1} d_{i-1} c_j = c'' x_i d_i {x'_i}^{-1} {c''}^{-1}
$$
for all $i \in \{ 2, \ldots, n-j+1 \}$ from $c^{(m)} x_i = d_{i-1} c^{(m)} x'_i d_i^{-1}$.
Note that $d_{i-1}, d_i, x_i, x'_i, c'' \in C_{j+1}$ for all $i \in \{2, \ldots, n-j + 1 \}$.
Thus, as $c_j$ is a CM-representative, we see that $d_{i} \in C_{j+2}$ for all $i \in \{ 1, \ldots, n-j \}$.
This completes the induction. 
\end{proof}

For $j=m$ the claim implies that $d_i \in C_{m+2}$ for all $i \in \{1, \ldots, n-m \}$ and thus all such $d_i$ commute with $c^{(m)}$.
Hence from $c^{(m)} x_i = d_{i-1} c^{(m)} x'_i d_i^{-1}$ we have
$$
x_i = d_{i-1}  x'_i d_i^{-1}
$$
for all $i \in \{2, \ldots, n-m \}$. This finishes the proof.
\end{proof}

\subsection{Solutions to Equation (\ref{equn: main})}\label{subsec:the equation gm hkhpc in bcms amalg free prod}

\begin{proposition} \label{prop:gmhm} 
Let $G$ be a group that splits over a BCMS-$D$ subgroup $C$, and let $\Icl(G)$ be a CM-subgroup choice. Suppose $g \in G$ is cyclically reduced and CM-reduced. 
Suppose there is a cyclically reduced word $h \in G$ with $g^{N} = h^{N} c$ for some $c \in C$ and $N \geq D+2$.
Then there is an element $z$ which commutes with $g$ such that $g = h z$.
\end{proposition}
\begin{proof} 
By our assumption, we have that $g = c_{0} \cdots c_m$ 
where $c_0$ is cyclically reduced and $(c_0, \ldots, c_m)$ is a CM-sequence with associated CM-subgroup sequence $C_{m+2} \le \cdots \le C_0$, where $C_0=G$ and $C_1=C$. Note that $m \leq D$ since $C$ is a BCMS-$D$ subgroup.

By Corollaries \ref{corr:replacement} and \ref{corr:replacement for HNN} there are $d_i \in C$ for $0 \leq i \leq N$ with $d_0=e=d_{N}$ such that $g = d_{i-1} h d_i^{-1}$ for $1 \leq i \leq N-1$ and $g = d_{N-1} hc d_N^{-1}$.
Redefining $d_N^{-1}$ to be $c d_N^{-1}$ we get that 
$g = d_{i-1} h d_i^{-1}$ for $1 \leq i \leq N$ and thus
\begin{eqnarray} \label{eqn: dis}
d_{i-1}^{-1} g d_i = d_i^{-1} g d_{i+1}
\end{eqnarray}
for all $1 \leq i \leq N-1$.
\begin{claim}
For every $0 \leq j \leq m+1$ we have that 
$d_{i} \in C_{j+1}$ for all $0 \leq i \leq N-j$.
\end{claim}

\begin{proof}
We proceed by induction. For $j=0$ the claim is immediate as all terms are in $C_1 = C$.

Suppose the claim is true for some $0 \leq j \leq m$.
Write $g = c_0 \cdots c_m = c' c_j c''$ for $c' = c_0 \cdots c_{j-1}$ and $c'' = c_{j+1} \cdots c_m$.
Observe that by the induction hypothesis, $c'$ commutes with $d_i$ for all $0 \leq i \leq N-j$.
Thus for all $1 \leq i \leq N-j-1$, we deduce from equation (\ref{eqn: dis}) that
$$
c_j^{-1}  \left( d_i d_{i-1}^{-1} \right)  c_j = {c''} d_{i+1} d_{i}^{-1} {c''}^{-1}.
$$

By the induction hypothesis, ${c''} d_{i+1} d_{i}^{-1} {c''}^{-1} \in C_{j+1}$ for all such $i$. Thus
$d_i d_{i-1}^{-1} \in C_{j+2}$ for all $1 \leq i \leq N-j-1$ since $c_j$ is a CM-representative.
Recall that  $d_0 = e$. Thus for every $i \in \{ 1, \ldots, N-j-1 \}$ we have that
$$
d_{i} = d_i d_0^{-1} = (d_i d_{i-1}^{-1}) (d_{i-1} d_{i-2}^{-1}) \cdots (d_1 d_0^{-1}) \in C_{j+2}. 
$$
This shows the claim.
\end{proof}

In particular for $j=m+1$ the claim implies that $d_i \in C_{m+2}$ for all $0 \leq i \leq N-m-1$. Since $m \leq D$ and $D+2 \leq N$, we have that $d_1 \in C_{m+2}$. 
Thus $d_1$ commutes with $g$. This concludes the proof of Proposition \ref{prop:gmhm} as $g =d_0^{-1}h d_1=h d_1$.
\end{proof}

\begin{proposition} \label{prop:sol to main equn}
Let $G$ be a group that splits over a BCMS-$D$ subgroup $C$, and let $\Icl(G)$ be a CM-subgroup-choice.
Let $g,h\in G$ be cyclically reduced words with $|g| > |h|$ and let $h'$ be a prefix of $h$.
Suppose
$$
g^N = h^k h' c
$$
for some $c \in C$ and $N \geq D+2$.

Then there is a cyclically reduced element $x \in G$ such that $g = x^{n_g} c$ for some $n_g \geq 2$ and $c \in C$ that commutes with $x$.
\end{proposition}
\begin{proof}
Let $C_0=G$ and $C_1=C$. We inductively prove the following claim:
\begin{claim}
There are two coprime integers $n_g, n_h \in \Z_+$ and $0\le n_h'<n_h$ such that for every $m \geq 0$ either
\begin{enumerate}[(i)]
\item
there is a CM-sequence $(c_{0}, \cdots, c_m)$ with the associated CM-subgroup sequence $C_{m+2} \le \cdots \le C_0$ and
 elements $d_g, d_h \in C$ such that 
\begin{eqnarray*}
d_g g d_g^{-1} &=& c^{(m)} z_1  \cdots c^{(m)} z_{n_g} \mbox{, and} \\
d_h h d_h^{-1}  &=& c^{(m)} {z'}_1 \cdots c^{(m)} z_{n_h}',
\end{eqnarray*}
for $c^{(m)} = c_{0} \cdots c_{m}$ and $z_i, z'_i \in C_{m+1}$, or\label{item: solve eqn claim case 1}
\item
there is an $n \leq m$ and a CM-sequence $(c_{0}, \ldots, c_n)$ with the associated CM-subgroup-sequence $C_{n+2} \le \cdots \le C_0$ and elements $d_g, d_h \in C$ such that 
\begin{eqnarray*}
d_g g d_g^{-1} &=& c^{(n)} z_1  \cdots c^{(n)} z_{n_g} \mbox{, and} \\
d_h h d_h^{-1}  &=& c^{(n)} {z'}_1  \cdots c^{(n)} z_{n_h}',
\end{eqnarray*}
for $c^{(n)} = c_{0} \cdots c_{n}$ and $z_i, z'_i \in C_{n+2}$.\label{item: solve eqn claim case 2}
\end{enumerate}
\end{claim}

\begin{proof}
We first show that the claim is true for $m = 0$. Let $d$ be the greatest common divisor of $|g|$ and $|h|$. Note that $c$ lies in $C$, which is the subgroup that $G$ splits over, so it can be ignored whenever we measure the length of a reduced word. Since both $g$ and $h$ are cyclically reduced and $h'$ is a prefix of $h$, we have $N|g|=|g^N|=|h^k h'|=k|h|+|h'|$. Hence $d$ also divides $|h'|$.
Thus we can write $g = g_1 \cdots g_{n_g}$, $h = h_1 \cdots h_{n_h}$ and $h' = h_1 \cdots h_{n_h'}$, where $n_g = |g|/d$, $n_h = |h|/d$, $n_h'=|h'|/d$ and all the $g_i$ and $h_i$ are reduced words of length $d$. Note that $n_g>n_h\ge 1$ since $|h|<|g|$.
%In the case where $G$ is an amalgamation $G=A\star_C B$, both $|g|$ and $|h|$ are even as $g$ and $h$ are AB-alternating, and thus $d$ is even as well.

Then we have reduced decompositions
$$
(g_1 \cdots g_{n_g})^N = (h_1 \cdots h_{n_h})^k h_1 \cdots h_{n'_h-1} (h_{n'_h} c).
$$
By Corollaries \ref{corr:replacement} and \ref{corr:replacement for HNN} there are elements $d_0, \ldots, d_{N n_g} \in C_1 = C$ with $d_0 = e$ and $d_{N n_g}=c^{-1}$ such that
$g_{i}  = d_{i-1} h_{i} d_i^{-1}$ for all $1 \leq i \leq  N n_g$, where the index $i$ in $g_i$ and $h_i$ is taken mod $n_g$ and $n_h$ respectively.
Thus for all $1 \leq i \leq n_g$, we have
$$
g_i = d_{i-1} h_i d_i^{-1} = d_{i-1} h_{i+n_h} d_i^{-1} = d_{i-1} d_{i+n_h-1}^{-1} g_{i+n_h} d_{i+n_h} d_{i}^{-1},
$$
and hence $g_i \in C g_{i+n_h} C$.
As $n_h$ and $n_g$ are coprime we see that $g_i \in C g_1 C$ for all $1 \leq i \leq  n_g$, and by $g_{i}  = d_{i-1} h_{i} d_i^{-1}$ we have $h_i \in C g_1 C$ for all $1 \leq i \leq  n_h$.
Let $c_{0} \in C$ be the CM-representative of $C g_1 C$ provided by $\Icl(G)$. Then the above calculations show that
\begin{eqnarray*}
g &=& z_0 c_{0} z_1 \cdots c_{0} z_{n_g} \mbox{, and} \\
h &=& z'_0 c_{0}  z_1' \cdots c_{0} z_{n_h}',
\end{eqnarray*}
for some ${z_i},z_i' \in C=C_1$.
Conjugating $g$ and $h$ by $z_0$ and $z'_0$ respectively and possibly changing ${z_{n_g}}$ and $z_{n_h}'$ we achieve case (\ref{item: solve eqn claim case 1}) of the claim with $m=0$.
Note that $c_0$ is cyclically reduced by the expression above since $g$ is cyclically reduced and $|g|=n_g|g_1|=n_g|c_0|$.

Now suppose that the claim is true for some $m \geq 0$. 
We prove it for $m+1$.
If item (\ref{item: solve eqn claim case 2}) of the claim holds for $m$ then clearly it holds for $m+1$ and we are done. Thus suppose that item (\ref{item: solve eqn claim case 1}) holds for $m$. 
We will argue similarly as in the case of $m = 0$.
By the induction hypothesis we have that
\begin{eqnarray*}
g  &=& d_g^{-1} c^{(m)} z_1  \cdots c^{(m)} z_{n_g} d_g \mbox{, and} \\
h  &=& d_h^{-1} c^{(m)} z_1'  \cdots c^{(m)} z_{n_h} d_h,
\end{eqnarray*}
for some $d_g, d_h \in C_1$, $c^{(m)}  = c_{0} \cdots c_m$, and $z_i,z_i'\in C_{m+1}$. Since $h'$ is a prefix of $h$, we have a reduced decomposition $h=h'\cdot h''$ for some reduced word $h''$. Comparing it to the reduced decomposition
$$h=\left( d_h^{-1} c^{(m)} z_1'  \cdots c^{(m)} z_{n_h'}\right) \left( c^{(m)} z_{n_h'+1} c^{(m)} z_{n_h} d_h\right),$$
by Corollaries \ref{corr:replacement} and \ref{corr:replacement for HNN} we observe that $h'=d_h^{-1} c^{(m)} z_1'  \cdots c^{(m)} z_{n_h'} d_{h'}$ for some $d_{h'}\in C_1$.
Thus
$$
d_g^{-1} \left( c^{(m)} z_1 \cdots c^{(m)}  z_{n_g} \right)^N d_g =  d_h^{-1} \left( c^{(m)}  {z'}_1  \cdots c^{(m)}  {z'}_{n_h} \right)^k  \left( c^{(m)}  {z'}_1 \cdots c^{(m)}  {z'}_{n_h'} \right) d_{h'} c.
$$
Applying Proposition \ref{prop: normal form cm amalgamated free prod} to this equation with $n=N\cdot n_g$, we obtain 
elements $d_1, \cdots, d_{N n_g - m} \in C_{m+2} = Z_{C_{m+1}}(c_{m})$ such that
$z_{i}  = d_{i-1} z'_{i} d_i^{-1}$ for all $i \in \{2, \ldots, N n_g - m \}$, where the index $i$ in $z_i$ and $z'_i$ is taken mod $n_g$ and $n_h$ respectively.

Note that $m \leq D$ since $(c_0, \ldots, c_m)$ is a CM-sequence, and thus $m+2\le D+2\le N$. It follows that $(N-2) n_g \ge m\cdot n_g > m$ since $n_g\ge 2$. That is, we have $2 n_g < N n_g - m$ and
thus $n_g+1+n_h \leq N n_g - m$ as $|h|<|g|$.

Hence
$$
z_i = d_{i-1} z_i' d_i^{-1} = d_{i-1} z_{i+n_h}' d_i^{-1} = d_{i-1} d_{i+n_h-1}^{-1} z_{i+n_h} d_{i+n_h} d_{i}^{-1}
$$
for all $2 \leq i \leq n_g+1$
where indices in $z_i$ are taken mod $n_g$.
As $n_h$ and $n_g$ are coprime we see that $z_i \in C_{m+2} z_1 C_{m+2}$ for all $1 \leq i \leq n_g$. Combining with $z_i = d_{i-1} z_i' d_i^{-1}$ we have $z'_i \in C_{m+2} z_1 C_{m+2}$
for all $1 \le i \le n_h$.

If $z_1\in C_{m+2}$ then all $z_i, z_i' \in C_{m+2}$ and we achieve item (\ref{item: solve eqn claim case 2}) of the claim with $n = m$ and thus we are done.

Otherwise, let $c_{m+1} \in C_{m+1}\setminus C_{m+2}$ be the CM-representative of $C_{m+2} z_1 C_{m+2}$ provided by $\Icl(G)$. Using the fact that elements in $C_{m+2}$ commute with $c^{(m)}$, it follows that there are $y_i, y'_i \in C_{m+2}$ such that
\begin{eqnarray*}
d_g g d_g^{-1} &=& y_0 c^{(m+1)} y_1 \cdots c^{(m+1)} y_{n_g} \mbox{, and} \\
d_h h d_h^{-1}  &=& y'_0 c^{(m+1)} y'_1 \cdots c^{(m+1)} y'_{n_h}
\end{eqnarray*}
for $c^{(m+1)} = c^{(m)} c_{m+1}$. 

Conjugating $d_g g d_g^{-1}$ and $d_h h d_h^{-1}$ by $y_0$ and $y'_0$ respectively, we achieve item (\ref{item: solve eqn claim case 1}) of the claim for $m+1$ and thus the result follows.
\end{proof}

As $C<G$ is a BCMS-$D$ subgroup, by the claim above, there is some $n \leq D$, $d_g \in C$, and a CM-sequence $(c_0, \ldots, c_n)$ such that
$$
d_g g d_g^{-1} = c^{(n)} z_1  \cdots c^{(n)} z_{n_g}
$$ 
with $z_i \in C_{n+2}$ and $c^{(n)}=c_0\cdots c_n$. Thus all $z_i$ commute with $c^{(n)}$ and we have
$$
d_g g d_g^{-1} = \left( c^{(n)} \right)^{n_g} z
$$
with $z =  z_1 \cdots z_{n_g}$.
Let $x = d_g^{-1} c^{(n)} d_g$ and $c = d_g^{-1} z d_g$.
Then $g = x^{n_g} c$ and $c$ commutes with $x$. By construction we have $|x|=|c_0|=|g_1|=|g|/n_g$ and $|x^{n_g}|=|g|$, thus $x$ is cyclically reduced.
This finishes the proof of Proposition \ref{prop:sol to main equn}. 
\end{proof}

\subsection{Proof of Theorem \ref{thm:CM gap}}
\label{subsec:proof of the cm gap thm}

We use the following reduced form of integral chains to prove Theorem \ref{thm:CM gap}.
\begin{lemma}\label{lemma: reduced form for main thm}
	Let $G$ be a group that splits over a BCMS-$D$ subgroup $C$. Any integral chain $d$ is equivalent to a chain $d' = d_1 + d_2$ where
	\begin{enumerate}
		\item \label{item: ab alter} $d_1 = \sum_{i=1}^n g_i$ for some $n\ge 0$, where every $g_i$ is cyclically reduced (see Definitions \ref{def:ab-alternating} and \ref{def:reduced form}) and does not conjugate into any vertex group,
		\item \label{item: in vertex group} every term of $d_2$ lies in some vertex group,
		\item \label{item: thm no inverses} there is no $1 \leq i \leq j \leq n$ such that $g_i = {g'} c$ where $g'$ is a conjugate of $g_j^{-1}$ and $c \in C$ commutes with $g'$,
		\item \label{item: no power} there is no $1 \leq i \leq n$ such that $g_i = x^m c$ for some $m>1$, $x \in G$, and $c \in C$ so that $x$ and $c$ commute, and
		\item \label{item: CM_reduced} for every $1 \leq i \leq n$ we have that $g_i$ is CM-reduced (Definition \ref{def: CM reduced elts}).
	\end{enumerate}
\end{lemma}
\begin{proof}
	Given an expression $d'=d_1 + d_2$ of integral chains, where $d_1 = \sum_{j=1}^{m} k_j h_j$ with cyclically reduced words $h_j \in G$ and $k_j\in\Z_+$, and every term of $d_2$ lies in some vertex group, associate a complexity $n(d') = \sum_{j=1}^m | h_j |$.
	
	There exists a chain equivalent to $d$ that admits such an expression by replacing elements in $d$ by suitable conjugates so that they are either cyclically reduced or in a vertex group.
	
	Let $d' = d_1 + d_2$ with $d_1=\sum_{i=1}^n k_i g_i $ be an expression of this form for a chain equivalent to $d$ where $n(d')$ is minimal among such equivalent chains.
	We claim that $d'$ satisfies the conditions (\ref{item: ab alter})--(\ref{item: no power}).
	Each $g_i$ is cyclically reduced by our requirement, and the first two conditions are easy to verify.
	If there are $1 \leq i \leq j \leq n$ such that $g_i = g' c$ where $c$ commutes with $g'$ and $g'$ is conjugate to $g_j^{-1}$, then $g_i$ is equivalent to the chain $g' + c$ by (\ref{item: comm equiv}) of Definition \ref{def: equivalent chains} and equivalent to $-g_j + c$ by equivalence (\ref{item: linear equiv}) and (\ref{item: conj equiv}) of Definition \ref{def: equivalent chains}. Thus we may cancel $g_i$ and $g_j$ at the cost of changing $d_2$ until one term has coefficient zero to reduce $n(d')$. Similarly we see that if $g_i = x^m c$ where $m>1$ and $c$ commutes with $x$, then we may replace $k_i g_i$ by $mk_i x+c$, which has smaller complexity since $|x| < m|x|=|x^m| = |g_i|$.
	
	Finally we can always make the chain $d'$ above further satisfy (\ref{item: CM_reduced}): by Proposition \ref{prop: cm reduced} we may replace every (cyclically reduced) $g_i$ by $h_i + c_i$ where $h_i$ is CM-reduced, $c_i \in C$ lies in the edge group (and thus in a vertex group). Moreover, Proposition \ref{prop: cm reduced} shows that $g_i$ is conjugate to $h_i c_i$ by an element of $C$, so $h_i\in Cg_i C$ must be represented by a cyclically reduced word as $g_i$ is, and we have $|h_i|=|g_i|$. This operation does not affect the complexity of the expression and thus the chain $d'$ admits a desired expression.
\end{proof}

We can now prove Theorem \ref{thm:CM gap}:
\begin{reptheorem}{thm:CM gap}
Let $G$ be a graph of groups where each edge group is a BCMS-$D$ subgroup of $G$. Let $c$ be an integral chain in $G$. Then either $c$ is equivalent (Definition \ref{def: equivalent chains}) to a chain $\tilde{c}$ such that every term lies in a vertex group or
$$
\scl_G(c) \geq \frac{1}{12 (D+2)}.
$$
\end{reptheorem}
\begin{proof}
Fix a CM-subgroup choice $\Icl(G)$. Assume first that the graph of groups is either an amalgamated free product or an HNN extension over a BCMS-$D$ subgroup $C$.

Let $c' = c_1 + c_2$ be a chain equivalent to $c$ as in Lemma \ref{lemma: reduced form for main thm} with $c_1 = \sum_{i=1}^n g_i $.

Suppose $n>0$ and without loss of generality assume that $g_1$ has the longest length.
Set $N = D+2$ and suppose that
$$
\scl_G(c) < \frac{1}{12 N}.
$$
By Theorem \ref{thm: no long pairing} there is some $1 \leq j \leq n$ and a cyclic conjugate $h$ of $g_j^{-1}$ such that 
$$
g^N = h^k h' c,
$$
where $h'$ is a prefix of $h$ and $c \in C$.
Since $|g_1|$ is maximal among all $g_i$ we conclude that $|g| \geq |h|$.
Now consider two cases:
\begin{itemize}
\item $|g| = |h|$. Since all of $g$, $h$ and $h'$ are cyclically reduced, we must have $g^N = h^N c$ in this case. Since $g$ is CM-reduced, by Proposition \ref{prop:gmhm} there is some $z \in C$ which commutes with $g$ such that $g = h z$. This contradicts (\ref{item: thm no inverses}) of Lemma \ref{lemma: reduced form for main thm}.
\item $|g| > |h|$. In this case, Proposition \ref{prop:sol to main equn} implies that there is some $x \in G$, $m \geq 2$ and $c \in C$ such that $g = x^m c$. This contradicts  (\ref{item: no power}) of Lemma \ref{lemma: reduced form for main thm}.
\end{itemize}

Therefore we must have
$$
\scl_G(c) \geq \frac{1}{12 N} = \frac{1}{12(D+2)},
$$
unless $c$ is equivalent to a chain where all terms lie in vertex groups.

When $G$ is a general graph of groups, the chain is supported on a finite subgraph, so we can proceed by induction on the number of edges in the support. At each step, any chosen edge group $C$ splits the group as  an amalgamated free product or an HNN extension over $C$, depending on whether the edge separates the graph.
Note that any BCMS-$D$ edge subgroup of $G$ lying in a subgroup $H$ is also a BCMS-$D$ subgroup of $H$. 
Thus either at some stage what we have shown above implies the desired gap, or we can keep replacing the chain by equivalent ones supported in subgraphs with strictly smaller number of edges until every term lies in vertex groups.
\end{proof}

\section{Gaps for Graph Products of Groups} \label{sec: gap in graph prod of groups}

In this section we apply Theorem \ref{thm:CM gap} from the previous section to obtain gap results for graph products.
We will use basic notions and properties of graph products in Section \ref{sec:graph product of groups}.

The lower bounds of scl for integral chains depends on the existence of certain induced subgraphs. %, called \emph{opposite paths} of $\Gamma$.
Let $\Delta_n$ be the simplicial graph with vertex set $\Vrm(\Delta_n) = \{v_0, \ldots, v_n \}$ and edge set $\Erm(\Delta_n) = \{ (v_i, v_j): |i-j| \geq 2 \}$. 
We call this graph the \emph{opposite path of length $n$}. 
For any simplicial graph $\Gamma$ we define 
$$
\Delta(\Gamma) \defeq \max \{ n \mid \Delta_n \mbox{ is an induced subgraph of } \Gamma \}.
$$
%See Section \ref{subsec:opposite paths delta m}.
The lower bound we establish has size determined by $\Delta(\Gamma)$. 
The bound applies to all integral chains except for those equivalent (Definition \ref{def: equivalent chains}) to \emph{vertex chains}.
\begin{definition}
	A vertex chain is a chain of the form $c = \sum_{v \in \Vrm} c_v$, where each $c_v$ is a chain in the vertex group $G_v$.
\end{definition}

\begin{theorem}[Gaps for Graph Products of Groups] \label{thm: gap for graph products}
	Let $\Gcl(\Gamma)$ be a graph product and let $c$ be an integral chain of $\Gcl(\Gamma)$. Then either
	$$
	\scl_{\Gcl(\Gamma)}(c) \geq \frac{1}{12 (\Delta(\Gamma) + 2)},
	$$
	or one of the following equivalent statements holds:
	\begin{itemize}
		\item[(i)] $c$ is equivalent (Definition \ref{def: equivalent chains}) to a vertex chain,
		\item[(ii)] the pure factor chain $c^{\pf}$ (Definition \ref{def:pure factor chain}) is a vertex chain.
	\end{itemize} 
\end{theorem}

We will study vertex chains in detail in Section \ref{sec:vertex chains}. In particular, we prove the following theorem that computes the stable commutator length of a vertex chain $c = \sum_{v \in \Vrm} c_v$ in terms of $\scl_{G_v}(c_v)$ and the structure of the defining graph.
\begin{theorem}[Vertex chains]\label{thm: scl of vertex chains, beginning of sec}
	Let $\Gcl(\Gamma)$ be a graph product of groups and let $c=\sum_{v\in \Vrm(\Gamma)} c_v$ be a vertex chain, where each $c_v$ is a chain in the vertex group $G_v$.
	Then $\scl_{\Gcl(\Gamma)}(c)$ can be computed as a linear programming problem if each $\scl_{G_v}(c_v)$ is known, and it is rational if each $\scl_{G_v}(c_v)$ is.
	Moreover, 
	$$\scl_{\Gcl(\Gamma)}(c)\ge \scl_{G_v}(c_v)$$ 
	for any vertex $v$.
%	Then $\scl_{\Gcl(\Gamma)}(c)$ is the \emph{fractional stability number} (see Section \ref{sec:vertex chains}) of the graph $\Gamma$ with weight $\scl_{G_v}(c_v)$ at each vertex $v$, which
%	\begin{enumerate}
%		\item can be computed as a linear programming problem if each $\scl_{G_v}(c_v)$ is known,
%		\item is rational if each $\scl_{G_v}(c_v)$ is, and
%		\item is no less than $\scl_{G_v}(c_v)$ for any vertex $v$.
%	\end{enumerate}
\end{theorem}

See the end of Section \ref{subsec:computation by linear program} for a proof.

Combining with Theorem \ref{thm: gap for graph products}, we have:
\begin{corollary}\label{cor: gap pres by graph prod}
	Let $G=\Gcl(\Gamma)$ be a graph product of groups over a finite graph $\Gamma$, where each vertex group $G_v$ has a spectral gap $C_v>0$ for integral chains. Then $G$ also has a gap $C=\min\{\frac{1}{12 (\Delta(\Gamma) + 2)}, C_v\}$ for integral chains.
\end{corollary}

In particular, we have a gap theorem for RAAGs and RACGs; see Theorem \ref{thm: gap theorem for RAAGs and RACGs}.

We can also construct integral chains with small stable commutator length.
\begin{theorem}[Chains with small stable commutator length] \label{thm:small scl graph prod}
	Let $\Gcl(\Gamma)$ be a graph product of groups and let $\Delta(\Gamma)$ be as above. Then there is an explicit integral chain $\delta$ in $\Gcl(\Gamma)$ such that
	$$
	\frac{1}{12(\Delta(\Gamma) + 2)} \le \scl_\Gamma(\delta) \leq \frac{1}{\Delta(\Gamma)}.
	$$
\end{theorem}

This shows that the estimate in Theorem \ref{thm: gap for graph products} is accurate up to a scale of $12$.
%The authors do not know the precise size of the gap. In particular, the precise gap for free groups is still open.

This section is organized as follows. In Section \ref{subsec: canonical choice} we define the canonical CM-subgroup choice in a graph product $\Gcl(\Gamma)$ and show the nice behavior of CM-subgroup sequences with respect to this choice. 
In Section \ref{subsec:opposite paths delta m}, we show that the subgroup $\Gcl(\Lambda)$ associated to any induced subgraph $\Lambda\subset \Gamma$ has BCMS-$\Delta(\Gamma)$. 
In Section \ref{subsec:scl in opposite paths} we will see that opposite paths are sources of integral chains with small stable commutator length. 
Then we prove Theorems \ref{thm: gap for graph products} and \ref{thm:small scl graph prod} in Section \ref{subsec: proofs graph prod}. 
In Section \ref{subsec: gaps for RAAGs and RACGs} we deduce the gap results in the special case of RAAGs and RACGs.
Finally as applications, we construct groups with interesting scl spectra in Section \ref{subsec: scl spec on a delta inf}. 
%\ncomm{add here: applications to raags and racgs}

\subsection{Canonical CM-choice}\label{subsec: canonical choice}
Let $\Gcl(\Gamma)$ be a graph product of groups. Every induced subgraph $\Lambda \subset \Gamma$ 
induces a subgroup $\Gcl(\Lambda) < \Gcl(\Gamma)$. We find nice CM-representatives with respect to such subgroups.%, and use them to show that these subgroups are always BCMS-$M$ for some $M$.

\begin{lemma} \label{lemma: Lambda has property CM}
	Let $\Lambda \subset \Gamma$ be an induced subgraph of $\Gamma$, let $g \in \Gcl(\Gamma) \setminus \Gcl(\Lambda)$ and let $\bar{g}$ be the element with the shortest length among all elements in $\Gcl(\Lambda) g \Gcl(\Lambda)$.
	Then
	\begin{enumerate}
		\item $\bar{g}$ is a CM-representative, and \label{item: shortest is CM}
		\item the centralizer $Z_{\Gcl(\Lambda)}(\bar{g}) = \Gcl(\Theta)$ where $\Theta$ is the induced subgraph of $\Lambda$ that consists of all vertices of $\Lambda$ adjacent to all vertices in the support of $\bar{g}$.\label{item: centralizer}
	\end{enumerate}
\end{lemma}
\begin{proof}
	Let $\gb$ be a word of minimal syllable length in $\Gcl(\Lambda) g \Gcl(\Lambda)$. 
	Then $\gb$ is in particular reduced by Lemma \ref{lemma: normal form for graph products}.
	
	Suppose that there are some $h_1,h_2 \in \Gcl(\Lambda)$ such that $\bar{g} h_1 \bar{g}^{-1} = h_2^{-1}$. Then $h_2 \bar{g} h_1 = \bar{g}$. We may assume that $h_1, h_2$ are written as reduced words. By Lemma \ref{lemma: normal form for graph products} there are three cases:
	\begin{itemize}
		\item some letter in $h_2$ merges with another in $\gb$ and commutes with all the letters in between; 
		\item some letter in $h_1$ merges with another in $\gb$ and commutes with all the letters in between; or
		\item some letter in $h_2$ merges with another in $h_1$ and commutes with all the letters in between.
	\end{itemize}
	The first two cases can not occur by our choice of $\gb$ as we can remove the letter that merges with $h_1$ or $h_2$ in $\gb$. Thus we should keep having the last case until the word
	$h_1 \gb h_2$ reduces to $\gb$. The process implies that $h_1=h_2^{-1}$ and both commute with every letter of $\gb$. This shows that $\gb$ is a CM-representative.
	
	The observation above also implies that a reduced word $h \in \Gcl(\Lambda)$ commutes with $\gb$ if and only if every letter in it commutes with all those in $\gb$. This shows $Z_{\Gcl(\Lambda)}(\gb)= \Gcl(\Theta)$ as in (\ref{item: centralizer}). 
\end{proof}

As the minimal representatives in the double cosets yields nice and controlled centralizers, we always use them as our CM-choice in what follows. It is not important for our purposes but the method above shows that $\bar{g}$ is the unique element of minimal length in $\Gcl(\Lambda) g \Gcl(\Lambda)$.

\begin{definition}[canonical CM-choice] \label{def: canonical cm}
	Let $\Gamma$ be a simplicial graph and let $\Gcl(\Gamma)$ be a graph product of groups. We define the canonical CM-subgroup choice $\Icl(\Gcl(\Gamma))$ as follows:
	For any induced subgraph $\Lambda \subset \Gamma$ and any $g \in \Gcl(\Gamma)$ we choose $\gb$ a CM-representative of $g$ for $\Gcl(\Lambda) \le \Gcl(\Gamma)$ as an element with the smallest syllable length in $\Gcl(\Lambda) g \Gcl(\Lambda)$.
	For any other CM-subgroups we choose the CM-representatives arbitrarily.
\end{definition}

Note that for any induced subgraph $\Lambda$ of $\Gamma$, item (\ref{item: shortest is CM}) of Lemma \ref{lemma: Lambda has property CM} shows that $\Gcl(\Lambda) \le \Gcl(\Gamma)$ is a CM-subgroup. Moreover, under the canonical choice all CM-subgroup sequences have the form
$$
\Gcl(\Lambda_{n+2}) \le \Gcl(\Lambda_{n+1}) \le \dots \le \Gcl(\Lambda_1) \le \Gcl(\Gamma),
$$
where $n\ge0$ and $\Lambda_{n+2} \subset \cdots \subset\Lambda_1 = \Lambda$ is a proper nested sequence of induced subgraphs except that possibly $\Lambda_{n+2}=\Lambda_{n+1}$.
Thus either $\Gcl(\Lambda_{n+2}) = \Gcl(\Lambda_{n+1})$ or $\Gcl(\Lambda_{n+2})$ is a proper CM-subgroup by Lemma \ref{lemma: Lambda has property CM}.

Thus to show that $\Gcl(\Lambda) \le \Gcl(\Gamma)$ is a BCMS-$D$ subgroup, we need to control the length of CM-sequences with respect to the canonical choice. This is what we do in the next subsection.

%\begin{proposition} \label{prop: induced subgraphs property BCMS}
%	Let $\Gcl(\Gamma)$ be a graph product on a finite simplicial graph $\Gamma$ and let $\Lambda < \Gamma$ be an induced subgraph of $\Gamma$. Then $\Gcl(\Lambda) < \Gcl(\Gamma)$ has property BCMS-$M$ for some $M$.
%\end{proposition}
%\begin{proof}
%	Item (\ref{item: shortest is CM}) of Lemma \ref{lemma: Lambda has property CM} shows that $\Gcl(\Lambda) < \Gcl(\Gamma)$ is a CM-subgroup. 
%	By Proposition \ref{prop: bcms enough to check for subgroup choices} it is enough to check the BCMS-$M$ property for the canonical CM-subgroup choice. Using Lemma \ref{lemma: Lambda has property CM} we know that under this choice all CM-subgroup sequences have the form
%	$$
%	\Gcl(\Lambda_{n+2}) < \Gcl(\Lambda_{n+1}) < ... < \Gcl(\Lambda_1) < \Gcl(\Gamma),
%	$$
%	where $n\ge0$ and $\Lambda_{n+2} \subset \cdots \Lambda_1 = \Lambda$ is a proper nested sequence of induced subgraphs except that possibly $\Lambda_{n+2}=\Lambda_{n+1}$.
%	In particular such a sequence is finite since $\Gamma$ is. 
%	Thus $\Gcl(\Lambda) < \Gcl(\Gamma)$ has property BCMS-$M$.
%\end{proof}

%Observe that from the proof we can crudely estimate that $M \leq |\Vrm(\Gamma)|-1$ for any subgroup $\Gcl(\Lambda) < \Gcl(\Gamma)$.

\subsection{The opposite paths $\Delta_m$ and lengths of CM-sequences} \label{subsec:opposite paths delta m}
Now we find the maximal length of CM-sequences in a given graph product on a graph $\Gamma$ with respect to the canonical CM-choice.
Then we show that the subgroup associated to any induced subgraph of $\Gamma$ is BCMS-$D$ for $D=\Delta(\Gamma)$.

Recall that for a graph $\Gamma$, we define $\Delta(\Gamma)$ to be the largest number $m\in\Z_+$ such that $\Delta_m$ is an induced subgraph of $\Gamma$. 
The only graphs where $\Delta_1$ does not embed as an induced subgraph are complete graphs (including the graph with a single vertex).
We set $\Delta(\Gamma) = 0$ if $\Gamma$ is a complete graph. %\ncomm{set $\Delta$ on the complete graph to be $0$.}
If all $\Delta_m$ are induced subgraphs of $\Gamma$, then $\Gamma$ is necessarily infinite, and we set $\Delta(\Gamma)=\infty$.

For example we see that $\Delta(\Delta_m)=m$. Observe also that $\Delta(\Gamma) \leq |\Gamma|-1$. % unless $\Gamma$ is complete.
We will see that $\Delta(\Gamma)$ controls the length of the longest CM-sequence in subgroups of $\Gcl(\Gamma)$ associated to induced subgraphs.

On the one hand, for arbitrary nontrivial vertex groups, a graph product on the graph $\Delta_n$ has a CM-subgroup sequence of length $n+1$.
For $n \in \Z_+$ and $i \in \{1, \ldots, n \}$ let $\Delta^i_n$ be the induced subgraph of $\Delta_n$ with vertex set
$$
V(\Delta^i_n) = \{ v_i, \ldots, v_n \}.
$$
%By Lemma \ref{lemma: Lambda has property CM} we have know $\Gcl(\Delta^1_n) < \Gcl(\Delta_n)$ has property BCMS.
For arbitrary nontrivial elements $g_i\in G_{v_i}$, we have a CM-sequence $(g_0, \ldots, g_n)$ of length $n+1$, and the associated CM-subgroup sequence is 
$$
\{ e \} \le \{ e \} \le \Gcl(\Delta_n^n) \le \cdots \le \Gcl(\Delta_n^1) \le \Gcl(\Delta_n).
$$

On the other hand, we can find an induced subgraph isomorphic to some $\Delta_m$ from a CM-sequence. 

\begin{lemma}\label{lemma: CM length and delta subgraphs}
	Let $\Gamma_0$ be a graph and let $\Gamma_1 \subset \Gamma_0$ be an induced proper subgraph. Fix arbitrary nontrivial vertex groups to form a graph product $\Gcl(\Gamma_0)$. 
	For the canonical CM-choice, let $(c_0, \ldots, c_{m})$ be a CM-sequence with respect to $\Gcl(\Gamma_1) < \Gcl(\Gamma_0)$ of length $m+1$, 
	and let $C_{m+2} \le \cdots \le C_0$ be the associated CM-subgroup sequence. 
	Then there is an induced subgraph $\Delta_m$ of $\Gamma$.
\end{lemma}

To prove Lemma \ref{lemma: CM length and delta subgraphs}, we first observe some basic relationship between the graphs defining the subgroups $C_i$ and those supporting $c_i$. Recall that $C_{i+2}= Z_{C_{i+1}}(c_i)$ for all $0\le i\le m$.

\begin{lemma}\label{lemma: properties of Lambda and Gamma}
	In the setting of Lemma \ref{lemma: CM length and delta subgraphs}, there are induced subgraphs $\Gamma_{m+2} \subset \cdots \subset \Gamma_1 \subset \Gamma_0$ 
	such that for each $0 \leq i \leq m$ there is an induced subgraph $\Lambda_i \subset \Gamma_i$ with the following properties:
	\begin{enumerate}
		\item $\Lambda_i$ is the induced subgraph on the support of $c_i$ for all $0\le i\le m$, \label{item: induct Lambda}
		\item For each $0\le i\le m$, $\Gamma_{i+2}$ is the induced subgraph consisting of vertices in $\Gamma_{i+1}$ adjacent to all those in $\Lambda_i$, \label{item: induct Gamma}
		\item $C_{i} =\Gcl(\Gamma_{i})$ for all $0\le i\le m+2$, \label{item:induct Gamma property}
		\item $\Lambda_i\setminus \Gamma_{i+1}\neq\emptyset$ for any $0\le i\le m$, and \label{item:induct Lambda property}
		\item $\Lambda_i \subset \Gamma_i \setminus \Gamma_{i+2}$ for every $0 \leq i \leq m$. \label{item: lambda_i not in gamma i+2}
	\end{enumerate}
\end{lemma}
\begin{proof}
	Bullet (\ref{item:induct Gamma property}) holds for $i \in \{0,1 \}$ by definition. Now we consider $i\ge 2$.
	Inductively from $i=2$ to $i=m+2$, we take bullet (\ref{item: induct Lambda}) as the definition of $\Lambda_i$, based on which we define $\Gamma_{i+2}$ as in bullet (\ref{item: induct Gamma}). Then 
	$\Lambda_i\subset \Gamma_i$ and $\Gamma_{i+2}\subset \Gamma_{i+1}$ by definition, and bullet (\ref{item:induct Gamma property}) follows from Lemma \ref{lemma: Lambda has property CM}. Then bullet
	(\ref{item:induct Lambda property}) holds since $c_i\notin C_{i+1}$ (as a CM-representative).
	
	To see bullet (\ref{item: lambda_i not in gamma i+2}) recall that every vertex of $\Gamma_{i+2}$ is adjacent to all vertices in $\Lambda_i$. If a reduced expression of $c_i$ contains a letter in $G_v$ for some $v\in \Gamma_{i+2}$, then we can shuffle it to the end of $c_i$, contradicting to the choice of $c_i$.
\end{proof}

Lemma \ref{lemma: CM length and delta subgraphs} follows from the case $i=m$ in Lemma \ref{lemma:delta subgraph of gamma} below, which is stated in a way to suit its proof by induction. In below, we say the induced subgraph of $\Gamma_0$ on a sequence of (distinct) vertices $(v_0,\dots, v_i)$ is isomorphic to $\Delta_i$ as \emph{labeled graphs} if $v_j$ and $v_k$ are adjacent in $\Gamma_0$ if and only if $|j-k|\ge2$. 

\begin{lemma}\label{lemma:delta subgraph of gamma}
	In the setup of Lemmas \ref{lemma: CM length and delta subgraphs} and Lemma \ref{lemma: properties of Lambda and Gamma}, for each $1\le i\le m$, there is a sequence of distinct vertices $V_i =(v_0, \ldots, v_i)$ of $\Gamma_0$ such that
	\begin{itemize}
		\item $v_1, \ldots, v_i \in \Gamma_{m-i+1}$,
		\item $v_0 \in \Lambda_{m-i} \setminus \Gamma_{m-i+1}$,
	\end{itemize}
	and the induced subgraph of $\Gamma_0$ on $V_i$ is isomorphic to $\Delta_i$ as labeled graphs.
\end{lemma}
\begin{proof}
	We show this lemma by induction on $i$. First consider the base case $i=1$. 
	Let $u$ be an arbitrary vertex in $\Lambda_m\setminus\Gamma_{m+1}$, which exists by bullet (\ref{item:induct Lambda property}) of Lemma \ref{lemma: properties of Lambda and Gamma}. 
	There are two possibilities:
	\begin{itemize}
		\item If $\Lambda_{m-1}\cap \Gamma_{m}=\emptyset$, there is some $v_0\in \Lambda_{m-1}$ not adjacent to $u$ since $u\notin \Gamma_{m+1}$; See bullet (\ref{item: induct Gamma}) of Lemma \ref{lemma: properties of Lambda and Gamma}. 
		Then $v_0\in \Lambda_{m-1}\setminus \Gamma_m=\Lambda_{m-1}$ and 
		$V_1=(v_0,u)$ satisfies the desired properties.
		
		\item If $\Lambda_{m-1}\cap \Gamma_{m}\neq \emptyset$, write $c_{m-1}$ as a reduced word and let $g_{v_1}$ be the
		last letter in $c_{m-1}$ that is supported on some $v_1\in \Gamma_m$. Then there must be some letter $g_{v_0}$ in $c_{m-1}$ supported on $v_0\in \Lambda_{m-1}$
		appearing after $g_{v_1}$ such that $v_0$ and $v_1$ are not adjacent, since otherwise we can shuffle $g_{v_1}$ all the
		way to the end of $c_{m-1}$ contradicting the fact that $c_{m-1}$ has the shortest syllable length in
		$C_m c_{m-1} C_m$ and $C_m=\Gcl(\Gamma_m)$. Note that $v_0\notin \Gamma_{m}$ since $g_{v_1}$ is the last
		letter on a vertex in $\Gamma_m$. Thus $V_1=(v_0,v_1)$ satisfies the desired properties.
	\end{itemize}
	
	Suppose the lemma holds for some $1\le i\le m-1$ with a sequence of vertices $V_i =( v_0, \ldots, v_i)$. The simplest attempt to obtain $V_{i+1}$ is to add a suitable vertex $w_0$ at the beginning of $V_i$.
	Since $v_0 \notin \Gamma_{m-i+1}$, there is some vertex $w_0$ in $\Lambda_{m-i-1}$ that is not adjacent to $v_0$.
	Note that $v_0\notin \Lambda_{m-i-1}$ since otherwise it must be adjacent to all vertices in $\Gamma_{m-i+1}$ and in 
	particular to $v_1$, contradicting the induction hypothesis. 
	Combining with $v_s\in \Gamma_{m-i+1}$ for $s\ge1$ and $\Gamma_{m-i+1}\cap \Lambda_{m-i-1}=\emptyset$ 
	by bullet (\ref{item: lambda_i not in gamma i+2}) of Lemma \ref{lemma: properties of Lambda and Gamma},
	all vertices $w_\ell\in \Lambda_{m-i-1}$ we construct below are distinct from those in $V_i$.
	
	Ideally we would like to choose $w_0$ above so that it lies in $\Lambda_{m-i-1} \setminus \Gamma_{m-i}$, in which case $V_{i+1}\defeq (w_0, v_0,v_1,\ldots,v_i)$ is a desired sequence:
	Observe that $w_0\in \Lambda_{m-i-1}$ is adjacent to all $v_1, \ldots, v_i \in \Gamma_{m-i+1}$ but not to $v_0$.
	
	The remaining (harder) case is when every vertex in $\Lambda_{m-i-1} \setminus \Gamma_{m-i}$ is adjacent to $v_0$. In this case we show the following claim to construct another sequence $W$ of vertices so that the concatenated sequence $(W,V)$ has the desired properties once we cut it down to have exactly $i+1$ vertices by removing some vertices in the tail. The vertices in $W$ are listed in reverse order to reflect the order they appear in the inductive process below.
	
	\begin{claim} \label{claim:existence of W}
		There is a sequence
		$W=(w_k,\ldots,w_0)$ of vertices for some $k\ge1$ such that
		\begin{itemize}
			\item $w_0,\ldots,w_{k-1}\in \Lambda_{m-i-1}\cap\Gamma_{m-i}$, 
			\item $w_k\in \Lambda_{m-i-1} \setminus \Gamma_{m-i}$, 
			\item the induced subgraph on $W$ is isomorphic to $\Delta_k$ as labeled graphs, that is, for $0 \leq s < t \leq k$ the vertices $w_s$ and $w_t$ are adjacent in $\Gamma_0$ if and only if $|s-t| \geq 2$,
			\item $v_0$ is adjacent to $w_\ell$ iff $\ell> 0$.
		\end{itemize}
	\end{claim}
	\begin{proof}[Proof of Claim \ref{claim:existence of W}]
		By our assumption, there is some $w_0\in \Lambda_{m-i-1}\cap \Gamma_{m-i}$ not adjacent to $v_0$. 
		Choose $g_{w_0}$ to be the last letter on $c_{m-i-1}$ supported on a vertex $w_0$ with this property. 
		Now inductively we can find letters $g_{w_1},\ldots, g_{w_k}$ of $c_{m-i-1}$ supported on vertices $w_1,\ldots,w_k \in \Lambda_{m-i-1}$ such that 
		\begin{itemize}
			\item for each $1\le \ell\le k$, $g_{w_\ell}$ is the last letter on $c_{m-i-1}$ after $g_{w_{\ell-1}}$ such that $w_\ell$ is not adjacent to $w_{\ell-1}$, 
			\item $w_\ell\in\Lambda_{m-i-1}\cap\Gamma_{m-i}$ for all $\ell<k$, and
			\item $w_k\in \Lambda_{m-i-1} \setminus \Gamma_{m-i}$.
		\end{itemize}
		We are guaranteed to end up with some $w_k \notin \Gamma_{m-i}$: if $w_k\in \Gamma_{m-i}$, $g_{w_k}$ cannot
		commute with all letters after it on $c_{m-i-1}$ by the minimality of $c_{m-i-1}$, so we
		can continue the sequence by adding the last letter $g_{w_{k+1}}$ on $c_{m-i-1}$ after $g_{w_{k}}$ with the property that $w_{k+1}$ is not adjacent to $w_k$.
		
		Then by construction $W=(w_k,\ldots,w_0)$ consists of distinct vertices and the corresponding induced subgraph in $\Gamma_0$ is isomorphic to $\Delta_{k}$ as labeled graphs. 
		By our choice of $w_0$, we see $w_\ell$ is adjacent to $v_0$ iff $\ell>0$.
		This constructs the desired sequence $W$ in Claim \ref{claim:existence of W}.
	\end{proof}
	
	Now we finish the proof of Lemma \ref{lemma:delta subgraph of gamma}.
	By Claim \ref{claim:existence of W}, for all $0\le \ell\le k$, $w_\ell$ is adjacent to $v_1, \ldots, v_i$ as $w_\ell \in \Lambda_{m-i-1}$ and $v_1, \ldots, v_i \in \Gamma_{m-i+1}$.
	Then for the concatenated sequence $\widetilde{V}_{i+1}\defeq(W,V)$, its corresponding induced subgraph of $\Gamma_0$ is isomorphic to $\Delta_{i+k+1}$ as labeled graphs, and all vertices lie in $\Gamma_{m-i}$ except that the first vertex $w_k$ lies in $\Lambda_{m-i-1}\setminus\Gamma_{m-i}$.
	Thus by taking the first $i+2$ vertices in the sequence $\widetilde{V}_{i+1}$ as our $V_{i+1}$, this finishes the inductive proof of Lemma \ref{lemma:delta subgraph of gamma}.
\end{proof}
	
Now we deduce Lemma \ref{lemma: CM length and delta subgraphs} from Lemma \ref{lemma:delta subgraph of gamma}.
\begin{proof}[Proof of Lemma \ref{lemma: CM length and delta subgraphs}]
	The case of $i=m$ in Lemma \ref{lemma:delta subgraph of gamma} implies that the induced subgraph of $\Gamma_0$ with vertex set $V_m = (v_0, \ldots, v_m)$ is $\Delta_m$.
\end{proof}

\begin{proposition}\label{prop: BCMS-M by opposite length}
	Let $\Gamma$ be a simplicial graph where $D\defeq \Delta(\Gamma)<\infty$. Let $\Gcl(\Gamma)$ be a graph product on $\Gamma$ with arbitrary fixed nontrivial vertex groups.
	Then for any induced subgraph $\Lambda$ of $\Gamma$, the subgroup $\Gcl(\Lambda)$ is has property BCMS-$D$.
\end{proposition}
%\begin{corollary}\label{cor: BCMS-M by opposite length}
%	Let $\Gamma$ be a simplicial graph where $M\defeq \Delta(\Gamma)<\infty$. Let $\Gcl(\Gamma)$ be a graph product on $\Gamma$ with arbitrary fixed nontrivial vertex groups.
%	Then for any induced subgraph $\Lambda$ of $\Gamma$, the subgroup $\Gcl(\Lambda)$ is has property BCMS-$M$.
%\end{corollary}
\begin{proof}
	It is enough to check the BCMS-$D$ property using the canonical CM-subgroup choice by Proposition \ref{prop: bcms enough to check for subgroup choices}.
	As we explained at the end of Section \ref{subsec: canonical choice}, 
	it suffices to control the length of CM-sequences.
	By Lemma \ref{lemma: CM length and delta subgraphs}, for any CM-sequence $(c_0,c_1,\cdots,c_m)$, there is an induced subgraph of $\Gamma$ isomorphic to $\Delta_m$. Thus by definition $D=\Delta(\Gamma)\ge m$. Hence $\Gcl(\Lambda)$ is a BCMS-$D$ subgroup.
\end{proof}

\subsection{Stable commutator length in opposite paths} \label{subsec:scl in opposite paths}
Let $\Delta_m$ be the opposite path on the vertices $\{ v_0, \ldots, v_m \}$ as described above.
In this section we will see that for any (nontrivial) vertex groups $(G_v)_{v \in \Vrm(\Delta_m)}$ the associated graph product $\Gcl(\Delta_m)$ has an integral chain with small stable commutator length.
Choose a nontrivial element $g_i \in G_{v_i}$ for every vertex $v_i$ of $\Delta_m$. 
For any $m\ge2$, define a chain $\delta_m$ in $\Gcl(\Delta_m)$ as
$$
\delta_m \defeq g_{0,m} - g_{0,m-1} - g_{1,m} + g_{1,m-1}.
$$
where $g_{i,j} \defeq g_i \cdots g_j$.

The following computation leads to an upper bound of $\scl(\delta_m)$.
\begin{lemma} \label{lemma: the c_m}
	Given $m\ge2$ and $0\le i\le m$, for every $1 \leq j \leq m-i+1$ we have
	$$
	g_{i,m}^j = g_{i,m-1}^j c_j,
	$$
	where $c_j$ is recursively defined as follows: $c_1 = g_m$ and for $1 \leq j \leq m-i$
	$$
	c_{j+1} := g_{m-j,m-1}^{-1} c_j g_{m-j, m }.
	$$
\end{lemma} 
\begin{proof}
	We proceed by induction.
	For $j=1$ the result is obvious.
	Suppose the conclusion holds for some $j \in \{1, \ldots, m-i \}$. 
	Then
	$$
	g_{i,m}^{j+1} = g_{i,m}^j \cdot g_{i,m} = g_{i,m-1}^j c_j g_{i,m}
	$$
	Since $c_j$ commutes with all $g_i, \ldots, g_{m-j-1}$ as it is a product of terms $g_k$ for $k \geq m-j+1$,
	we see that
	\begin{eqnarray*}
		g_{i,m}^{j+1} &=& g_{i,m-1}^j c_j \cdot g_{i,m}\\
		&=& g_{i,m-1}^j g_{i,m-j-1} c_j g_{m-j,m} \\
		&=&
		g_{i,m-1}^{j+1} g_{m-j,m-1}^{-1}  c_j g_{m-j, m}  \\
		&=&
		g_{i,m-1}^{j+1} c_{j+1}.
	\end{eqnarray*}
\end{proof}

\begin{proposition} \label{prop: delta_m has low scl}
	Let $m\ge2$ and $\delta_m$ be the chain in $\Gcl(\Delta_m)$ defined as above. Then
	$$
	\frac{1}{12(m+2)} \leq \scl_{\Gcl(\Delta_m)}(\delta_m) \leq \frac{1}{m}.
	$$
\end{proposition}

\begin{proof}  
	By Lemma \ref{lemma: the c_m} we have $g_{i,m}^m = g_{i,m-1}^m c_m$ for $i \in \{0,1\}$, and $c_m$ does not depend on $i$.
	Thus by Lemma \ref{lemma: pants chain} we have
	$$
	\scl(g_{i,m}^m - g_{i,m-1}^m - c_m) \leq \frac{1}{2}
	$$
	for $i \in \{0,1\}$. Therefore,
	\begin{eqnarray*}
		\scl (m \cdot \delta_m) &=& \scl \left( g_{0,m}^m - g_{0,m-1}^m - g_{1,m}^m + g_{1,m-1}^m \right) \\
		&\leq & \scl \left( (g_{0,m}^m - g_{0,m-1}^m - c_m) - (g_{1,m}^m - g_{1,m-1}^m - c_m) \right) \leq 1
	\end{eqnarray*}
	by the triangle inequality. Hence we conclude that 
	$$
	\scl(\delta_m) \leq \frac{1}{m}.
	$$
	
	On the other hand we see that $\Delta(\Delta_m)=m$, and thus by Theorem \ref{thm: gap for graph products} (proved below) 
	we have $\scl (\delta_m) \geq \frac{1}{12(m+2)}$, since $\delta_m$ is already a pure factor chain that is not a vertex chain. 
	This finishes the proof.
\end{proof}

\subsection{Proofs of Theorems \ref{thm: gap for graph products} and Theorem \ref{thm:small scl graph prod}} \label{subsec: proofs graph prod}

We now prove Theorems \ref{thm: gap for graph products} and \ref{thm:small scl graph prod}.

We first prove the following lemma dealing with an essential part of Theorem \ref{thm: gap for graph products}.
\begin{lemma} \label{lemma:small chains equiv to zero}
	Fix an integer $D \geq 1$. If a graph $\Gamma$ satisfies $\Delta(\Gamma) \leq D$, then every integral chain $c$ in $\Gcl(\Gamma)$ either has $\scl_{\Gcl(\Gamma)}(c) \geq \frac{1}{12(D+2)}$ or is equivalent to an integral vertex chain.
\end{lemma}
\begin{proof}
	For any integral chain $c=\sum_i g_i$, define its support $\supp(c)$ to be the union of $\supp(g_i)$. 
	Let $\Lambda$ be the induced subgraph of $\Gamma$ on $\supp(c)$, which is finite and $\Delta(\Lambda)\le \Delta(\Gamma)$ by definition. 
	We may reduce the assertion to the case $\Gamma=\Lambda$ as follows.
	Note that $\scl_{\Gcl(\Lambda)}(c)=\scl_{\Gcl(\Gamma)}(c)$ since $\Gcl(\Lambda)$ is a retract of $\Gcl(\Gamma)$.
	If $c$ is not equivalent to a vertex chain in $\Gcl(\Gamma)$, neither is it as a chain in $\Gcl(\Lambda)$.
	Hence it suffices to prove the lemma assuming $\Gamma$ to be a finite graph.
	
	We proceed by induction on the size $|\Gamma|$.
	The assertion trivially holds when $|\Gamma|=1$ since $c$ must be a vertex chain in this case.
	
	Suppose for some $n\ge1$ the assertion holds for all integral chains $c$ in any graph product $\Gcl(\Gamma)$ with $|\Gamma|\le n$. 
	Consider an integral chain $c$ in some graph product $\Gcl(\Gamma)$ with $|\Gamma|=n+1$ and $\Delta(\Gamma)\le D$ such that $c$ is not equivalent to a vertex chain.
	We need to show 
	$$\scl_{\Gcl(\Gamma)}(c)\ge\frac{1}{12(D+2)}.$$

	Pick any vertex $v$ in $\Gamma$. If $\Gamma = \st(v)$, where $\st(v)$ denotes the star of $v$, then $\Gcl(\Gamma) = G_v \times \Gcl(\Lk(v))$, where $\Lk(v)$ denotes the link of $v$. 
	Then by Proposition \ref{prop:chains on direct products}, $c$ is equivalent to a sum of integral chains $c_v + c'$,
	where $c_v$ is supported on $G_v$ and $c'$ is supported on $\Lk(v)$. 
	Here $c'$ cannot be equivalent to a vertex chain since $c$ is not.
	Note that $\Delta(\Lk(v))\le \Delta(\Gamma)\le D$ since $\Lk(v)$ is an induced subgraph of $\Gamma$.
	Thus by the induction hypothesis and Proposition \ref{prop:chains on direct products}, 
	we have $\scl_{\Gcl(\Gamma)}(c)\ge \scl_{\Gcl(\Lk(v))}(c')\ge \frac{1}{12(D+2)}$. 
	
	Now assume $\Gamma \neq \st(v)$. Then $\Gcl(\Gamma)$ splits non-trivially as an amalgam $\Gcl(\Gamma) = \Gcl(\st (v)) \star_{\Gcl(\Lk(v))} \Gcl(\Gamma \setminus v)$.
	We know that $\Gcl(\Lk(v)) < \Gcl(\Gamma)$ is a BCMS-$\Delta(\Gamma)$ subgroup by Proposition \ref{prop: BCMS-M by opposite length}. 
	Thus, by Theorem \ref{thm:CM gap}, it suffices to consider the case where $c$ is equivalent to an integral chain $\tilde{c}$ 
	such that every term of $\tilde{c}$ lies in $\Gcl(\st(v))$ or $\Gcl(\Gamma \setminus v)$. 
	We may again split every term supported on $\Gcl(\st(v))$ into terms in $G_v$ and in $\Gcl(\Lk(v)) < \Gcl(\Gamma \setminus v)$.
	Thus $\tilde{c}$ and $c$ are equivalent to a chain $c' + c_v$, where $c'$ is supported on $\Gamma \setminus v$ and $c_v$ is supported on $G_v$.
	By the monotonicity of scl for the retraction $\Gcl(\Gamma)\to \Gcl(\Gamma \setminus v)$, we deduce that 
	$$\scl_{\Gcl(\Gamma)}(c)=\scl_{\Gcl(\Gamma)}(c'+c_v)\ge \scl_{\Gcl(\Gamma \setminus v)} (c').$$ 
	As $c'$ is not equivalent to a vertex chain since $c$ is not, we have $\scl_{\Gcl(\Gamma \setminus v)} (c')\ge \frac{1}{12(D+2)}$ by the induction hypothesis.
\end{proof}

\begin{proof}[Proof of Theorem \ref{thm: gap for graph products}]
	Let $c$ be an integral chain in a graph product $\Gcl(\Gamma)$.
	By Lemma \ref{lemma:small chains equiv to zero}, either $\scl_{\Gcl(\Gamma)}(c) \geq \frac{1}{12(\Delta(\Gamma)+2)}$ or $c$ is equivalent to a vertex chain. 
	By Proposition \ref{prop: pure factor chain}, $c$ is equivalent to such a vertex chain 
	if and only if $c^{\pf}$ is a vertex chain.
\end{proof}

\begin{proof}[Proof of Theorem \ref{thm:small scl graph prod}]
	For any graph $\Gamma$ and a graph product $\Gcl(\Gamma)$ on $\Gamma$, the inclusion $i_m \col \Gcl(\Delta_m) \to \Gcl(\Gamma)$ is a retract, where $m = \Delta(\Gamma)$. By Proposition \ref{prop: delta_m has low scl} there is an integral chain $\delta_m$ in $\Gcl(\Delta_m)$ such that
	$$
	\frac{1}{12(m+2)} \leq \scl_{\Gcl(\Delta_m)}(\delta_m) \leq \frac{1}{m}.
	$$
	Since a chain in the retract has the same scl as in the whole group (Proposition \ref{prop: mono and retract}), we conclude that $\delta = i_m(\delta_m)$ has the same property. This concludes the proof.
\end{proof}

\subsection{Applications to right-angled Artin Groups and right-angled Coxeter groups}\label{subsec: gaps for RAAGs and RACGs}

Our gap theorems can be simplified in the case of right-angled Artin Groups and right-angled Coxeter groups.

\begin{theorem}[RAAGs and RACGs]\label{thm: gap theorem for RAAGs and RACGs}
	Let $G$ be the right-angled Artin (or Coxeter) group with defining graph $\Gamma$. Then for any integral chain $c$ not equivalent to the zero chain, we have
	$\scl_G(c)\ge \frac{1}{12(\Delta(\Gamma)+2)}$.
\end{theorem}
\begin{proof}
	Note that any null-homologous chain of the form $\sum_v c_v$ in $G$ is equivalent to the zero chain since each vertex group is abelian, where each $c_v$ is a chain in the vertex group $G_v$.
	Thus the result follows from Theorem \ref{thm: gap for graph products}.
\end{proof}

By Theorem \ref{thm:small scl graph prod}, the gap above cannot be uniform in the class of RAAGs, although there is uniform gap $1/2$ for elements in RAAGs \cite{Heuer}.
It is natural to ask whether this holds analogously for RACGs.
\begin{question}
	Is there a uniform spectral gap for elements in RACGs?
\end{question}
Note that there is a uniform gap theorem \cite[Theorem F]{CH:sclgap} for elements in many graph products, but it does not apply to RACGs because of the existence of $2$-torsion.
However, we are able to characterize elements in RACGs with zero scl.

\begin{corollary}\label{cor: RACG gap for elements}
	Let $G$ be the right-angled Coxeter group with defining graph $\Gamma$. Then For any element $g\in G$, we have either $\scl_G(g)\ge \frac{1}{12(\Delta(\Gamma)+2)}$ or $\scl_G(g)=0$. 
	Moreover, the latter case occurs if and only if $g$ is conjugate to $g^{-1}$, or more precisely, $g=ab$ with $a^2=id$ and $b^2=id$.
\end{corollary}
\begin{proof}
	The first assertion directly follows from Theorem \ref{thm: gap theorem for RAAGs and RACGs}. 
	It also implies that $\scl_G(g)=0$ if and only if $g$ is equivalent to the zero chain. Hence by Proposition \ref{prop: element zero chain in RACG}, we obtain the more explicit characterization of such $g$.
	%If $g^2=id$, then clearly $\scl_G(g)=0$.
%	Conversely, if $\scl_G(g)=0$, then its pure factor chain $g^{pf}$ is a vertex chain by Theorem \ref{thm: gap for graph products}.
%	This means by Definition \ref{def:pure factor chain} that the pure decomposition of $g$ is given by
%	$$g=p\cdot \gamma_1\cdots \gamma_\ell\cdot p^{-1},$$
%	where each $\gamma_i$ is supported on a vertex $v_i$ so that $\{v_1,\cdots,v_\ell\}$ is a clique of size $\ell$ in $\Gamma$.
%	Note that $\gamma_i=\gamma_i^{-1}$ since each vertex group is $\Z/2$. As they commute with each other, we have
%	$$g^{-1}=p(\gamma_1\cdots \gamma_\ell)^{-1}p^{-1}=p\cdot\gamma_1^{-1}\cdots \gamma_\ell^{-1}\cdot p^{-1}=p\cdot\gamma_1\cdots \gamma_\ell\cdot p^{-1}=g.$$
\end{proof}
One can similarly characterize elements with zero scl in other graph products if elements with zero scl are understood in vertex groups.
%We will prove at the end of this section a uniform gap $1/12$ for elements in RACGs; see Theorem \ref{thm: uniform gap for elements in RACGs}.

We also get a uniform gap for integral chains if we add a hyperbolicity assumption. 
Since the only hyperbolic RAAGs are free groups, we focus on hyperbolic RACGs below.

\begin{corollary}\label{cor: hyp RACGs}
	Let $G=\Crm(\Gamma)$ be a hyperbolic right-angled Coxeter group. Then $\scl_G(c) \geq \frac{1}{60}$ for any integral chain not equivalent to the zero chain.
\end{corollary}
\begin{proof}
	It is known by \cite{hypRACGs} that $\Crm(\Gamma)$ is hyperbolic if and only if the graph $\Gamma$ has no induced subgraph isomorphic to the cyclic graph of length $4$.
	Note that the graph $\Delta_4$ contains such an induced subgraph with vertices $v_0,v_1,v_3,v_4$. Thus $\Delta(\Gamma)\le 3$ if $\Crm(\Gamma)$ is hyperbolic.
	Hence the result follows from Theorem \ref{thm: gap theorem for RAAGs and RACGs}.
\end{proof}

Based on this, we make the following conjecture.
\begin{conjecture}\label{conj: hyp special}
	There is a uniform constant $B>0$ such that any hyperbolic $C$-special (or $A$-special, see \cite{HaglundWise} for definitions) group has a spectral gap $B$ for integral chains.
\end{conjecture}

If the conjecture holds true, one can use it and the index formula (Proposition \ref{prop: index formula}) to establish effective lower bounds for the index of special subgroups in hyperbolic groups. For instance, it is a well-known theorem that every hyperbolic $3$-manifold group contains a finite index subgroup that is special (and hyperbolic) \cite{Agol}, but it is unknown whether the index has a uniform upper bound independent of the manifold. This connection was suggested to us by Danny Calegari and motivated this work on scl of integral chains in RAAGs, but we did not anticipate the spectral gap to be non-uniform.

One can also bound $\Delta(\Gamma)$ in terms of other invariants of the graph $\Gamma$.
\begin{corollary}
	If $\Gamma$ is a simplicial graph where each vertex has valence at most $m\ge0$, then integral chains in $\Arm(\Gamma)$ and $\Crm(\Gamma)$ have a gap
	$\frac{1}{12(m+3)}$.
\end{corollary}
\begin{proof}
	Note that in $\Delta_{m+2}$ the vertex $v_0$ is adjacent to $m+1$ vertices $v_2,v_3,\cdots,v_{m+2}$. Thus we must have $\Delta(\Gamma)\le m+1$. 
	We conclude by Theorem \ref{thm: gap theorem for RAAGs and RACGs}.
\end{proof}

The dimension of a right-angled Artin (resp. Coxeter) group $\Arm(\Gamma)$ (resp. $\Crm(\Gamma)$) associated to some simplicial graph $\Gamma$ 
is the largest size of cliques in $\Gamma$.
\begin{corollary}
	Any right-angled Artin (resp. Coxeter) group $G=\Arm(\Gamma)$ (resp. $G=\Crm(\Gamma)$) of dimension at most $d$ has
	a gap $\frac{1}{12(2d+1)}$ for integral chains.
\end{corollary}
\begin{proof}
	Note from the definition that $\Delta_{2d}$ contains a clique of size $d+1$ with vertices $v_0,v_2,\ldots, v_{2d}$.
	Thus $\Delta(\Gamma)\le 2d-1$, and the result follows from Theorem \ref{thm: gap theorem for RAAGs and RACGs}.
\end{proof}

\subsection{Groups with interesting scl spectra} \label{subsec: scl spec on a delta inf}
Theorem \ref{thm: gap theorem for RAAGs and RACGs} implies interesting properties of the spectrum of the infinitely generated right-angled Artin group $\Arm(\Delta_\infty)$.
\begin{proposition}
	The set of values obtained as scl of integral chains in $\Arm(\Delta_\infty)$ is dense in $\R_{\ge0}$, and in particular there is no spectral gap. 
	However, there is a gap $1/2$ for elements in $\Arm(\Delta_\infty)$.
\end{proposition}
\begin{proof}
	Note that $\Arm(\Delta_\infty)$ retracts to $\Arm(\Delta_m)$ for any $m\in\Z_+$. Thus 
	$$\scl_{\Arm(\Delta_\infty)}(\delta_m)=\scl_{\Arm(\Delta_m)}(\delta_m)\in \left[\frac{1}{12(m+2)},\frac{1}{m}\right].$$
	by Proposition \ref{prop: delta_m has low scl}, where $\delta_m$ is defined in Section \ref{subsec:scl in opposite paths}.
	Thus we obtain a sequence of integral chains whose scl is positive and converges to $0$. 
	Taking integer multiples of such integral chains proves the density. 
	The gap $1/2$ for elements in $\Arm(\Delta_\infty)$ is shown in \cite{Heuer}.
\end{proof}

No groups were previously known to have a gap for elements but no gap for integral chains.

With a small modification to the group, we can make scl values of \emph{elements} eventually dense in $\R_{\ge0}$.
\begin{proposition}
	Let $G=\Arm(\Delta_\infty)\star F_3$, where $F_3$ is the free group generated by $a,b,c$. Then $\scl_G(g)\ge 1/2$ for all $g\neq id\in G$, and the set $\{\scl_G(g)\mid g\in[G,G]\}$ is dense in $[3/2,\infty)$.
\end{proposition}
\begin{proof}
	If $g\neq id$ conjugates into $\Arm(\Delta_\infty)$ or $F_3$, the lower bound $1/2$ is known by \cite{Heuer} and \cite{DH91}. Otherwise, the lower bound $1/2$ follows from \cite{Chen:sclfpgap} since both factor groups are torsion-free.
	
	As for the density, recall that the integral chain $\delta_m=g_{0,m}-g_{1,m}-g_{0,m-1}+g_{1,m-1}$ has scl between $1/(12(m+2))$ and $1/m$. 
	Applying Proposition \ref{prop: formulas for scl} to
	$g=cbag_{0,m}^n a^{-1} g_{1,m}^{-n}b^{-1}g_{0,m-1}^{-n}c^{-1}g_{1,m-1}^n$ for any $n\in\Z_+$, we have 
	$$\scl_G(g)=\scl_{\Arm(\Delta_\infty)}(n\delta_m)+\frac{3}{2}\in \left[\frac{3}{2}+\frac{n}{12(m+2)},\frac{3}{2}+\frac{n}{m}\right].$$
	The density follows since $m$ and $n$ are arbitrary positive integers.
\end{proof}

\section{Scl of vertex chains} \label{sec:vertex chains}

We describe an algorithm to compute scl of vertex chains in Section \ref{subsec:computation by linear program}. 
This allows us to relate $\scl$ to the \emph{fractional stability number} ($\fsn$) of graphs in Section \ref{subsec:fsn as scl}.
In Section \ref{subsec:statistics of scl and fsn}, we observe and explain the similarity in histograms of scl and fsn on random words and graphs, respectively (Figure \ref{figure:scl and fsn intro}).

\subsection{Computation by linear programming} \label{subsec:computation by linear program}
Given any vertex chain $c$, we will give two linear programming problems $(P_c)$ and $(P_c^*)$ that both compute $\scl(c)$. They are dual to each other and thus yield dual solutions. Moreover, feasible solutions of $(P_c)$ yield quasimorphisms with controlled defects and feasible solutions to $(P_c^*)$ yield admissible surfaces.

To describe the linear programming problems, we introduce the following notion.
\begin{definition} \label{def:frac stab set}
	A \emph{stable measure} on a graph $\Gamma$ is a list of numbers $\mu=(\mu_v)_{v\in \Vrm}$, one for each vertex, such that
	\begin{itemize}
		\item the sum of $\mu_v$ over all vertices in any given clique $q$ of $\Gamma$ is at most $1$;
		\item $\mu_v\ge 0$ for each $v$.
	\end{itemize}
\end{definition}

A set $S$ of vertices is called \emph{a stable set} if they are pairwise non-adjacent in $\Gamma$. Equivalently, each clique contains at most one vertex in $S$. Thus the indicator function of any stable set is a stable measure.

Given a stable measure $\mu$ and a vertex chain $c$, let 
$$|\mu|_c \defeq \sum_{v\in\Vrm}  \mu_v\cdot \scl_{G_v}(c_v),$$
which is linear in $\mu_v$. 
Then maximizing $|\mu|_c$ among stable measures is a linear programming problem $(P_c)$ since the defining properties of a stable measure are linear inequalities in $\mu_v$.
Note that the set of stable measures is a compact convex rational polyhedron in $\R^{|\Vrm|}$, and thus the problem $(P_c)$ has an optimal solution at a rational point.

In general, one can replace $\scl_{G_v}(c_v)$ by other weights on vertices, and the corresponding problem is called the \emph{fractional weighted stability number} in graph theory; see \cite[Page 333]{lovasz}. Note that the result is rational if the weights are, since the feasible set is a rational polyhedron.

To describe its dual problem $(P_c^*)$, we introduce weighted clique cover.
\begin{definition}
	Given a vertex chain $c$, a \emph{weighted clique cover} with respect to $c$ is a list of real numbers (considered as weights) $y=\{y_q\}$, indexed by the cliques $q$ of $\Gamma$ such that
	\begin{itemize}
		\item the sum of $y_q$ over all cliques containing any given vertex $v$ is at least $\scl_{G_v}(c_v)$.
		\item $y_q\ge 0$ for every clique $q$.
	\end{itemize}
	For any weighted clique cover $y$, let $|y|\defeq \sum y_q$, where the sum is taken over all cliques $q$ of $\Gamma$.
\end{definition}

Then minimizing $|y|$ over all weighted clique covers with respect to $c$ is the linear programming problem $(P_c^*)$ dual to the problem $(P_c)$. Thus they have the same optimal value by the strong duality theorem of linear programming, explained as follows.
\begin{lemma}\label{lemma: strong duality}
	For any vertex chain $c$ in a graph product $\Gcl(\Gamma)$, 
	we have
	$$\max_\mu |\mu|_c=\min_y |y|,$$
	where the maximization is taken over stable measures $\mu$ and the minimization is taken over weighted clique covers $y$.
\end{lemma}

\begin{proof}
Let $Cl(\Gamma)$ be the set of cliques of $\Gamma$. 
Let $M_\Gamma$ be the $0$-$1$ matrix where the columns are indexed by the vertices $v \in \Vrm(\Gamma)$ and the rows are indexed by all cliques $q \in Cl(\Gamma)$ such that the ($q$,$v$)-entry is $1$ if and only if $v \in q$.

Then in matrix form, the problem $(P_c)$ is to maximize $s^T\cdot \mu$ subject to $M_\Gamma\cdot \mu \leq 1_{Cl(\Gamma)}$ and $\mu\ge 0$. 
Here $1_{Cl(\Gamma)}$ is the vector of $1$'s of length $|Cl(\Gamma)|$ and $s$ is the vector indexed by $\Vrm(\Gamma)$ with entry $\scl_{G_v}(c_v)$ at vertex $v$. 
By the strong duality theorem of linear programming \cite[Page 91 (19)]{linprog} the optimal value agrees with
the minimal value of $1_{Cl(\Gamma)}^T \cdot y$ subject to $M_\Gamma^T\cdot y \ge s$ and $y\ge0$, which is the matrix form of $(P_c^*)$. 
\end{proof}

The main result of this subsection is that both $(P_c)$ and $(P_c^*)$ compute $\scl_{\Gcl(\Gamma)}(c)$.
\begin{theorem}\label{thm: scl of vertex chains}
	For any vertex chain $c$ in a graph product $G=\Gcl(\Gamma)$, we have
	$$\scl_G(c)=\max_\mu |\mu|_c=\min_y |y|,$$
	where the maximization is taken over stable measures $\mu$ and the minimization is taken over weighted clique covers $y$.
\end{theorem}

By Lemma \ref{lemma: strong duality}, to prove Theorem \ref{thm: scl of vertex chains}, it suffices to establish the following two lemmas.
%\ref{lemma: vertex chain scl upper bound} and \ref{lemma: vertex chain scl lower bound}.
\begin{lemma}\label{lemma: vertex chain scl upper bound}
	For any vertex chain $c$ in a graph product $G=\Gcl(\Gamma)$, we have $\scl_G(c)\le |y|$ for any weighted clique cover $y$ with respect to $c$.
\end{lemma}

\begin{lemma}\label{lemma: vertex chain scl lower bound}
	For any vertex chain $c$ in a graph product $G=\Gcl(\Gamma)$, we have $\scl_G(c)\ge |\mu|_c$ for any stable measure $\mu$.
\end{lemma}

To prove Lemma \ref{lemma: vertex chain scl upper bound},
we first show that scl of a vertex chain is increasing in the coefficients.
\begin{lemma}\label{lemma: vertex chain monotone}
	Fix a chain $c_v$ in each vertex group $G_v$. 
	Given numbers $\lambda_v\ge\lambda'_v\ge 0$ for each vertex $v$, we have
	$\scl_G(\sum_v \lambda_v c_v)\ge \scl_G(\sum_v \lambda'_v c_v)$.
\end{lemma}
\begin{proof}
	It suffices to show that scl is non-decreasing in every single $\lambda_u$ fixing $\lambda_v$ for all $v\neq u$. Split $G$ as an amalgam
	$\Gcl(\st(u))\star_{\Gcl(Lk(u))}\Gcl(\Gamma\setminus\{u\})$.
	Then we think of the vertex chain $c=\sum_v \lambda_v c_v=\lambda_u c_u + \sum_{v\neq u} \lambda_v c_v$ as a sum of
	two chains supported on the two factor groups.
	
	By \cite[Theorem 6.2]{CH:sclgap}, we have
	$$
	\scl_{G}(c)=\inf_{d} [
	\scl_{\Gcl(\st(u))}(\lambda_u c_u + d) + \scl_{\Gcl(\Gamma \setminus \{u\} )}(-d +
	\sum_{v\neq u} \lambda_v c_v)],
	$$
	where the infimum is taken over all chains $d$ in $\Gcl(Lk(v))$.
	
	Since $\Gcl(\st(u))=G_u\times \Gcl(Lk(u))$ is a direct product, by Proposition \ref{prop:chains on direct products} we have
	$\scl_{\Gcl(\st(u))}(\lambda_u c_u + d)=\max(\lambda_u \scl_{G_u}(c_u),\scl_{\Gcl(\Lk(v))}(d))$, which is clearly non-decreasing in $\lambda_u$.
	Thus $\scl_{G}(c)$ is non-decreasing in $\lambda_u$ by the formula above.
\end{proof}

\begin{proof}[Proof of Lemma \ref{lemma: vertex chain scl upper bound}]
	Recall that each induced subgroup is a retract in a graph product, so scl in the two groups agree for any chain in the subgroup (Proposition \ref{prop: mono and retract}).
	Without loss of generality, assume $\scl_{G_v}(c_v)>0$ for each vertex $v$, as otherwise we may consider the problem on the induced subgroup supported on those vertices with this property.
	Given a weighted clique cover $y$, for each clique $q$, define a vertex chain $d_q=\sum_{v\in q} \frac{y_q}{\scl_{G_v}(c_v)} c_v$. 
	Since $\Gcl(q)$ is the direct product of vertex groups $G_v$ for $v\in q$, 
	by Proposition \ref{prop:chains on direct products}, we have 
	$$\scl_{\Gcl(\Gamma)}(d_q)=\scl_{\Gcl(q)}(d_q)=\max_{v\in q} \scl_{G_v}\left(\frac{y_q}{\scl_{G_v}(c_v)} c_v\right)=y_q.$$
	Consider the vertex chain $\sum_q d_q=\sum_v \frac{\sum_{q\ni v} y_q}{\scl_{G_v}(c_v)} c_v$. 
	Note that the coefficient of $c_v$ is $\frac{\sum_{q\ni v} y_q}{\scl_{G_v}(c_v)}$, which is no less than $1$, the coefficient of $c_v$ in $c$, by the definition of weighted clique cover.
	Thus by Lemma \ref{lemma: vertex chain monotone}, we have
	$$\scl(c)=\scl(\sum_v c_v) \le \scl(\sum_q d_q)\le \sum_q \scl(d_q)=\sum_q y_q=|y|.$$ 
\end{proof}

To prove Lemma \ref{lemma: vertex chain scl lower bound}, we construct quasimorphisms and use Bavard's duality.

Given a quasimorphism $f_v$ on each vertex group $G_v$, we can combine them to obtain
a quasimorphism $f$ on the graph product $G=\Gcl(\Gamma)$ as follows.

For each element $g\in G$ with reduced expression $g=g_1\cdots g_n$, we naturally have a vertex chain
$s(g) \defeq \sum_i g_i$.
This only depends on $g$ since reduced expressions are unique up to syllable shuffling.

Define $f(c)\defeq \sum_v f_v(c_v)$ for all vertex chains and extend $f$ to $\Gcl(\Gamma)$ by setting 
$$f(g) \defeq f(s(g))$$ using the splitting above. 
%\jcomm{$f$ is not necessarily homogeneous, so we should take harmonic representatives $f_v$ for extremal quasimorphisms in the vertex groups. Defining the homogeneous version also works, but I guess the defect bound would be harder to prove?} \ncomm{I guess 'extremal' quasimorphism is enoguh? You're right .. homogeneous is definitley harder...}

\begin{lemma}\label{lemma: vertex qm defect bound}
	If each $f_v$ is antisymmetric,
	then the function $f$ defined above is a quasimorphism on $G$ with defect 
	$D(f)= \sup_q \sum_{v\in q} D(f_v)$,
	where the supremum is taken over all cliques $q$ of $\Gamma$.
\end{lemma}
\begin{proof}
	For each clique $q$ and each vertex $v\in q$, we can find elements $g_v,h_v\in G_v$
	with $f_v(g_v)+f_v(h_v)-f_v(g_v h_v)$ arbitrarily close to $D(f_v)$.
	Then for $g_q\defeq\prod_{v\in q} g_v$ and $h_q\defeq\prod_{v\in q} h_v$, we have
	$$f(g_q)+f(h_q)-f(g_qh_q)=\sum_{v\in q}[f_v(g_v)+f_v(h_v)-f_v(g_vh_v)],$$
	which can be made arbitrarily close to $\sum_{v\in q} D(f_v)$. This proves the ``$\ge$'' direction.
	
	For the reversed direction, for any $g,h\in G$, we have reduced expressions $g = g_0 q_g x$ and $h = x^{-1} q_h h_0$ as in Proposition \ref{prop:normal forms for products},
	where
	$\supp(q_g)=\supp(q_h)=q=\{v_1,\cdots,v_k\}$ is a clique,
	$q_g=g_1\cdots g_k$, $q_h=h_1\cdots h_k$ with $g_i,h_i\in G_{v_i}$,
	and $gh$ admits a reduced expression $gh=g_0 q_{gh} h_0$ with $q_{gh}=(g_1h_1)\dots (g_k h_k)=q_gq_h$.
	Since each $f_i$ is antisymmetric, we have $f(x)+f(x^{-1})=0$. 
	The definition of $f$ implies that for the reduced expression $g=g_0 q_g x$ we have $f(g)=f(g_0)+f(q_q)+f(x)$ and similarly for $h$ and $gh$.
	Hence
	$$
	|f(g)+f(h)-f(gh)|=|f(q_g)+f(q_h)-f(q_{gh})|=\left|\sum_{i=1}^k f_{v_i}(g_i+h_i-g_i h_i)\right|\le \sum_{i=1}^k D(f_{v_i}),
	$$
	where the second inequality uses the formula derived in the first part of the proof. This proves the equality.
\end{proof}

Now we are in a place to prove Lemma \ref{lemma: vertex chain scl lower bound}.
\begin{proof}[Proof of Lemma \ref{lemma: vertex chain scl lower bound}]
	For each vertex $v$, let $\phi_v$ be an extremal antisymmetric quasimorphism as in Proposition \ref{prop:extremal qm} for the chain $c_v$, i.e.\ we have $\bar{\phi}_v(c_v)=\scl_{G_v}(c_v)$ and $D(\phi_v)=1/4$.
	
	Given any stable measure $\mu=(\mu_v)$, let $f_v=\mu_v\cdot \phi_v$ and let $f$ be the quasimorphism obtained as above by combining $f_v$'s. Then for each clique $q$, 
	we have $\sum_{v\in q} D(f_v)=\frac{1}{4}\sum_{v\in q} \mu_v\le 1/4$ by the definition of stable measures.
	Thus by Lemma \ref{lemma: vertex qm defect bound}, we have $D(f)\le 1/4$, and thus $D(\bar{f})\le 1/2$ by Proposition \ref{prop:homog rep}.
	Note that $\bar{f}(g_v)=\mu_v\cdot\bar{\phi}_v(g_v)$ for each $g_v\in G_v$ and similarly for any chain in $G_v$.
	By Bavard's duality, we have
	$$\scl_G(c)\ge \frac{\bar{f}(c)}{2D(\bar{f})}\ge \bar{f}(c)=\sum_v \mu_v\cdot \bar{\phi}_v(c_v)=\sum_v \mu_v \cdot\scl_{G_v}(c_v)=|\mu|_c.$$
\end{proof}

\begin{proof}[Proof of Theorem \ref{thm: scl of vertex chains}]
	We have $\max_\mu |\mu|_c\le \scl_{\Gcl(\Gamma)}(c)\le\min_y |y|$ by Lemmas \ref{lemma: vertex chain scl upper bound} and \ref{lemma: vertex chain scl lower bound}. By Lemma \ref{lemma: strong duality}, we know $\max_x |x|_c=\min_y |y|$, which proves the equality. 
\end{proof}

Summarizing the results in this subsection, we give a proof of Theorem \ref{thm: scl of vertex chains, beginning of sec}.
\begin{proof}[Proof of Theorem \ref{thm: scl of vertex chains, beginning of sec}]
	The linear programming problems $(P_c)$ and $(P_c^*)$ both compute $\scl_{\Gcl(\Gamma)}(c)$ by Theorem \ref{thm: scl of vertex chains}. Since the optimal solution of $(P_c)$ is achieved at some rational point, we see that $\scl_{\Gcl(\Gamma)}(c)$ is rational when $\scl_{G_v}(c_v)$ is rational for all $v$. Finally, by taking $x_v=1$ and $x_u=0$ for all $u\neq v$, we obtain a stable measure and
	from the formulation $(P_c)$ we clearly have
	$$\scl_{\Gcl(\Gamma)}(c)\ge \scl_{G_v}(c_v)$$
	for each vertex $v$.
\end{proof}

Theorem \ref{thm: scl of vertex chains} yields an algorithm to compute stable commutator length of vertex chains. This algorithm has been implemented in Python. The code may be found on the second author's website \footnote{\url{https://www.nicolausheuer.com/code.html}}.

\subsection{scl and fractional stability number} \label{subsec:fsn as scl}

In this section we consider the case where all the vertex terms in the vertex chain have the same stable commutator length. This relates scl to well-studied invariants in graph theory.

To be explicit, we construct for a given graph $\Gamma$ a graph $D_\Gamma$ and a chain $d_\Gamma$ in the right-angled Artin group $\Arm(D_{\Gamma})$, such that $\Arm(D_{\Gamma})$ can be also viewed as a graph product over $\Gamma$ and $d_{\Gamma}$ is a vertex chain where each term has scl $1/2$.
%
%We show that the stable commutator length of $d_{\Gamma}$ is half the \emph{fractional stability number} of the graph.

\begin{definition}[Double Graph] \label{def:double graph}
For a graph $\Gamma$ with vertex and edge set $\Vrm(\Gamma)$ and $\Erm(\Gamma)$, let 
$D_\Gamma$ be the graph with vertex and edge set 
\begin{eqnarray*}
\Vrm(D_\Gamma) &=& \{ \att_v, \btt_v \mid v \in \Vrm(\Gamma) \} \mbox{ and} \\
\Erm(D_\Gamma) &=& \{ (\att_v, \att_w), (\att_v, \btt_w), (\btt_v, \att_w), (\btt_v, \btt_w) \mid (v,w) \in \Erm(\Gamma) \}.
\end{eqnarray*}
Moreover, let $d_{\Gamma} = \sum_{v \in \Vrm(\Gamma)} [\att_v, \btt_v]$ in $\Arm(D_\Gamma)$. Then $D_\Gamma$ is called the \emph{double graph} and $d_{\Gamma}$ the \emph{double chain} associated to $\Gamma$.
\end{definition}

\begin{definition}[Fractional Stability Number] \label{defn:fsn}
Let $\Gamma$ be a graph. Then the \emph{fractional stability number} of $\Gamma$ is defined as
$$
\fsn(\Gamma) \defeq \max_\mu \sum_{v} \mu_v ,
$$
where the maximum is taken over all stable measures $\mu$.
\end{definition}

The fractional stability number of a graph is the \emph{fractional chromatic number} of its opposite graph. 
This invariant appears more frequently in the literature. For a reference to fractional stability number see \cite{scheinerman2011fractional}. 
The results of the previous section implies:

\begin{theorem}[$\scl$ and $\fsn$] \label{thm: scl and fsn}
Let $\Gamma$ be a graph and let $D_\Gamma$ and $d_\Gamma$ be the associated double graph and double chain respectively.  Then
$$
\scl_{\Arm(D_{\Gamma})}(d_\Gamma) = \frac{1}{2} \fsn(\Gamma),
$$
where $\fsn(\Gamma)$ is the fractional stability number of $\Gamma$.
\end{theorem}
\begin{proof}
Note that for each $v\in \Gamma$, the vertices $\att_v$ and $\btt_v$ are not adjacent and hence $\Arm(\{\att_v,\btt_v\})$ is a free group $F_v=F(\att_v, \btt_v)$ of rank two.
Also note that $\att_v$ and $\btt_v$ are both (resp. not) adjacent to $\att_u$ and $\btt_u$ if $v$ is (resp. not) adjacent to $u$.
Thus we observe that $\Arm(D_{\Gamma})$ is a graph product over $\Gamma$, where the vertex groups are the free groups $F_v$. 
In this view, $d_{\Gamma}$ is a vertex chain where each vertex term is $[\att_v, \btt_v]$, which satisfies $\scl_{F_v}([\att_v, \btt_v]) = 1/2$. 
Thus the result follows from Theorem \ref{thm: scl of vertex chains}. 
\end{proof}

For the rest of this subsection, we apply known results of $\fsn$ on graphs to deduce properties of $\scl$ in such groups.
We first describe the full spectrum of $\fsn$ on graphs. Note that the full spectrum of scl is not known even in 
the best understood case of free groups.

\begin{proposition}[\protect{see also \cite[Proposition 3.2.2]{scheinerman2011fractional}}] \label{prop: rat realization of fsn}
	The set of numbers that appear as $\fsn(\Gamma)$ for some nonempty graph $\Gamma$ is
	$$\{1\}\cup [2,\infty)\cap \Q.$$
\end{proposition}
\begin{proof}
	We already know that $\fsn(\Gamma)$ is always rational since the feasible set is a rational polyhedron. It is also easy to notice that $\fsn(\Gamma)\ge 1$ since each vertex is a stable set, and that $\fsn(\Gamma)\ge 2$ whenever there are two non-adjacent vertices.
	
	So it suffices to construct graphs to achieve all rational numbers $r\ge2$. For any $m\ge2$ and $n\ge2m$, let $\Gamma_{m,n}$ be the graph with $n$ vertices $v_1,\ldots,v_n$ such that it is the union of cliques on $v_{i+1},\ldots,v_{i+m}$ for all $1\le i\le n$, where indices are taken mod $n$. We claim that $\fsn(\Gamma_{m,n})=n/m$, from which the result would follow.
	
	Using $n\ge 2m$, it is straightforward to check that the cliques used to described $\Gamma_{m,n}$ are all the maximal cliques. Thus having weight $1/m$ on all vertices is a stable measure, which shows $\fsn(\Gamma_{m,n})\ge n/m$.
	
	On the other hand, assigning weight $1/m$ to each maximal clique (and $0$ to all smaller cliques) is a weighted clique cover, and hence $\fsn(\Gamma_{m,n})\le n/m$ by the dual problem. Thus $\fsn(\Gamma_{m,n})=n/m$.
\end{proof}

For comparison, it is known that scl in free groups has a sharp lower bound $1/2$, and based on experiments, the spectrum seems to be proper in $[1/2,3/4)$ and dense in $[3/4,\infty)$. However, it appears to be much harder if possible at all to construct families of elements or integral chains in free groups with scl achieving arbitrary rational numbers greater than $1$.

Combining Theorem \ref{thm: scl and fsn} and Proposition \ref{prop: rat realization of fsn} we deduce:
\begin{theorem}[Rational realization] \label{thm: rat real sec}
For every rational number $q \geq 1$ there is an integral chain $c$ in a right-angled Artin group $\Arm(\Gamma)$ such that $\scl_{\Arm(\Gamma)}(c) = q$.
\end{theorem}

Computing the fractional stability number is NP-hard \cite{grotschel1981ellipsoid}. This implies that computing scl in RAAGs is also NP-hard.
\begin{theorem}[NP-hardness] \label{thm: np hard sec}
Unless $P=NP$, there is no algorithm which, given a simplicial graph $\Gamma$, an element $g \in \Arm(\Gamma)$ and a rational number $q \in \Q^+$ decides if $\scl_{\Arm(\Gamma)}(g) \leq q$ in polynomial time in $|\Vrm(\Gamma)| + |g|$. The same holds for chains. 
\end{theorem}
\begin{proof}
It is known \cite{grotschel1981ellipsoid} that computing $\fsn$ for a graph $\Gamma$ is NP-hard. Given a graph $\Gamma$, we may in polynomial time construct the double graph and the double chain $d_\Gamma \in \Arm(D_{\Gamma})$. By Theorem \ref{thm: scl and fsn}, computing $\scl(d_\Gamma)=\frac{1}{2}\fsn(\Gamma)$ is NP-hard as well.

Let $\tilde{D}_\Gamma$ be the graph obtained from $D_\Gamma$ by adding $|V(\Gamma)|$ isolated vertices. Then $\Arm(\tilde{D}_{\Gamma})$ is a free product $\Arm(D_\Gamma)\star F_{|V(\Gamma)|}$.
Using Proposition \ref{prop: formulas for scl}, we may in polynomial time construct an element $\tilde{d}$ in $\Arm(\tilde{D}_{\Gamma})$ such that $\scl(\tilde{d}) = \scl(d_{\Gamma}) + \frac{\Vrm(\Gamma)-1}{2}$. Thus computing scl of elements in RAAGs is also NP-hard.
\end{proof}

\subsection{Histograms of scl and fsn} \label{subsec:statistics of scl and fsn}

Although it is NP-hard, we may compute fsn relatively quickly for graphs with up to $30$ vertices. This allows us to perform computer experiments on the distribution of $\fsn$ for random graphs.
The result of these experiments is recorded (rescaled by $1/2$) in Figure \ref{fig:fsn} in the introduction. 
Here we considered $50,000$ random graphs with $25$ vertices, where between every two vertices there is an edge with probability $1/2$. This reveals an interesting distribution of $\fsn$ on random graphs: Values with low denominator appear much more frequently and the histogram exhibits a self-similar behavior.

The same type of histogram has been observed for stable commutator length of random elements in the free group (Figure \ref{fig:scl}). Here we consider $50,000$ uniformly chosen random words of length $24$ in the commutator subgroup of the free group $F_2$. See \cite[Section 4.1.9]{Cal:sclbook} for a discussion of this phenomenon and comparison to Arnold's tongue. 
Explanations of these patterns in the frequency for either $\scl$ or $\fsn$ are not known.

In this section we will give a brief statistical analysis of both $\scl$ and $\fsn$. We show that both scl and fsn can be modeled using the same type of distributions which we describe in Definition \ref{defn:dist X}. 
While this is purely heuristic, 
it indicates that $\fsn$ and $\scl$ converge for large graph sizes / word lengths to a similar distribution; see Question \ref{quest: dist scl and fsn similar}.

Let $SCL$ denote the random variable $2 \cdot \scl(W)$ where $W$ is the random variable with a uniform distribution on $\{w \in F_2 \mid |w| \leq 24 \}$ and let $FSN$ be the random variable $\fsn(\Gamma)$ where $\Gamma$ is a random variable with uniform distribution $\{ \Gamma \mid \Vrm(\Gamma) = 25 \}$.
We note that the factor of $2$ for $\scl$ is intended and indeed necessary. In light of the relationship to Euler characteristic (Definition \ref{defn: scl via euler charac}) and Bavard's Duality Theorem (Theorem \ref{thm:bavard}), $2 \cdot \scl$ seems to be the more natural invariant. 
The histograms of $50,000$ independent instances of $SCL$ and $FSN$ may be found in Figure \ref{figure:scl and fsn sec}.

We make two crucial heuristic observations:
\begin{enumerate}
\item For large integers $n$, we observe that $\Pbb(\mbox{$X$ has denominator n}) \sim \frac{\phi(n)}{n^d}$, where $\phi$ is Euler's Totient function and $X$ is $SCL$ or $FSN$. This is depicted in Figure \ref{fig:denominator_power}. Experimentally we may estimate that $d \sim 1.7$ for SCL and $d \sim 2.5$ for $FSN$.  
It is also apparent that for smaller $n$ this heuristic does not hold, and that instead this coefficient is much smaller. This suggests that the exponent may be approximated by $d \cdot (1 - n^\beta)$ for some negative $\beta$.

\item For a fixed denominator $n$, let $X_n$ be the random variable of $X$ conditioned on that $X$ has denominator $n$. Then $X_n$ follows roughly a normal distribution $B_n$ (rounded to the closest rational in $1/n$) with fixed mean $\mu$ and standard deviation $\sigma_n$; see Figure \ref{fig:denominator_dist}. The standard deviation appears to be roughly of the form $\sigma_n = c_1 \cdot n^{c_2}$; see Figure \ref{fig:different_std}.
\end{enumerate}
This suggests that both the histogram of $\scl$ and $\fsn$ are the result of an interference of several (rounded) `normal' distributions $B_n$. 

 \begin{figure}
  \centering

  \subfloat[]{\includegraphics[width=0.3\textwidth]{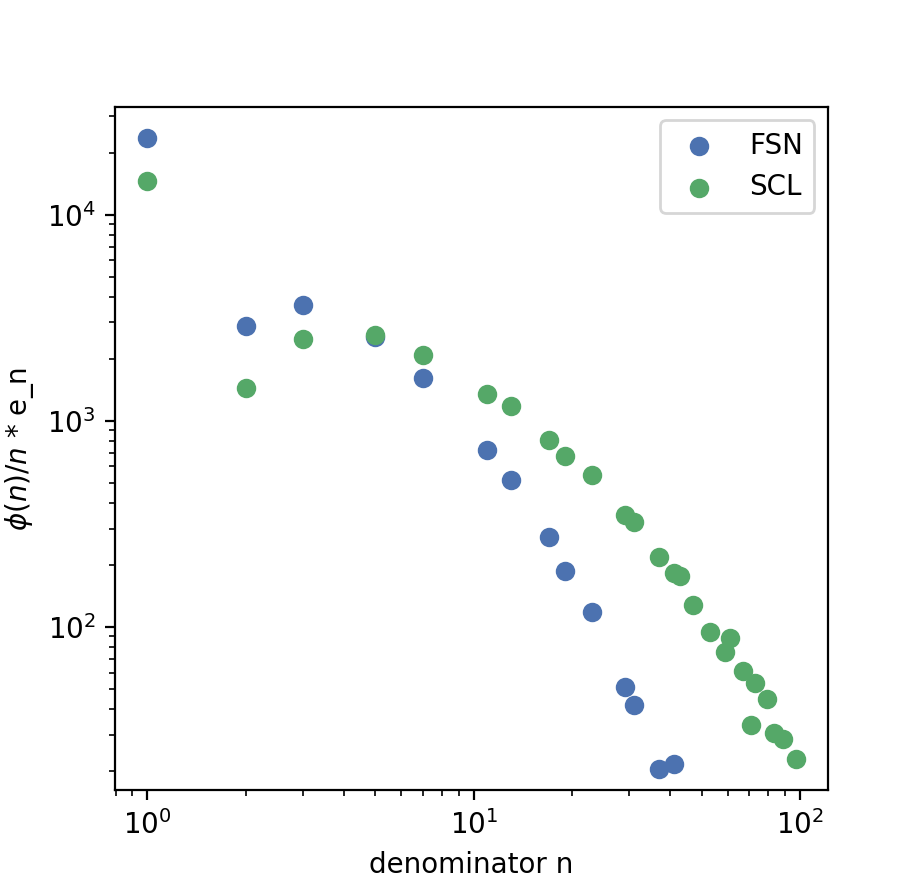} \label{fig:denominator_power}} 
  \hfill
  \subfloat[]{\includegraphics[width=0.3\textwidth]{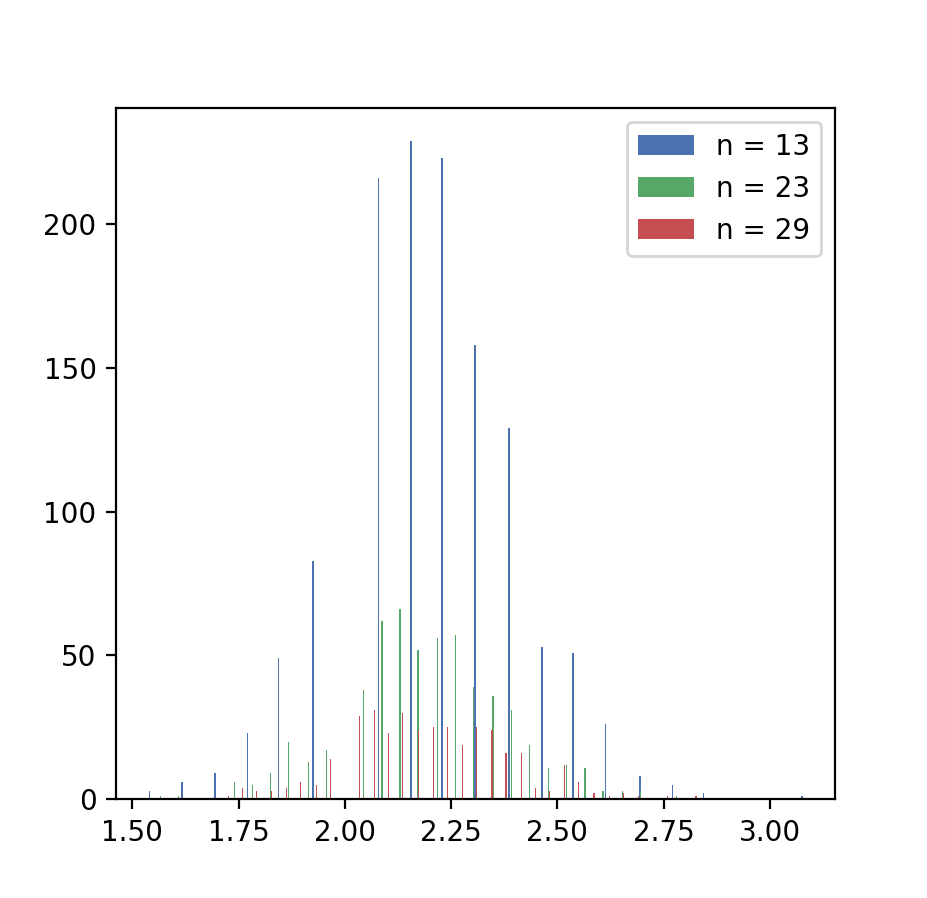} \label{fig:denominator_dist}}
  \hfill
  \subfloat[]{\includegraphics[width=0.3\textwidth]{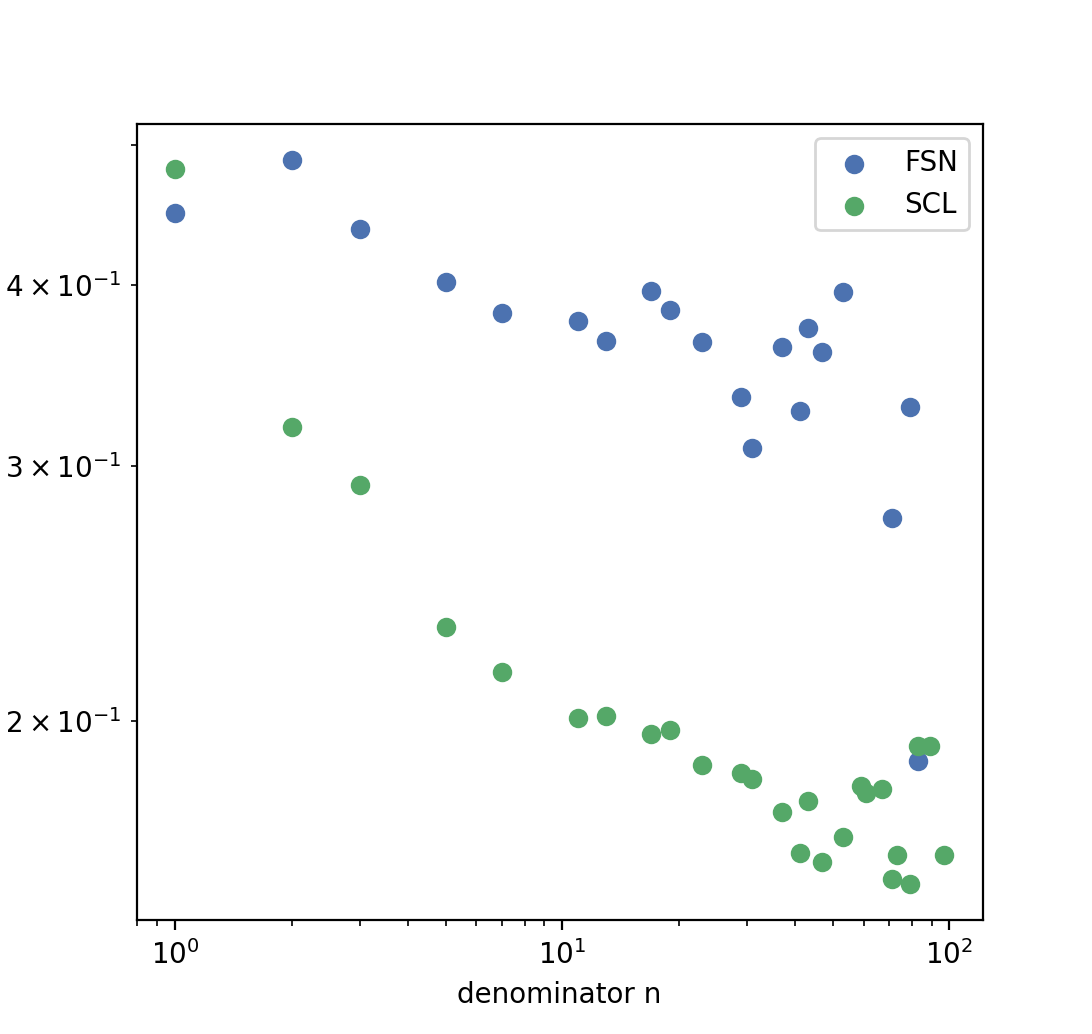} \label{fig:different_std}}

  \caption{Statistical analysis of SCL and FSN: Let $X$ be either a random $\scl$ in $F_2$ on words of length $24$ or a $\fsn$ of a random graph on $25$ vertices. Let $X_n$ denote the set of elements with denominator exactly $n$. Figure \ref{fig:denominator_power} plots $e_n = \# X_n$, the number of elements having denominator $n$ for $50,000$ random samples. Figure \ref{fig:denominator_dist} shows the distribution of SCL having denominator $13$, $23$ and $29$ for $50,000$ samples. Figure  \ref{fig:different_std} shows the different standard sample deviations of $X_n$.} \label{fig:statistics scl and fsn}
\end{figure}

 \begin{figure}[!tbp]
  \centering
   \subfloat[$2 \cdot \scl(w)$ for $w \in F_2$ in the commutator subgroup with length $24$ uniformly chosen for $50,000$ instances (green) vs. $50,000$ random instances of the $X$ distribution modeled with parameters $d = -2$, $\beta = -0.2$, $\mu = 2.164$, $c_1 = 0.3$ and $c_2 = -0.14$ (blue)]{\includegraphics[width=\textwidth]{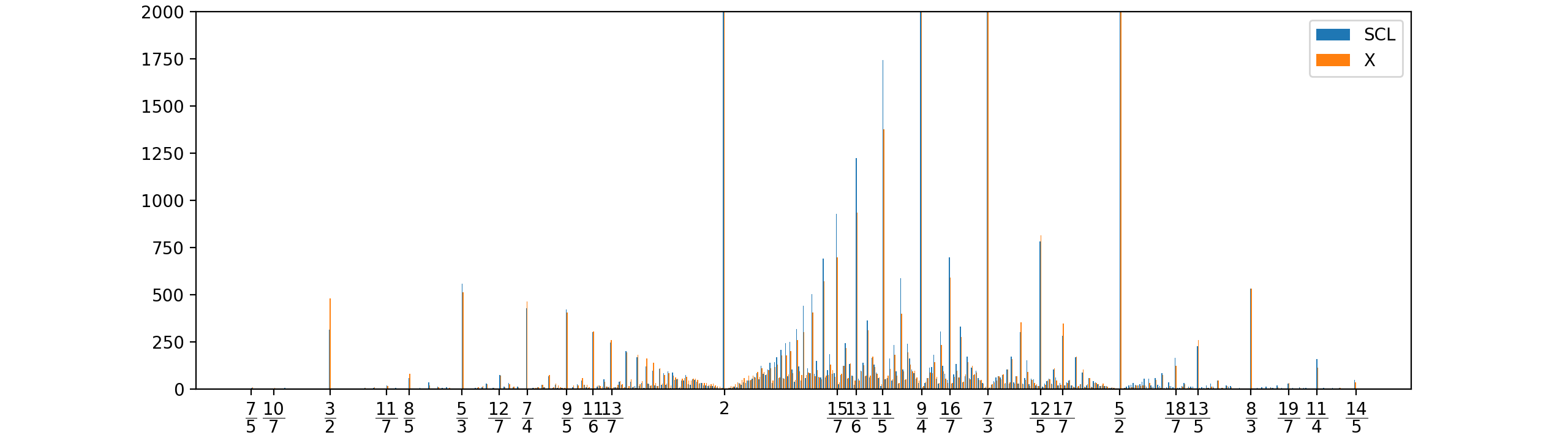}\label{fig:scl_sec}}
  \\
  \subfloat[$\fsn(\Gamma)$ for $\Gamma$ uniformly chosen as a graph with $25$ vertices for $50,000$ instances (green) vs. $50,000$ random instances of the $X$ distribution with parameters  $d = -2.8$, $\beta = -0.2$, $\mu = 6.141$, $c_1 = 0.5$ and $c_2 = -0.1$ (blue).]{\includegraphics[width=\textwidth]{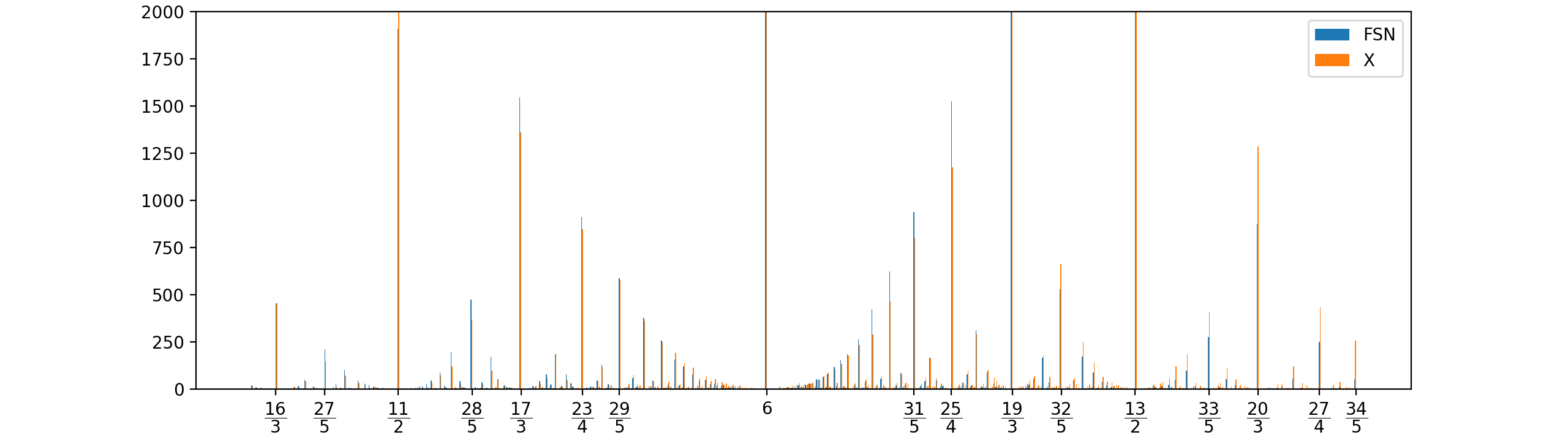}\label{fig:fsn_sec}}  
  \caption{Modeling $2 \cdot \scl$ and $\fsn$ using the $X$ distribution. In both cases, we truncated the spikes to fit the figure.} \label{figure:scl and fsn sec}
\end{figure}

These observations lead us to the following construction of a random variable $X$ depending on real parameters $d, \beta, \mu, c_1, c_2$.

\begin{definition}[The distribution $X$] \label{defn:dist X}
Let $d<-1$, $\beta<0$, $c_1> 0$, $c_2<0$, and $\mu$ be real parameters.
Define the random variable $X = X(d, \beta, \mu, c_1, c_2)$ as follows:

Set $p(n, \beta, d) = n^{(1-n^{\beta}) \cdot d}$. Choose with probability $p(n, \beta, d) / \sum_{n=1}^\infty p(n, \beta, d)$ an integer $n \in \N$.
Choose the rational $X$ in $\frac{1}{n} \Z$ as follows: Let $N_n$ be the random variable with distribution $\Ncl(\mu, \left(c_1 \cdot n^{c_2} \right)^2 )$, the normal distribution with mean $\mu$ and standard deviation $c_1 \cdot n^{c_2}$. Set $X$ to be the number in $\frac{1}{n} \Z$ closest to $N_n$.
\end{definition}

The distribution of $X$ may be found on the second authors website \footnote{\url{https://www.nicolausheuer.com/code.html}}. We may use this to fit $X$ to $SCL$ and $FSN$. The result of this experiment is shown in Figure \ref{figure:scl and fsn sec}. At least qualitatively, $X$ is a good approximation of the distribution of $SCL$ and $FSN$. 

Based on this, we ask:
\begin{question} \label{quest: dist scl and fsn similar}
Is there a \emph{natural} distribution $Y$ indexed by some parameter set $\mathcal{P}$ such that there are sequences of parameters $s_n$, $f_n$ for $n \in \N$ such that as $n \to \infty$, both the random variable $\scl(w)$, for $w$ uniformly chosen from $\{ w \in [F_2, F_2] \mid |w| = 2 \cdot n \}$, and $\fsn(\Gamma)$ where $\Gamma$ is uniformly chosen among all graphs with $n$ vertices converge almost surely to $Y(s_n)$ and $Y(f_n)$, respectively?
\end{question}

\bibliographystyle{alpha}
\bibliography{gap_for_chains}

%\noindent
%\emph{Lvzhou Chen}\\[.5em]
%  {\small
%  \begin{tabular}{@{\qquad}l}
%Department of Mathematics\\ University of Chicago\\ Chicago, Illinois, USA \\
%    \textsf{lzchen@math.uchicago.edu}, 
%    \textsf{http://math.uchicago.edu/$\sim$lzchen/}
%  \end{tabular}}
%  
%
%
%
%
%\medskip
%
%
%\noindent
%\emph{Nicolaus Heuer}\\[.5em]
%  {\small
%  \begin{tabular}{@{\qquad}l}
%Department of Pure Mathematics and Mathematical Statistics \\
%	Centre for Mathematical Sciences \\
%	University of Cambridge \\ Cambridge, UK \\
%    \textsf{nh441@cam.ac.uk},
%    \textsf{https://www.dpmms.cam.ac.uk/$\sim$nh441/} 
%  \end{tabular}}

\end{document}